\documentclass[11pt]{amsart}
\usepackage{graphicx, color}
\usepackage{amssymb}
\usepackage{hyperref} 

\definecolor{darkgreen}{rgb}{0,0.55,0}
\DeclareMathAlphabet{\mathpzc}{OT1}{pzc}{m}{it}

\newcommand{\e}{\varepsilon}

\newcommand{\curl}{\operatorname{curl}}

\newcommand{\J}{\mathbb{J}}
\newcommand{\p}{\partial}
\newcommand{\C}{\mathbb{C}}
\newcommand{\R}{\mathbb{R}}
\newcommand{\Z}{\mathbb{Z}}

\newcommand{\logep}{ \left| \ln \e \right| }

\newcommand{\rest}{\mathbin{\vrule height 1.6ex depth 0pt width
0.13ex\vrule height 0.13ex depth 0pt width 1.3ex}}

\def\barint{\mathop{\vrule width 6pt height 3 pt depth -2.5pt \kern -8.8pt
\intop}}

\newcommand{\chara}{{\mathbf{1}}}
\newcommand{\LV}{\left|}
\newcommand{\RV}{\right|}
\newcommand{\LA}{\left<}
\newcommand{\RA}{\right>}
\newcommand{\LC}{\left(}
\newcommand{\RC}{\right)}

\newcommand{\LB}{\left[}
\newcommand{\RB}{\right]}
\newcommand{\LN}{\left\|}
\newcommand{\RN}{\right\|}

\newcommand{\E}{\mathbb{E}}
\newcommand{\calD}{{\mathcal{D}}}
\newcommand{\calB}{\mathcal{B}}
\newcommand{\calH}{{\mathcal{H}}}
\newcommand{\calP}{{\mathcal{P}}}
\newcommand{\calL}{{\mathcal{L}}}

\newtheorem{theorem}{Theorem}
\newtheorem{proposition}{Proposition}

\newtheorem{lemma}{Lemma}

\theoremstyle{remark}
\newtheorem{remark}{Remark}
\newtheorem{example}{Example}

\begin{document}
\title[Hydrodynamic limit]{Hydrodynamic limit \\ of the Gross-Pitaevskii equation}
\author{Robert L.~Jerrard}
\address{Department of Mathematics, University of Toronto}\email{rjerrard@math.toronto.edu}
\author{Daniel Spirn}
\address{Department of Mathematics, University of Minnesota}\email{spirn@math.umn.edu}

\date{\today}
\thanks{{The  first author was partially supported by the National Science and Engineering Research Council of Canada under operating grant 261955, and the second author was partially supported by NSF grant DMS-0955687. We are grateful to these
agencies for their support. We would also like to thank Jeremy Quastel for explaining
to us the proof of Lemma \ref{L.Jeremy}. }}

\begin{abstract}
 {We study dynamics of vortices in solutions of the Gross-Pitaevskii equation 
$i\p_t u = \Delta u + \e^{-2} u(1 - |u|^2)$  
on $\R^2$ with nonzero degree at infinity.  }We prove that vortices move according to the classical Kirchhoff-Onsager ODE for a small but finite coupling parameter $\e$.  By carefully tracking errors we allow for asymptotically large numbers of vortices, and this lets us connect the Gross-Pitaevskii equation on the plane to two dimensional incompressible Euler equations through the work of Schochet \cite{Schochet2}.
 
\end{abstract}

\maketitle


\section{introduction}

In this paper we prove some theorems that 
relate 
the Gross-Pitaevskii
equation
\begin{equation} \label{nls} 
i \p_t u_\e = \Delta u_\e+ {1\over \e^2} u_\e \LC 1 - |u_\e|^2 \RC,
\qquad
u :(0,T)\times \R^2\to \C
\end{equation}
in the limit $\e\to 0$, for suitable sequences of
initial data, to the incompressible 
Euler equations, which may be written in the form
\begin{equation}
\partial_t \omega + v\cdot \nabla \omega = 0 , \qquad
\label{Euler}\end{equation}
where $v$ is recovered from $\omega$ by convolution
with the Biot-Savart kernel $K$:
\begin{equation}
v(t,x) \ = \  \int_{\R^2} K(x - y)\omega(t,y)\ dy
\qquad\quad
\mbox{ for } \ \ K(x) :=  \frac{x^\perp}{2\pi|x|^2} =   \frac{(-x_2, x_1)}{2\pi|x|^2} .
\qquad
\label{BiotSavart}\end{equation}
We interpret $\omega$ as specifying the distribution of vorticity in an ideal incompressible fluid, and $v$ as the associated velocity field.
In fact the equation makes sense as long as  $\omega(t, \cdot)$ is a measure on $\R^2$ for every $t$ such that $\omega(t,\cdot)\in H^{-1}(\R^2)$, with derivatives of $\omega$ understood in the sense of distributions.

In our main results we construct sequences of 
solutions $u_\e$ of \eqref{nls}
for which
the vorticity associated to $u_\e$
converges as $\e \to 0$, after some rescalings, to 
a solution $\omega$ of \eqref{Euler}, \eqref{BiotSavart}.
We will prove:

\begin{theorem}\label{T1}
Assume that $\omega_0$ is a probability measure on $\R^2$ with finite
second moment, and such that $\| \omega_0 \|_{H^{-1}}<\infty$
and $\int x \omega_0 dx = 0$.

For $\e\in (0,\e_0)$, with $\e_0$ fixed in Theorem~\ref{thmdynamics} below, assume that $n_\e$ is an integer such that 
{ $1 \ll n_\e \ll (\ln |\ln \e|)^{1/2}$,} and that $a^\e \in (\R^2)^{n_\e}$ is a 
collection of points such that
\[
\frac 1{n_\e} \sum_j \delta_{a^\e_j}\rightharpoonup \omega_0
\qquad
\mbox{ weak-*, as $\e\to 0$}.
\]
{Assume in addition that 
\begin{equation}
\frac 1{n_\e}\sum_{j} |a_j^\e|^2 \ \le M_0,\qquad
\quad
\frac {-1}{n_\e(n_\e-1)}\sum_{ j\ne k}  \ln |a_j^\e -a_k^\e| \ \le M_0
\label{Mbounds0}\end{equation}
for some $M_0>0$, independent of $\e$. }
Let $u_\e$ solve \eqref{nls} with initial data
\[
u_\e^0(x) \ = \ \prod_{j} \phi_\e( x - a^\e_j),
\]
for  $\phi_\e:\R^2\to \C$ described in \eqref{modelvortex} below.
Define the {\em current} $j(u_\e)$ associated to $u_\e$ by
\begin{equation}
j(u_\e) := (i u_\e,\nabla u_\e),\qquad
\qquad \mbox{ where }(v,w) := \frac 12 ( v \bar w + w \bar v) \mbox{ for }v,w\in \C.
\label{current.def}\end{equation}
Further define the {\em vorticity} $\omega_\e(u_\e)$ and {\em rescaled vorticity} $\widetilde \omega_\e$ 
by
\begin{equation}
\omega(u_\e) := \nabla\times j(u_\e),
\qquad\qquad \widetilde \omega_\e(t, x) := \frac 1{2\pi n_\e}\omega(u_\e)(\frac t{2\pi n_\e}, x).
\label{vrv.defs}\end{equation}
Then the rescaled vorticities $\{\widetilde \omega_\e \}_{\e\in(0,1]}$ 
are precompact in the space $C((0,T), X_{\ln}^*)$, described in \eqref{CoTXstar} below, 
 and 
any limit point is a weak solution  of the Euler equations \eqref{Euler}, \eqref{BiotSavart}
with initial data $\omega_0$.

\end{theorem}

The theorem implies that for fixed $\e$, the (unrescaled) vorticity $\omega(u_\e)$
is of order $n_\e\gg 1$ in certain weak norms, which roughly speaking
implies that the associated velocities are also of order $n_\e$.
The rescaling $t\mapsto \frac t{2\pi n_\e}$ in the definition of 
$\widetilde \omega_\e$ is thus needed to select a time
scale in which, informally, ``velocities are of order $1$". See also Remark
\ref{rem:trescale} below.

In Theorem \ref{T.random} in
Section \ref{S:hydro}, we also prove that 
 for certain {\em random} choices of initial vortex locations,
we can let the number of vortices diverge  as $\e\to 0$
more rapidly than in Theorem \ref{T1},
as long as we are satisfied with {\em almost sure} convergence, rather
than insisting on convergence for {\em every} sequence of initial 
data.

Conditions \eqref{Mbounds0} are a discrete version of the requirement that
$\omega_0 \in H^{-1}$ with finite second moment.
Any probability measure $\omega_0$ satisfying these
hypotheses can be approximated by a sequence of measures of the form
$\frac 1{n_\e} \sum \delta_{a^{n_\e}_j}$ 
satisfying \eqref{Mbounds0} for
some $M_0$. 
{For example, this applies to
the so-called ``vortex sheet" initial data
studied by Delort \cite{Delort} and others.}

The function $\phi_\e$ appearing in the statement of the theorem
can be taken to be the minimizer of the energy
\begin{equation}
\psi \mapsto \int_{B_{1}}e_\e(\psi) \ dx := \int_{B_{1}} \frac 12 |\nabla \psi|^2 + \frac 1 {4\e^2}(|\psi|^2-1)^2 \ dx
\label{modelvortex}\end{equation}
in the space $\{ \psi \in H^1_{loc}(\R^2; \C)  : \psi(x) = \frac{x_1+i x_2}{|x|} \ \mbox{ if }|x| \ge   \e^{1/2} \}$.

\begin{remark}
Results of a similar flavor as Theorem~\ref{T1} for the parabolic analogue of \eqref{nls} were obtained in \cite{KurzkeSpirn} on bounded domains with  Dirichlet boundary conditions.  In that setting it was shown that solutions with initial data with asymptotically large, albeit dilute, numbers of vortices converge weakly to a mean field model, predicted by \cite{E}.  
\end{remark}

\medskip

\subsection{point vortices and incompressible Euler}

Our argument is based on the fact that  both \eqref{nls} and \eqref{Euler}, \eqref{BiotSavart}
are known to be related to 
a system of 
ODEs, known as the Kirchhoff-Onsager law, that 
describes the evolution of a collection of $n$ point vortices 
on $\R^2$. One form of this system  is displayed in
\eqref{ODE} below.
In particular, we will use the following theorem, which
shows that the Euler equations arise as a hydrodynamic
limit of the point vortex system.

\begin{theorem} [Schochet \cite{Schochet2}]
Let  $b^n(t) = (b^n_1(t),\ldots, b^n_n(t))$ be
a sequence of solutions
of the point vortex ODEs
\begin{equation} \label{ODE}
{d \over dt} b^n_j = 
 \sum_{k \neq j} \frac 1n \frac 1 {2\pi} {( b^n_j - b^n_k )^\perp \over |b^n_j - b^n_k|^2 }
= 
 \sum_{k \neq j}\frac 1n K(b^n_j-b^n_k), 
\qquad \quad j=1,\ldots, n
\end{equation}
with $n\to \infty$, and {assume that there exists some $M_0>0$ such that
at time $t=0$,
\begin{equation}
\frac 1{n}\sum_{j} |b^n_j|^2 \ \le M_0,\qquad
\quad
\frac {-1}{n_\e(n_\e-1)}\sum_{j\ne k}  \ln |b_j^n -b_k^n| \ \le M_0
\label{Mbounds}\end{equation}
for every $n$.}

Then the solution $b^n(t)$ exists
and satisfies \eqref{Mbounds} for all $t>0$, and the sequence
\[
\omega^n(t) := \sum_{i=1}^n \frac 1 n  \delta_{b^n_i(t)}
\]
is precompact with respect to the topology induced by the norm
\[ 
\| \omega \|_{C(0,T;W^{-2,1})} := \sup_{0\le t \le T}\| \omega(t)\|_{W^{-2,1}(\R^2)} 
\]
for every $T>0$.
Moreover, if $\omega$ is any limit of a convergent subsequence,
then for every $t$, $\omega(t)$ is a probability measure in $H^{-1}(\R^2)$
with finite second moment, and  $\omega$  is a weak
solution of the incompressible Euler equations \eqref{Euler}, \eqref{BiotSavart}.
\label{Thm.Schochet}\end{theorem}

Schochet in fact proves compactness in somewhat stronger topologies than
the one described above.

\begin{remark}\label{rem:weakstrong}
if $\omega(t, \cdot)$ is a measure on $\R^2$ for every
$t$, then it is said to satisfy the 
weak vorticity formulation of the Euler equations if
\[
\int \int_{\R^2} \partial_t \zeta(t,x)  \omega(t, dx) \ dx \ dt
+ \int \iint_{\R^2\times\R^2} H_\zeta(t,x,y)\, \omega(t,dx)\, \omega(t,dy)\ dx \ dy \,dt = 0
\]
for all smooth $\zeta$ with compact support in $(0,T)\times \R^2$, where
\[
H_\zeta(t,x,y) = \frac 12 K(x-y) (\nabla \zeta(t,x) - \nabla\zeta(t,y)).
\]
This is the notion of weak solution appearing in Theorem \ref{Thm.Schochet}.
It was introduced by DiPerna and Majda \cite{DM} and later used by Delort \cite{Delort} to prove the existence
of weak solutions if the initial data is a measure in $H^{-1}$. The formulation
we that we use comes from work of
Schochet \cite{Schochet1}, which in some ways simplifies the treatment of
\cite{DM, Delort}.  

A theorem of Brenier, de Lellis and Sz\'ekelyhidi \cite{BdLSz} states that
every weak solution in this sense, that agrees with a smooth solution at time $t=0$, continues
to do so for all $t>0$. As a result, if in  Theorem \ref{Thm.Schochet} the measures
$\omega^n(0)$ converge weak-* to a measure of the form $\omega_0(x) dx$, 
where $\omega_0$ is a smooth function, then in fact the full sequence $(\omega^n(t))$
converges for every $t$, without passing to a subsequence, and the limit is
the unique smooth solution of the Euler equations \eqref{Euler}, \eqref{BiotSavart}
with initial data $\omega_0(x)$. 
\end{remark}

\subsection{{infinite-energy solutions of the Gross-Pitaevskii equation}}

{ Most of our effort will be devoted to
proving estimates that connect the Gross-Pitaevskii equation \eqref{nls} with 
point vortex ODEs, see Theorem \ref{thmdynamics} below. Results of this sort have been known  since the late '90s, 
see \cite{CJ0,CJ, LX}.
The chief novelty of our work is that 
\begin{itemize}
\item  we consider {\em infinite-energy} solutions of \eqref{nls} on $\R^2$, and
\item  we establish quantitative results, with explicit error estimates, that allow for 
a divergent number of vortices as $\e\to 0$.
\end{itemize}
These issues have been handled individually, in \cite{BJS} and \cite{JSp2}
respectively, so what is new here is that
we address both of them at once.  We describe below, particularly in Sections \ref{subs:wknorm} and \ref{subs.org}, some of the obstacles that arise in carrying this out.
}

Theorem \ref{T1}  will follow by combining 
Theorem \ref{thmdynamics}  with Theorem \ref{Thm.Schochet}.

Throughout our discussion of the Gross-Pitaevskii equation, we will regard $\e$ and $n$
as fixed, and we will often omit sub- and superscripts indicating the
dependence of other quantities on $n$ and $\e$. Thus
we will write $n$ rather than $n_\e$,
we will denote a point in $(\R^2)^n$ by $a  = (a_1,\ldots, a_n)$
rather than $a^n$, we will write the solution of \eqref{nls} as
$u$ rather than $u_\e$, and so on.

We will be interested in solutions 
of the Gross-Pitaevskii equation \eqref{nls} with
the property that $u(t,x) \approx e^{iD\theta}$ for large $|x|$, for some  nonzero
integer\footnote{In some of our results, we allow vorticity of mixed sign, and under these conditions
the ``degree at infinity" $D$ need not equal the number $n$ of vortices.}
 $D$.
Such functions do not belong to any very convenient Sobolev space, 
and they also have the property that the natural energy density
\begin{equation}
e_\e(u) :=  \frac 12 |\nabla u|^2 + \frac 1 {4\e^2}(|u|^2-1)^2
\label{eep.def}\end{equation}
is not integrable.
However, the solutions we are interested in all belong to spaces
\[
[U^D] + H^1 := \{ u_1 + u_2 : u_1\in [U^D], u_2\in H^1(\R^2;\C) \}
\]
where 
 for every $D$, $U^D$ is a fixed smooth function such that
\begin{equation}
\label{UD.def}
{
\mbox{ $U^D= e^{iD\theta}$ 
outside some compact set,}
}
\end{equation} 
and 
\[
[U^D] := \{ \tau_y U^D  ,\  y\in \R^2  \}, \qquad\mbox{ where } 
\tau_y U^D(x) := U^D(x-y).
\]
In this setting, we can appeal to results of Bethuel
and Smets \cite{BS}, which show that for $u_0\in [U^D]+H^1$, there exists a unique solution
$u \in C(\R; u_0 + H^1(\R^2))$ of \eqref{nls} such that $u(0)= u_0$.
(In fact Bethuel and Smets consider somewhat more general initial data than we describe here.)
Moreover, for such solutions, a sort of ``energy renormalized at infinity" is conserved.
In order to simplify some formulas, we henceforth always assume\footnote{This can be arranged by taking $U^D = f(r) e^{iD\theta}$, for $f$ such that $f(r) = 1$ for $r\ge 4$, and 
$\int_0^4 [r(f')^2 + D^2 f^2/r] \ dr = D^2 \ln 4$. It is not hard to see that this is possible.}
 that 
\begin{equation}
\frac 12 \int_{B(R)} |\nabla U^D|^2 \ dx = \pi D^2\ln R
\qquad
\mbox{for all $R  \ge 4$}.
\label{Psin0}\end{equation}
Then for $u \in [U^D]+H^1$, we define
\begin{equation}
\E_\e(u) :=
 \lim_{R\to \infty}\int_{B(R)} [e_\e(u) - \frac 12 |\nabla U^D|^2]\  dx
\overset{\eqref{Psin0}}{\ = \ } 
\lim_{R\to \infty}
\left[\int_{B(R)} e_\e(u)\, dx  - \pi D^2 \ln R \right].
\label{EE.def}
\end{equation}
For $u_0\in [U^D]+H^1$, the solutions provided by \cite{BS} satisfy
\begin{equation}
\E_\e(u(t)) = \E_\e(u_0)  \quad\mbox{ for all }t>0.
\label{consEr}\end{equation}
In particular, $\E_\e(u(t))$ is well-defined and finite
for every $t$. (See also Lemma \ref{lem.Edef} below.)

In addition to the current
$j(u) = (iu, \nabla u)$ and energy density $e_\e(u)$ defined above,
another quantity with a natural physical 
interpretations is the {\em density } $|u|^2$.
For any solution $u$,
the following identities hold:
\begin{align}
{1\over 2} {d \over dt} |u|^2 & = \operatorname{div} j(u)  \label{consmass} \\
{1\over2 } {d \over dt} j(u) & =   \operatorname{div}( \nabla u \otimes \nabla u ) + \nabla P  \label{consmomentum} 
\end{align}
where the prescribed pressure, $P:= - {1\over 2} \LV \nabla u \RV - {1\over 2} (u, \Delta u) + {|u|^4 - 1 \over 4 \e^2}$.  
By taking the curl of \eqref{consmomentum}, we obtain an equation for the evolution of the vorticity.
It is conventional to express this in terms of ${J(u)}:= \frac 12 \nabla\times j( u) = \frac 12 \omega(u)$,
which a short computation shows is equal to the Jacobian determinant
$\det \nabla u$, here viewing $u$ as a map $\R^2\to \R^2$ in the natural way. 
The resulting equation (for the ``half-vorticity'') is
\begin{equation}
{d \over dt} J(u)  = \curl  \operatorname{div}( \nabla u \otimes \nabla u ) = \J_{kl} \p_{x_k} \p_{x_m} \LC u_{x_m}, u_{x_l} \RC, \qquad\quad \J := \left(\begin{array}{rr} 0 &1\\-1&0\end{array}\right).
\label{consjac}
\end{equation}
Equations  \eqref{consmass} and \eqref{consjac}, together with conservation of energy, have been used crucially in all studies of vortex motion for \eqref{nls}. 

\subsection{{ a weak norm}\label{subs:wknorm}
}

We will  be interested in describing solutions for which the vorticity is close to
a sum of point masses. To formulate this condition more precisely, we
will introduce a suitable weak norm $X_{\ln}^*$, and we will impose conditions such as
\begin{equation}
\Big\| J(u) - \pi \sum_{i=1}^n  d_i \delta_{a_i} \Big\|_{X_{\ln}^*} \ll 1
\label{sample.eqn}\end{equation}
for some $a = (a_1,\ldots, a_n)\in (\R^2)^n$ and $d = (d_1,\ldots d_n)\in \{\pm 1\}^n$.
We will often, but not always, restrict our attention to the case $d_i=1$ for all $i$.

A good choice of norm is important and not  obvious and, indeed, is one of the main new issues in this paper.
We need the norm to be strong enough that
conditions like \eqref{sample.eqn} imply useful
information, such as lower bounds on $\E_\e(u)$, but weak enough that, for example,
\begin{equation}
\| J(u(t)) - J (u(s)) \|_{X_{\ln}^*}\to 0\mbox{ as } s\to t
\label{XcJ}\end{equation}
for the solutions $u$ of \eqref{nls} that we consider.
{Properties of this sort are not satisfied by the most natural analogs of
norms used in earlier work such as \cite{JSp2}, as we show in Example \ref{ex.discont}
at the end of Section \ref{ss:wknorm}.}

The norm we use {is a sort of weighted $\dot W^{-1,1}$ norm,} defined as follows.
For signed measures $\mu$ (or somewhat less regular distributions), 
we set
\begin{equation}
\| \mu \|_{X_{\ln}^*} := \sup \left\{ \int \phi \ d\mu \ : \ \phi\in C^1_c(\R^2),\ \ \  
\| \phi \|_{X_{\ln}} \le 1\right\}.
\label{Xstar.def}\end{equation}
where for Lipschitz continuous $\phi:\R^2\to \R$ (not necessarily with
compact support),  we write
\[
\| \phi\|_{X_{\ln}} := \operatorname{ess\,sup}_{x\in \R^2} \left( (1+\ln^+|x|) \ |D\phi(x)| \right).
\]
We also define $X_{\ln}^*$ to be
the space of distributions on $\R^2$ with finite $X_{\ln}^*$
norm.
The space $C((0,T); X_{\ln}^*)$  
is the (Banach) space of continuous maps from the interval
$(0,T)$ into $X_{\ln}^*$, endowed with the norm
\begin{equation}
\| \mu \|_{C((0,T); X_{\ln}^*)} = \sup_{0<t<T} \|\mu(t)\|_{X_{\ln}^*}.
\label{CoTXstar}\end{equation}

Some basic properties of the $X_{{\ln}}^*$ norm
are established in Section \ref{ss:wknorm}
below. 
These imply in particular that \eqref{XcJ} holds
for the Bethuel-Smets \cite{BS} solutions of \eqref{nls}.

\subsection{The Kirchhoff-Onsager energy}
As mentioned above, it will turn out that conditions such as \eqref{sample.eqn}
imply very precise lower bounds for $\E_\e$, which will depend on the vortex locations
$a := (a_1,\ldots, a_n)$ and degrees $d = (d_1,\ldots, d_n)$. These lower bounds will be expressed in terms
of the function
\begin{equation}
W_\e(a,d) := n(\pi |\ln \e| + \gamma) \ + \ W(a,d),
\qquad\qquad W(a,d) := - \pi \sum_{j\ne k} d_i d_j \ln |a_i - a_j|  .
\label{Wep.def}\end{equation}
Here $\gamma$ is a specific number,
first identified in \cite{BBH}, whose definition is recalled in
\eqref{gamma.def} below.
If $d_i=1$ for all $i$, we will write simply $W_\e(a)$ and $W(a)$.

Note that $W$ is the classical Kirchhoff-Onsager energy for interacting vortices.

Given $a = (a_1,\ldots, a_n) \in (\R^2)^n$, we will always write
\begin{equation}
\rho_{a} := \frac 14 (1 \wedge \min_{i\ne j} |a_i-a_j|), \qquad R_{a} := 4(1\vee  \max_j|a_j|) {.}
\label{rhoR.def}\end{equation}

\subsection{point vortices and the Gross-Pitaevskii equation}

We can now state the extension of results of \cite{JSp2}, \cite{BJS} that is needed for the proof of
our main theorem. In the statement of the theorem we regard $n$ and $\e$
as fixed.

\begin{theorem}  \label{thmdynamics}
Let $a(t) : (0,T)\to (\R^2)^n$
be a solution of the point vortex equations
\begin{equation}
\frac d{dt} a_j = 2\pi \sum_{j\ne k}  K(a_j - a_k) \qquad j=1,\ldots, n
\label{PV2}\end{equation}
{with initial data $a(0)$ satisfying
\begin{equation}
\frac 1{n}\sum_{j} |a_j|^2 \ \le M_0,\qquad
\quad
\frac {-1}{n(n-1)}\sum_{ j\ne k}  \ln |a_j -a_k| \ \le M_0
\label{Mbounds2}\end{equation}
}
for some $M_0>0$.

Let $u$ solve the Gross-Pitaevskii equation \eqref{nls} with initial data $u(0,x) = u_0(x)$ such that
\begin{align}
\| J (u_0) -  \pi \sum_{j=1}^n   \delta_{a_j(0)} \|_{X^*_{\ln}} 
&\leq  \e^{\frac 1 2}
\label{t1.h1} \\
\Sigma(u_0; a(0))
:= \E_\e (u_0) - W_\e(a(0))  &\leq \e^{1\over 2} .
\label{t1.h2}\end{align}

Then there exists some $C, \e_0>0$, depending only on $M_0$ above, such that if
\[
0 < \e <\e_0\qquad \mbox{ and } \ \ n \le \e_0\logep^{1/2}
\]
then 
\begin{equation}
\Big\|  J(u(t)) - \sum_{j=1}^n \pi  \delta_{a_j(t)} \Big\|_{X^*_{\ln}} 
\leq \e^{1\over 3}
\label{t1.c1}\end{equation}
for all for all $0\le t \le \tau_\star = \tau_\star(a(0), \e)$, where
\begin{equation}
\tau_\star := \sup \left\{ T > 0 \ : C {n}  \int_0^{T} \rho_{a(t)}^{-2} \ dt \ \le  |\ln \e|  \right\}.  
\label{Tstar.def}\end{equation}
\end{theorem}

\begin{remark}
Note that $a(t) = (a_1(t),\ldots, a_n(t))$ solves
the point vortex ODE \eqref{PV2} as scaled in Theorem 
\ref{thmdynamics}
if and only if $b(t) = a(\frac t{2\pi n})$ solves
the point vortex system
\eqref{ODE} in the scaling of Theorem
\ref{Thm.Schochet}. This can be seen as the reason for the
$t\mapsto \frac t{2\pi n_\e}$ rescaling in the definition \eqref{vrv.defs}
of the rescaled vorticity $\widetilde \omega_\e$.
\label{rem:trescale}\end{remark}

\begin{remark}\label{whypositive}
{The assumption that all vortices have the same sign
is used only to deduce  bounds, global in $t$,
on quantities such as $\rho_{a(t)}$ and $R_{a(t)}$
associated to a solution $a(t)$
of \eqref{PV2} satisfying \eqref{Mbounds2} at time $t=0$.
Thus our results apply also to collections of vortices of mixed
sign, as long as they satisfy the relevant bounds.
However, we know of no way to verify bounds, except in
very special cases, without
the assumption that all vortices are positive.}
\end{remark}

\begin{remark}\label{rem.tau}
{In \cite{JSp2}, we proved that estimates similar to \eqref{t1.c1} hold
(but on a bounded domain and with finite energy) for $0\le t \le \tau_\star^{old}$,
where
\begin{align*}
\tau_\star^{old} 
&:= \sup \left\{ T > 0 \ : C {n}  T \max_{0\le t \le T} \rho_{a(t)}^{-2} \ \le  |\ln \e|  \right\}\\
&
\hspace{9em}
\le
\sup \Big\{ T > 0 \ : C {n}  T\barint_{[0,T]} \rho_{a(t)}^{-2} \ \le  |\ln \e|  \Big\}
= \tau_\star.
\end{align*}
The $max$ appearing in the formula for $\tau_\star^{old}$ makes it  hard to estimate and  sensitive to bad behavior on small sets.
The fact that we have replaced a $max$  with an average
is vital for our discussion of randomly
chosen initial vortices in Section \ref{S:hydro}.
}\end{remark}

\begin{remark}\label{rem:more}
The proof of Theorem \ref{thmdynamics} yields a number of other estimates that we have not recorded here.
In particular, it shows that estimates \eqref{xis.def} and \eqref{Sigma.est1} - \eqref{gstab.ref3},
with $\eta(t)\le \e^{1/3}$,
hold for  $0 \le t \le \tau_\star$ where $\eta(t)$ is a nice proxy for 
{the aggregate distance between the actual vortex locations and
the positions predicted by the point vortex ODE, see \eqref{eta.def} for a definition}. Among other conclusions, these imply that for $0\le t \le \tau_\star$, the velocity
field $j(u)(t)$ is very close in certain norms to the velocity field associated to a collection
of vortices located at points $\xi_1(t),\ldots, \xi_n(t)$ such that $\sum |\xi_i(t) - a_i(t)| \le
C \e^{1/3}$.
\end{remark}

\subsection{organization of this paper}\label{subs.org}
{As mentioned above, the bulk of this paper is devoted to the proof
of Theorem  \ref{thmdynamics}, the main ingredients of which 
are the following.
\begin{itemize}
\item 
We develop some criteria that allow us to conclude that if a function $u$ satisfies  certain hypotheses, then the associated vorticity $J(u)$
 is  close to a sum of point masses 
$\pi \sum_{i=1}^n \delta_{\xi_i}$. This is done in Section \ref{sec.loc}.
\item
If $J(u) \approx \pi \sum_{i=1}^n \delta_{\alpha_i}$, then
a natural notion of ``surplus energy" is
the difference between the energy of $u$ and  (a good lower bound for) the
minimal energy
needed for a wave function with vortices at the points $\alpha_1\ldots, \alpha_n$,
see \eqref{t1.h2} above.
We show in Section \ref{s:gstab} that our surplus energy controls 
various quantities relating for example to the difference between
the current $j(u(t))$ and the model  current $2\pi \sum_{i=1}^n K(x-\xi_i)$ generated by ideal point vortices at $\xi_i(t),i=1,\ldots, n$.
\end{itemize}
The above parts of our argument are purely variational and do not rely in any way on the evolution
equation \eqref{nls}.
\begin{itemize}
\item
Considering now solutions of \eqref{nls},
we introduce in Section \ref{sec.gronwall} 
a scalar quantity $\eta(t)$ such that $\eta(t) \approx\sum_{i=1}^n |\xi_i(t)-a_i(t)|$
(the sum of the distances between the actual vortex locations
and the locations given by the point vortex ODE \eqref{PV2}).
This quantity $\eta(t)$ has the feature that it controls the surplus energy,
and the proof of Theorem \ref{thmdynamics} reduces to controlling the
growth of $\eta$.
\item This is carried out in Sections \ref{S.etadot}-\ref{S.conc} via Gr\"onwall's inequality, making strong use
of the variational estimates mentioned above, as well as conservation
laws for the Gross-Pitaevskii equation, discussed above.
\end{itemize}
In fact these ingredients are all present, in some form, in all rigorous work on vortex dynamics in
\eqref{nls}, dating back to \cite{CJ0}. The distinctive
features of our analysis arise from
the quantitative nature of our estimates, together with
the fact that we need notions of ``surplus energy"
and (especially)  ``closeness of $J(u)$ to  $ \pi \sum_{i=1}^n \delta_{\xi_i}$''
adapted to the function space $[U^D]+H^1$ in which we work.
}

{Our surplus energy 
is built out of the energy $\E_\e(u)$ introduced in \cite{BS} and already used in \cite{BJS} for the study
of vortex dynamics in infinite-energy solutions on $\R^2$. A significant
difference between our analysis and that of \cite{BJS}
is that the latter relies on variational estimates (along the lines described above)
on balls $B_R$ for large $R$, with errors that depend on $\e$ and $R$
in a way that is not always explicit.
Results on vortex dynamics are
proved  by first letting $\e\to 0$
and then $R\to \infty$, using conservation laws
and arguments based on
Almeida's notion of topological sectors, see
\cite{A}, 
to control the  energy at infinity. }

{This procedure bypasses a large number of difficulties, 
at the expense of making it impossible to formulate refined estimates
for solutions of \eqref{nls} for small but positive $\e$, of the sort developed in \cite{JSp2}
for finite-energy solutions on bounded domains and pursued
here on $\R^2$. Such estimates are necessary for applications
such as the hydrodynamic limit in Theorem \ref{T1}. We therefore
proceed by adapting to our setting the more quantitative approach 
of \cite{JSp2}. Throughout our arguments, a crucial role is played by the 
$\|\cdot \|_{X_{\ln}^*}$ norm, introduced in Section \ref{subs:wknorm},
which we use to estimate the extent to which $J(u) \approx \pi \sum \delta_{a_i}$.
In particular, we prove in 
Section \ref{sec.annuli} that
information about $J(u)$ in the $\|\cdot \|_{X_{\ln}^*}$ norm implies
lower energy bounds outside large balls. This replaces the arguments from
\cite{BJS} based on topological sectors. Other aspects of the $\| \cdot \|_{X_{\ln}^*}$
norm that we need include an interpolation inequality 
(Lemma \ref{interp.X.W}), elliptic estimates (Lemma \ref{poissonX}),
integration by parts involving certain functions that diverge at infinity (proof of 
Proposition \ref{Propgammastability}), 
the continuity of the Jacobian as a map into $X_{\ln}^*$  (Lemma \ref{continuityJacobian}), 
and various analogs of 
rather standard estimates on bounded domains (for example
Lemma  \ref{L.JJprime}). 
}

{The overall structure of the paper follows that of \cite{JSp2}, and so we refer, when possible, to prior arguments.  On the other hand where significant differences arise, such as those described above, we have included full details.  We have also tried to streamline or improve proofs, compared to those of \cite{JSp2}. 
One such  improvement is described in Remark
\ref{rem.tau} above.}

{Our analysis begins in Section \ref{sec.prelim}, which collects some
notation and establishes some basic and rather straightforward properties of
the $\| \cdot \|_{X_{\ln}^*}$ norm, the energy $\E_\e$,
renormalized energy $W(\alpha)$, and canonical harmonic map $u_*$,
and concludes in Section \ref{S:hydro} with the proof of Theorem \ref{T1},
together with a related result, see Theorem \ref{T.random}, concerning 
sequences of solutions with randomly chosen initial vortex locations.
}

\section{preliminaries} \label{sec.prelim}

\subsection{ some notation, and remarks about parameters}\label{ss.noworries}

{We will write $\chi_R$, for $R\ge 1$, to denote a family of functions satisfying
\begin{equation}  \label{chiRdef1}
\begin{split}
\chi_R =  \left\{ 
\begin{array}{lll}    1 & \hbox{ in } B_R \\
0  & \hbox{ in } \R^2 \backslash B_{2R} ,  \end{array} \right.
\end{split}
\end{equation}
and
\begin{equation}
\mbox{ $0 \leq \chi_R \leq 1$ in $B_{2R} \backslash B_R$, \qquad \quad
 $| \nabla \chi_R| \lesssim R^{-1}$. }
\label{chiRdef2}\end{equation}
}

{We write 
\[
A \lesssim B
\]
if there exists some universal constant $C$ such that $A\le CB$.
}
We will always follow the notational convention that
\begin{equation}\label{P.def}
\calP \mbox{ denotes a polynomial function of } \  n, \, R_a,  \, \Sigma(u;a), \, \E_\e(u),  \rho_a^{-1},
\end{equation}
where $\Sigma(u;a,d) = \E_\e(u) - W(a,d)$. 
The coefficients in $\calP$ are always understood to be independent of $\e$ and of all other parameters. 
It will always be
clear which function $u$ and points $a\in \R^{2n}$ we have
in mind. We will 
allow $\calP$ to change from one line to the next.

In Theorem \ref{thmdynamics} we assume that  $0<\e<\e_0$,  $d_i=1$ for all $i$, and
\begin{equation}
n \le \e_0 |\ln \e|^{1/2},
\qquad\quad
\frac 1{n}\sum_{j} |a_j|^2 \ \le M_0,\qquad
\quad
\frac {-1}{n(n-1)}\sum_{ j\ne k}  \ln |a_j -a_k| \ \le M_0.
\label{prevailing}\end{equation}
{These assumptions imply that 
\[
 \sum_{i\ne j} |a_i - a_j|^2 - \ln|a_i-a_j| \le M_1n^2 := 5M_0n^2.
\]
Since $|x|^2 - \ln|x| > 0$ for all $x$,  it follows that every term in
the sum is bounded by $M_1n^2$, and hence that
$- \ln |a_i - a_j| \le M_1 n^2$ for all $i\ne j$.
Thus \begin{equation}
\e^{M_1 \e_0^2} \le \exp(-M_1n^2) \le 4\rho_{a}\le \min_{i\ne j}|a_i-a_j|,
\qquad \quad
R_{a} \le 4\sqrt{M_0} \e_0 |\ln \e|^{1/2}.
\label{rhoRbds}\end{equation}
In particular, by choosing $\e_0$ small (depending on $M_1 = 5M_0$),} we can arrange that $\rho_{a}$ is larger than any
fixed positive power of $\e$.
We also deduce from \eqref{prevailing}  that
\begin{equation}
-n^2 \ln( \frac 12 R_{a})  \ \le \ 
{\frac 1 \pi W(a) =  - \sum_{j\ne k} \ln |a_i - a_k| \  \le M_0 n^2}
\label{Wbounds}\end{equation}
so it follows from \eqref{rhoRbds} and the fact that $n \le \e_0 |\ln \e|^{1/2}$ that
if $\e_0$ is small enough, then
\begin{equation}
c n| \ln \e| \ \le \ W_{\e}(a)   \  =  W(a) + n(\pi\logep +\gamma) 
\ \le  \ C n |\ln \e| \ \le \  C |\ln \e|^{3/2}.
\label{Wep.bounds}\end{equation}
It will also be the case in many of our arguments that 
$|\Sigma(u, a)| \le 1$, where $\Sigma(u,a)  =\E_\e(u) - W_\e(a)$.
Indeed, we often assume that $\Sigma \le 1$, and we will
prove lower bounds which imply that $\Sigma\ge -1$
under hypotheses that will prevail throughout 
most of our proofs.
It follows that under these conditions, 
\begin{equation}
c n| \ln \e| \ \le \ \E_\e(u)  
\ \le  \ C n |\ln \e| \ \le \  C |\ln \e|^{3/2}.
\label{bbEbounds}\end{equation}
From these facts, we see that 
for any given term of the form $\calP$
described in \eqref{P.def},
\begin{equation}
\forall \beta>0,\quad \exists \e_0 >0 \mbox{ such that } \ \calP \le \e^{-\beta} \quad \mbox{ for }0<\e<\e_0
\label{Psmall}\end{equation}
when \eqref{prevailing} holds {and $|\Sigma|\le 1$}, where $\e_0$ depends only
on $M_0$, $\beta$ and on the coefficients in $\calP$. 


\subsection{a weak norm}\label{ss:wknorm}

{In the introduction we have defined the ${X_{\ln}^*}$ norm, see \eqref{Xstar.def}, which
we often use to measure errors in vorticity. 
The direct analog of the norm used for this purpose
in \cite{JSp2}
is the $\dot{W}^{-1,1}(\R^2)$ norm, which we define
by
\[
\LN \mu \RN_{\dot{W}^{-1,1}(\Omega)} :=  \sup \left\{ \int \phi \ d\mu \ : \ \phi\in C^1_c(\Omega),\ \ \  \| \nabla \phi \|_{L^\infty(\Omega)} \le 1 \right\}.
\]
This norm is however too strong for our needs
in this paper.  
In particular, it turns out that the map $u \mapsto J(u)$ is
not continuous as a function from $[U^D] + H^1$ 
into $\dot{W}^{-1,1}(\R^2)$ --- see Example \ref{ex.discont} at the end of this section.
The following lemma shows that the $X_{\ln}^*$ norm, which is a sort of {\em weighted} $\dot{W}^{-1,1}$ norm,
does not suffer from this drawback.}

\begin{lemma}  \label{continuityJacobian}
Suppose that $U\in H^1_{loc}(\R^2)$ and that $|DU(x)| \le  C(1+|x|)^{-1}$.

If $u \in U +H^1(\R^2)$ and $w\in H^1(\R^2)$ then
\begin{equation}
\| J(u+w) - J(u)\|_{X^*_{\ln}} \le 
C(u) ( \| w \|_{H^1}  + \| w \|_{H^1}^2 ).
\label{CJ1}\end{equation}
\end{lemma}

\begin{proof}
{For this proof only, we} will write $\omega(x) := 1+ \ln^+ |x|$. 
Observe that our hypotheses imply that
$DU/\omega \in L^2(\R^2)$, and it immediately follows
that $Du/\omega \in L^2(\R^2)$.

Now fix $\zeta\in C^1_c(\R^2)$ such that  $\|\zeta\|_{X_{\ln}}\le 1$, or equivalently $|D\zeta(x)| \le 1/\omega(x)$ for all $x$.
It is convenient to think  of $u$ and $w$ as $\R^2$-valued functions and to write
for example $J(u) = u_{x_1}\times u_{x_2}$.
Temporarily assuming that $u$ is $C^2$,
we expand the determinant and  integrate by parts to find that
\begin{align*}
\int \zeta [ J(u+w) - J(u)] \ dx
&=
\int \zeta \left( u_{x_1}\times w_{x_2} + w_{x_1}\times u_{x_2} + w_{x_1}\times w_{x_2}\right) dx \\
&=
 \int \zeta_{x_2}\,  w\times u_{x_1}  - \   \zeta_{x_1}\, w\times u_{x_2}  \ +  \ 
 \frac 12 (\zeta_{x_2}\, w\times w_{x_1} - \zeta_{x_1} w\times w_{x_2}) \ dx \\
& \le
\int
|w| \frac {|Du|}\omega +  \ |w| \ |\nabla w| \ dx \\
&  \le \| \frac {Du} \omega\|_{L^2} \| w \|_{L^2}   + \| w\|_{H^1}^2 .
\end{align*}
{The same inequality  holds for $u\in U +H^1(\R^2)$,}
by density, and  \eqref{CJ1}
follows with
$C(u) = \max (  \| Du/\omega\|_{L^2}  , 1)$.
\end{proof}


We will sometimes need to consider expressions of the form
$\int \phi \ d\mu$, where $\| \mu\|_{X_{\ln}^*}<\infty$ and 
$\phi$ is a function with, say, logarithmic growth at infinity.
The next lemma assures that such expressions make sense.

\begin{lemma}
Assume that $\phi\in L^\infty_{loc}(\R^2)$ is Lipschitz continuous,
and that there exists some $p>0$ such that $|D\phi(x)| \le C\min(1,  |x|^{-p})$
for a.e. $x$.

Let $(\chi_R)_{R\ge 1}$ satisfy \eqref{chiRdef1}, \eqref{chiRdef2}.
Then the definition
\[
\int \phi \ d\mu  := \ \lim_{R\to \infty} \int \chi_R \ \phi \ d\mu
\]
makes sense in that the limit on the right-hand side exists and is independent of the specific choice
$(\chi_R)_{R\ge 1}$. In addition,
\begin{equation}
\int \phi \ d\mu  \ \le \| \phi\|_{X_{\ln}} \ \|\mu\|_{X_{\ln}^*}. 
\label{dual.est}\end{equation}
\label{lem.dual}\end{lemma}

\begin{remark}
It follows from the Lemma that if $\| \mu\|_{X_{\ln}^*}<\infty$, then
$\int 1 \ d\mu = 0$.
\end{remark}

\begin{proof}
First note that for any $x\in \R^2$,
\[
|\phi(x)| \le |\phi(0)| + \int_0^{|x|} |D\phi|(s \frac x{|x|}) ds \  \le |\phi(0)| +  C\int_0^{|x|}\min(1, s^{-p})ds .
\]
It follows from this and the assumptions on $\chi_R$ that 
\[
\sup_{x\in \R^2} (1+|\ln^+|x|) |\phi|\, |D\chi_R| \to 0 \qquad\mbox{ as }R\to \infty.
\]
As a result, 
\begin{equation}
\| \chi_R \phi\|_{X_{\ln}} \to
\| \phi\|_{X_{\ln}}  \qquad  \mbox{ as }R\to \infty,
\label{normlim}\end{equation}
and so
\[
\limsup_{R\to \infty}\int (\chi_R\phi) d\mu \le \| \mu\|_{X_{\ln}^*} \ 
\| \phi\|_{X_{\ln}} .
\]
Next, {for $R_1\ne R_2$, similar}
computations to those above show that
\begin{align*}
\| (\chi_{R_1} - \chi_{R_2}) \phi \|_{X_{\ln}} 
\to 0\quad \ \ 
\mbox{ as }\min (R_1, R_2 )\to \infty.
\end{align*}
It follows that
\[
\int (\chi_{R_1} - \chi_{R_2})\phi \ d\mu \    \le 
\| (\chi_{R_1} - \chi_{R_2}) \phi \|_{X_{\ln}}
  \| \mu \|_{X_{\ln}^*}   \to 0\quad \ \  \mbox{ as }\min ( R_1, R_2)\to \infty.
\]
These estimates imply the conclusions of the lemma.
\end{proof}

The next lemma relates that $X_{\ln}^*$ norm and the $\dot{W}^{-1,1}$ norm.

\begin{lemma}  \label{lemmaW1controlfromX}
 If $\|\mu \|_{X_{\ln}^*}<\infty$
and $\Omega$ is any bounded set, then
\begin{equation}
\left\| \mu \right\|_{\dot{W}^{-1,1} (\Omega)} \leq \sup_{x\in \Omega} (1+ \ln^+ |x|)\  \| \mu \|_{X_{\ln}^*}.
\label{convertnorm}\end{equation}
\end{lemma}

\begin{proof}
If $\phi\in W^{1,\infty}_0(\Omega)$ then (extending $\phi$
by zero) $\| \phi \|_{X_{\ln}}\le \sup_\Omega(1+\ln^+|x|) \|D\phi\|_{L^\infty(\Omega)} $, so
\[
\int \phi\, d\mu \le \| \phi \|_{X_{\ln}} \| \mu\|_{X_{\ln}^*} \le \| D\phi \|_{L^\infty(\Omega)} \ 
 \sup_\Omega(1+\ln^+|x|)\| \mu\|_{X_{\ln}^*}.
\]
Since $\phi\in W^{1,\infty}_0(\Omega)$ is arbitrary,  the conclusion follows.\end{proof}

\begin{example}\label{ex.discont}
Here we give an example, {not needed in the rest of this paper,}
to prove that $\phi\mapsto J(U^1+\phi)$
is not continuous as a map from $H^1(\R^2)$ into  $\dot{W}^{-1,1}(\R^2)$, 
so that it is indeed necessary
to introduce weaker norms, such as the $X_{\ln}^*$ norm.

First, for $R>1$ let $f_R:\R^2\to \R$ be a function such that
$|\nabla f_R| \le 2$,  $0 \le f_R\le 1$, and 
\begin{align*}
&f_R(x_1,x_2) = 0\mbox{ unless }4R\le x_1 \le 6R \mbox{ and } |x_2|\le R,\\
&f_R(x_1,x_2) = 1 \mbox{ if }4R+1 \le x_1 \le 6R-1\mbox{ and }|x_2|\le R-1.
\end{align*}
Now let $\phi_R := (f_R, 0)$ and $\phi := \sum_{j=1}^\infty \frac 1{j R_j} \phi_{R_j}$ for $R_j = 10^j$.
It is easy to see that
\[
\| \phi_R\|_{L^2}^2 \approx R^2, \qquad
\| \nabla \phi_R\|_{L^2}^2 \approx R,
\]
and since $\phi_{R_j}$ and $\phi_{R_k}$ have disjoint support if $j\ne k$, it follows that
\[
\| \phi \|_{L^2}^2 = \sum_j  (j R_j)^{-2}\| \phi_{R_j}\|_{L^2}^2  \approx \sum_j j^{-2}<\infty.
\]
Similarly $\| \nabla \phi \|_{L^2}^2 < \infty$ and thus $\phi\in H^1$.

We claim however that 
\begin{equation}
\sup \left\{ \int \zeta J(U^1+\phi) \ dx: \zeta \in W^{1,\infty}\mbox{ with compact support}, \mbox{Lip}(\zeta)\le 1
\right\} = +\infty
\label{ctrex}\end{equation}
for $U^1\in H^1_{loc}(\R^2)$ such that $U^1 = \frac x{|x|}$ outside $B_1$. 
To prove this, let $\zeta_R$ be a compactly supported function such that
$\| \nabla \zeta_R \|_\infty \le 1$, 
$\zeta_R(x_1,x_2) = 0$ unless  $2R \le x_1 \le 6R$,
and 
\[
\zeta_R(x_1,x_2) = 2R - (|4R  - x_1|) \mbox{ if } 2R \le x_1\le 6R\mbox{ and }|x_2|\le R.  
\]
Then for any $R\ge 1$, noting that $\zeta_R = 0$ on the support of $J(U^1)$
and arguing  as in the proof of Lemma \ref{continuityJacobian},
\begin{align*}
\int \zeta_R J(U^1+ \phi_R) \ dx
&=
\int \zeta_R  \left[J(U^1+ \phi_R) - J(U^1)\right] dx \\
&=
\int  \zeta_{R,x_2} \ \phi_{R} \times U^1_{x_1}  - \zeta_{R,x_1} \ \phi_R \times  U^1_{x_2}
+ \zeta_R J(\phi_R) \ dx.
\end{align*}
However, $J(\phi_R)\equiv 0$, and $\zeta_{R,x_2} = 0$ on the support of $\phi_R$, so the
first and third  terms vanish. And by an easy explicit computation of the 
remaining term, it follows that
\[
\int \zeta_R J(U^1+ \phi_R) \ dx
\ = \int f_R \frac{ x_1^2}{R^3} \ dx \  \gtrsim  \ R.
\]
Now define $\zeta^m := \sum_{j=1}^m \zeta_{R_j}$. 
Since $J(U^1+\phi) = 0$ away from the support of $\phi$, and noting that
$\mbox{supp}(\zeta_{R_j}) \cap \mbox{supp}(\phi_{R_k})$ is nonempty if any only if 
$j=k$, we find that
\[
\int \zeta^m J(U^1 + \phi) \ dx = \sum_{j=1}^m  \int \zeta_{R_j} J(U^1+ \frac 1{ j R_j} \phi_{R_j}) \ dx
\ \gtrsim \  \sum_{j=1}^m \frac 1j.
\]
Thus $\| J(U^1+u)\|_{\dot{W}^{-1,1}} = +\infty$. 
\end{example}



\subsection{more about energy, and the canonical harmonic map.}
\label{ss2.1}

Given $\alpha\in (\R^2)^n$ and $d\in \Z^n$,  we will write
$u_*(\alpha, d)$ to denote the canonical harmonic map with 
singularities of degree $d_i$ at $\alpha_i$, for $i=1,\ldots, n$,
see \cite{BBH},
defined by
\begin{equation} \label{canonicalharmonicmap}
u_*(\alpha,d)(x)  =  \prod_{j=1}^n \LC {x -  \alpha_j \over |x - \alpha_j|} \RC^{d_j} .
\end{equation}
We will write simply $u_*$ when there is no possibility of confusion. 
One easily checks that
\begin{equation}
j(u_*)(x) = \sum_{i=1}^n d_i \frac{ (x-\alpha_i)^\perp}{|x-\alpha_i|^2}
\ = \ 2\pi \sum_{i=1}^n d_i  K(x-a_i) 
\label{jstar.form}\end{equation}
where $K$ is the Biot-Savart kernel. We remark that
\begin{equation}
\nabla\times j(u_*) = 2\pi \sum_{i=1}^n d_i \delta_{\alpha_i},
\qquad\qquad
\nabla\cdot j(u_*) = 0.
\label{jstar.eqns}\end{equation}

We record some properties of $j(u_*)$ that will be needed in the sequel. These are
mostly well-known and generally easy to verify.   {Given
$\alpha\in (\R^2)^n$, we will write
\begin{equation}
\R^2_\sigma(\alpha) :=  \R^2 \setminus \cup_{j=1}^n B_\sigma(\alpha_j),
\qquad
\qquad
B_{R,\sigma}(\alpha) := B_R \setminus \cup_{j=1}^n B_\sigma(\alpha_j).
\label{cutout}\end{equation}
We will write simply $\R^2_\sigma$ or $B_{R,\sigma}$ when $\alpha$ is clear from the context.}

\begin{lemma}Assume that $\alpha\in (\R^2)^n$ and $d\in \{\pm 1\}^n$,
and let $D = \sum d_i$. 
Then $u_*(\alpha,d)$ satisfies the following.
First,
\begin{equation} \label{Ggradientestimateerror}
\LV   j(u_*)(x)  - { D x^\perp \over |x|^2} \RV \leq C { n R_\alpha \over |x|^2 }\qquad
\mbox{ if }|x|\ge R_\alpha .
\end{equation}

Second, if $\sigma<\rho_\alpha$ then
\begin{equation}
{\left\|  j(u_*) \right\|_{L^\infty(\R^2_\sigma)}} \le  {C n \over \sigma} 
\label{Ggradientestimateerror1}\end{equation}
and
\begin{equation} \label{FiniteWapprox}
\lim_{R\to\infty}\left( \int_{B_{R,\sigma}} \frac 12 |\nabla u_* |^2  dx - \pi D^2 \ln R\right)
=
n \pi \ln \frac 1\sigma + W(\alpha, d) + O( \frac{n^3 \sigma^2}{\rho_\alpha^2}),
\end{equation}
where $W(\alpha, d)$ is defined in \eqref{Wep.def}.
\label{lem.ustar}\end{lemma}

\begin{proof}
Estimates \eqref{Ggradientestimateerror} and \eqref{Ggradientestimateerror1}
are easily verified from the explicit formula \eqref{jstar.form} for $j(u_*)$,
and \eqref{FiniteWapprox} is proved following computations in
\cite{BBH}, and keeping track of error terms as in Lemma 12  of \cite{JSp2}.
\end{proof}

We will also need some estimates of $W$, which follow directly 
from the definition \eqref{Wep.def}:
\begin{equation} \label{Wderivatives}
\LV \nabla_i W(\alpha, d) \RV \lesssim {n \over \rho_\alpha}, \quad \LV \nabla_i \nabla_jW(\alpha, d) \RV \lesssim {n \over \rho^2_\alpha}.
\end{equation}

We collect some facts about the dependence of 
$j(u_*(\alpha, d))$ on $\alpha\in \R^{2n}$.

\begin{lemma}
Let $\alpha, \alpha' \in (\R^2)^n$ with $\min_{j=1,\ldots,n} | \alpha_j - \alpha'_j| \leq {\rho_\alpha \over 4}$ and $d \in \{\pm 1\}^n$.  If $r,r' \leq \rho_\alpha$, then
\begin{align}  \label{jstarLipschitz}
\LN j(u_*(\alpha, d)) - j(u_*(\alpha', d)) \RN_{L^\infty(\R^2_r(\alpha) \cap \R^2_{r'}( \alpha'))} 
 & \leq {1 \over \min\{r, r'\}^2} \sum_{j=1}^n \LV \alpha_j -\alpha'_j \RV
 \\
\label{jstarL2bound}
 \LN j(u_*(\alpha,d)) - j(u_*(\alpha',d)) \RN_{L^2(\R^2_r(\alpha) \cap \R^2_{r}(\alpha'))}
&\le 
C\left( 1 +  \sum \log( \frac{|\alpha_j - \alpha'_j|}r )\right),
\end{align}
and for all $1< p < 2$, there exists a constant $C$, depending only on $p$, such that
\begin{align}  \label{jstarLpbound}
\LN j(u_*(\alpha, d)) - j(u_*(\alpha', d)) \RN_{L^p(\R^2)} 
 \leq C n  \LC \sum_{j=1}^n \LV \alpha_j -\alpha'_j \RV \RC^{{2 \over p} - 1}.
 \end{align}
\end{lemma}

\begin{proof}
Bound \eqref{jstarLipschitz} follows from  
\begin{equation} \label{jalphajalphaprimediff}
\begin{split}
\LV j(u_*(\alpha, d)(x)) - j(u_*(\alpha',d)(x)) \RV 
&  \leq \sum_{j=1}^n |d_j| \LV {(x - \alpha_j)^\perp \over |x - \alpha_j|^2} -  {(x - \alpha'_j)^\perp \over |x - \alpha'_j|^2} \RV  \\
&  = \sum_{j=1}^n  { | \alpha_j - \alpha'_j | \over | x - \alpha_j| |x - \alpha'_j| }.
\end{split}
\end{equation}
To prove \eqref{jstarLpbound} let $\Delta_j = |\alpha_j - \alpha'_j|$ and $\overline{\alpha}_j = {\alpha_j + \alpha_j' \over2}$, and let
$B_j := B_{2\Delta_j}(\bar \alpha_j)$. Then
\begin{align*}
\LN j(u_*(\alpha,d)) - j(u_*(\alpha',d)) \RN_{L^p(\R^2)}
& \leq 
 \sum_{j=1}^n \| K(\cdot -\alpha_j) - K(\cdot - \alpha'_j)\|_{L^p(\R^2)}
 \\
&
\le
\sum_{j=1}^n \| K(\cdot -\alpha_j)\|_{L^p(B_j)} \ + \  \|K(\cdot - \alpha'_j)\|_{L^p(B_j)} \\
&\quad +
\sum_{j=1}^n
 \| K(\cdot -\alpha_j) - K(\cdot - \alpha_j')\|_{L^p(\R^2 \setminus B_j)}.
\end{align*}
Since $B_j \subset B_{3\Delta_j}(\alpha_j)$, a direct computation
shows that
\[
 \| K(\cdot -\alpha_j)\|_{L^p(B_j)} \le C_p \ \Delta_j^{\frac 2p-1},
\]
and similarly for $ \| K(\cdot -\alpha_j')\|_{L^p(B_j)}$. Next, from
\eqref{jalphajalphaprimediff} one can see that
\[
| K(x-\alpha_j) -  K(x-\alpha_j'| \ \le  \frac{ \Delta_j }{|x-\bar \alpha_j|^2} \qquad\mbox
{ if }|x-\bar \alpha_j| \le 2
\Delta_j,
\]
and then after a short calculation, one finds that
\[
 \| K(\cdot -\alpha_j) - K(\cdot - \alpha_j')\|_{L^p( \R^2 \setminus B_j)} \le C_p\ |\Delta_j|^{\frac 2p-1}.
\]
Combining these bounds, we obtain
\eqref{jstarLpbound}. A similar argument yields \eqref{jstarL2bound}.

\end{proof}

{Finally,  we record some properties,
mostly already proved in \cite{BJS}, of the energy $\E_\e(u)$ renormalized at infinity 
\[
\E_\e(u) =
\lim_{R\to \infty}\int_{B(R)} [e_\e(u) - \frac 12 |\nabla U^D|^2]\  dx,
\qquad \quad\mbox{$U^D$ fixed in \eqref{UD.def}, \eqref{Psin0}.}
\]
}
\begin{lemma}\label{lem.Edef}
If $u\in [U^D] + H^1$, then
$\E_\e(u)$
is well-defined and finite. Moreover, if $(\chi_R)_{R\ge 1}$
is any family of smooth, compactly supported functions 
satisfying \eqref{chiRdef1}, \eqref{chiRdef2}
then 
\begin{equation}
\E_\e(u) \ = \  
\lim_{R\to \infty} \int_{\R^2} \chi_R [e_\e(u) \, - \frac 12 |\nabla U^D|^2] \ dx .
\label{E.redef}\end{equation}
\end{lemma}

\begin{proof}
Following \cite{BJS}, we  write $u = \tilde U^D +v$ with $v\in H^1$, where $\tilde U^D = \tau_y U^d$
for some $y\in \R^2$. Then
\[
e_\e(u) - \frac 12 |\nabla \tilde U^D|^2  =
\frac 12| \nabla v|^2 - \Delta \tilde U^D \cdot v + \frac 1{4\e^2}(|u|^2-1)^2
+ \nabla \cdot ( \nabla \tilde U^D \cdot v).
\]
The hypotheses imply that the first three terms on the right-hand side are integrable over
$\R^2$, so in view of the dominated convergence theorem
it suffices  (after integrating by parts) to check that 
\[
0 =
 \lim_{R\to \infty}\int_{\partial B_R} \p_\nu \tilde  U^D \cdot  v \ d\calH^1
 = 
 \lim_{R\to \infty}\int_{\R^2} \nabla\chi_R  \cdot ( \nabla \tilde U^D \cdot v) \ dx.  
\]
The first of these is established in \cite{BJS} Lemma 3.3, and the second 
follows easily from properties of $\chi_R$.
\end{proof}

We remark that the Lemma implies that if $u, \widetilde u\in [U^D]+H^1(\R^2)$, then
\begin{equation}
\lim_{R\to \infty} \int_{B(R)} [e_\e(\widetilde u) - e_\e(u)] \ dx
= 
\lim_{R\to \infty} \int \chi_R [e_\e(\widetilde u) - e_\e(u)]  \ dx
= \E_\e(\widetilde u) - \E_\e(u).
\label{limits.exist}\end{equation}

\section{Energy bounds on annuli  }
\label{sec.annuli}

Recall that $R_\alpha = 4 \vee 4 \max _j |\alpha_j|$.
In this section we will prove the following result.

{\begin{proposition} \label{Propannularlower}
Assume that $u\in [U^D] + H^1(\R^2)$ and that 
\begin{equation}
\| J(u) - \pi \sum  d_i\delta_{\alpha_i} \|_{X^*_{\ln}}  \le s_\e, 
\qquad\quad\sum d_i = D>0.
\label{annlbd.h1}\end{equation}
There
exists a universal constant $c_1$ such that if 
 $s_\e D^2 \le  {c_1}$ and 
$\e D^2 <  {c_1}$,
then
for any $R_1\ge R_\alpha$ and $R_2\ge 2R_1$, we have
\begin{equation}
\int_{B_{R_2}\setminus B_{R_1}} e_\e(u) \ dx \ - \ \pi D^2 \log(\frac {R_2}{R_1})\  \ge
  -  CD^4 \frac{\e}{R_1}.
\label{annulus.lbd}\end{equation}
Consequently, 
for every $R \geq R_\alpha$
\begin{equation} \label{energyupperboundBigBall}
\int_{ B_R}   e_\e (u) \ dx \leq \E_\e(u) +  \pi D^2 \ln R  +   CD^4 {\e \over R}.
\end{equation}
\end{proposition}}

{Note that the hypotheses allow $D$ to grow asymptotically large as $\e \to 0$, as will be necessary when we consider the hydrodynamic limit.}

\begin{proof}
{We may assume that $|u|\le1$, as otherwise we can replace $u$ by
$\min(1, |u|^{-1})u$. By a routine regularization argument, we may also assume that $u$ 
is smooth. We will focus on proving \eqref{annulus.lbd}, since  \eqref{energyupperboundBigBall}
is a direct consequence.}

1. For $r\in [R_1,R_2]$ let $m(r) = \inf_{\partial B_r}|u|$, and if $m(r)\ge \frac 12$ then
let $d(r) := \deg(u; \partial B_r)$. 
We define
\[
G := \{r\in [R_1, R_2] : m(r) \ge \frac 12, d(r) \ge D\},
\qquad
B :=  [R_1, R_2]\setminus G.
\]
If $\frac 12$ is a regular value of $|u|$, which we will assume to be the case, then $B$ consists of a finite union of relatively open intervals. (Otherwise we can replace $\frac 12$ by some nearby regular value of $|u|$
in the definition of $G$.)

It follows from the proof of Theorem 2.1 in   \cite{JerrardSIMA} that
\begin{equation}\label{good.r.est}
\int_{\partial B_r} e_\e(u) d\calH^1 \ge  \pi\frac  {m^2(r) d^2(r)}r + \frac 1 {C\e} (1 -m(r))^2 
\ge 
\frac{\pi d^2(r)}{r +  c_0\e d^2(r)} 
\quad
\quad
\mbox{if $m(r)\ge \frac 12$ }
\end{equation}
for a {universal constant $c_0$. The same proof also shows that}
\begin{equation}
\int_{\partial B_r} e_\e(u) d\calH^1 \ge \frac C \e \qquad
\quad
\mbox{if $m(r) <  \frac 12$ }.
\label{bad.r.est}\end{equation}
Thus
\begin{align*}
\int_{B_{R_2}\setminus B_{R_1}} e_\e(u) \ dx
&
\ge
\int_{R_1}^{R_2} \left( \int_{\partial B_r} e_\e(u) d\calH^1 \right) dr
\\
&
\ge
\int_{R_1}^{R_2}  \pi\frac{ D^2}{r+{c_0D^2} \e} \  dr
+ 
\int_B \left(
\int_{\partial B_r} e_\e(u)  d\calH^1 \ - \   \pi\frac{ D^2}{r+{c_0D^2}\e}
\right)dr 
\\
&
\ge
\pi D^2 (\ln(\frac{R_2}{R_1}) - {c_0D^2} \frac \e{R_1}) 
+
\int_B \left(
\int_{\partial B_r} e_\e(u)  d\calH^1 \ - \   \pi\frac{ D^2}{r+{c_0D^2}\e}
\right)dr 
\end{align*}
for ${c_0D^2} = c_0 D^2$. 
Thus we only need to show that the last term on the right-hand side is 
always positive, and to do this 
it suffices to demonstrate that {there exists some constant $c_1$ 
such that every component $I$ of $B$
satisfies
\begin{equation}\label{ann.bd.component}
\int_I  \pi\frac{ D^2}{r+{c_0D^2}\e} \ dr \ \le \ \int_I \int_{\partial B_r} e_\e(u) \, d\calH^1 \ dr
\qquad\quad
\mbox{ if }\e, s_\e \le \frac {c_1}{D^2}. 
\end{equation}}

2. Fix a component $I$  of $B$, 
and let $a<b$ denote its endpoints.
We claim that
\begin{equation}\label{I.est}
|I| \le    s_\e(1+\ln^+b)+ C \e \int_{B_b\setminus B_a}e_\e(u). 
\end{equation}
To see this,
define
\[
I_1 := \{ r\in I : m(r) \ge \frac 12, d(r) < D\},
\qquad
\qquad
I_2 := \{ r\in I : m(r) < \frac 12\}.
\]
{Note that $I = I_1\cup I_2$.}
It is clear from \eqref{bad.r.est} that 
\begin{equation}\label{I2.est}
|I_2| \le C \e \int_{B_b\setminus B_a}e_\e(u) \ dx .
\end{equation}
To estimate $|I_1|$, we define a test function
$\phi  \in X_{\ln}$ by specifying that $\phi$ is Lipschitz
with support in $B_b$, and that
\[
\nabla \phi(x) = \left\{ \begin{array}{ll} 
-{x \over |x|}    & \hbox{ for {\em a.e. }$x$ such that  } |x|   \in I_1  \\
0 & \hbox{ for {\em a.e. }$x$ such that } |x|  \not\in I_1. \end{array} \right.
\]
The definitions imply that $\| \phi \|_{X_{\ln}} \le 1+\ln^+b$,
so we deduce from \eqref{annlbd.h1}  that
\[
 \left | \int_{\R^2} \phi \left(  J(u) - \pi \sum d_j \delta_{\alpha_j} \right) dx  \right| 
\le 
\| \phi \|_{X_{\ln}} \ \|  J(u) - \pi \sum d_j \delta_{\alpha_j} \|_{X_{\ln}^*}
 \le  s_\e(1+\ln^+b).
\]
The definition of $\phi$ also implies that $\phi = |I_1|$ on $B_a$,
so 
\begin{align*}
s_\e (1+\ln^+b) \ \ge \ \int_{\R^2} \phi ( \pi \sum d_j \delta_{a_j}- J(u) ) dx  
& =\pi D |I_1|  + \frac 12\int_{\R^2} \nabla \phi \times j(u) dx  \\
&=
\int_{I_1}\left(\pi D +  \frac 12\int_{\partial B_r} \nabla \phi \times j(u) d\calH^1  \right) dr.
\end{align*}
For any $s \in I_1$, we can write $u|_{\partial B_s}$ in the form $u = \rho e^{i \varphi}$,
with a  well-defined phase $\varphi$.  Then $j(u) = \nabla \varphi +  {j(u) \over |u|} { |u|^2 - 1 \over |u|}$,
and since $|u|\ge \frac 12$ and $\|\nabla \phi\|_{L^\infty} \le 1$, straightforward estimates lead to
\[ 
\left | \int_{\partial B_r} \nabla \phi \times j(u) d\calH^1 - \int_{\partial B_r} \nabla \phi \times \nabla \varphi d\calH^1 \right |
\ \leq \  2\e \int_{\partial B_r} e_{\e}(u) d\calH^1 .
\] 
And the definition of $\phi$ implies that 
$-\nabla\phi \times \nabla\varphi = \tau \cdot\nabla \varphi$ on $\partial B_r$ for a.e. $r\in I_1$,
where $\tau(x) = \frac 1{|x|} (-x_2,x_1)$ is the unit tangent to $\partial B_r$ with the standard
(counterclockwise) orientation. It follows that
\[
- \int_{\partial B_r} \nabla \phi \times \nabla \varphi d\calH^1 \ = 2\pi\, d(r)
\]
for a.e. $r\in I_1$. Since 
$d(r)\le D-1$ for every $r\in I_1$, we combine the previous few inequalities
to find that
\[
s_\e (1+\ln^+b)  \ \ge \   \pi |I_1| -  \e \int_{I_1} \int_{\partial B_r} e_\e(u) d\calH^1 \ dr.
\]
Upon rearranging and combining this with \eqref{I2.est}, we 
obtain the claim \eqref{I.est}.

3. We now prove \eqref{ann.bd.component}.

First, note that if $I_1$ is empty, then $I = I_2$ and it follows from \eqref{I2.est}
(recalling that $R_1\ge 1$ by hypothesis) that 
\begin{equation}
\int_I \int_{\partial B_2}e_\e(u) \,d\calH^1 \ dr \ge \int_I \frac C \e  \ dr
\ge \int_I \frac{\pi D^2}{r+{c_0D^2}\e}dr
\qquad\mbox{ if } D^2 \e < \frac C{\pi }  
\label{I2empty}\end{equation}
and as a result that \eqref{ann.bd.component} holds {for a suitable choice of $c_1$.}

We henceforth assume that $I_1$ is nonempty, and we consider two cases.

{Case 1}: $|I| < \delta b$ for some $\delta \in (0,\frac 14)$ to be chosen.

If this holds, then since $R_2 \ge 2R_1$, the interval $I$ must be a proper subset of
$[R_1, R_2]$ and it follows that at least one endpoint must belong to $G$.
For concreteness we assume that $b\in G$; the other case is essentially identical. Fix some $r\in I_1$, and
let 
\[
\tilde u(x) := \begin{cases}u &\mbox{ if }|u|\le \frac 12\\
\frac u{2|u|} &\mbox{ if }|u|\ge \frac 12.
\end{cases}
\]
Since $b\in G$ and $r\in I_1$, the definition of $\tilde u$ implies that 
we can write
\[
\tilde u|_{\partial B_b} = \frac 12 e^{i \varphi_b}, \qquad
\tilde u|_{\partial B_r} = \frac 12 e^{i \varphi_r}, \qquad
\]
with $\varphi_b$ and $\varphi_r$ changing by at least  $2\pi D$ and at most $2\pi (D-1)$
respectively on the circles on which they are defined. 
Then, since $e_\e(u) \ge  \frac 12 |\nabla  \tilde u|^2  \ge J(\tilde u) = \frac 12 \nabla \times j(\tilde u)$ pointwise, and 
\begin{align*}
\int_{B_b\setminus B_a} e_\e(u) \, dx
\ \ge \ 
\int_{B_b\setminus B_r} J (\tilde u )\, dx
\ &= \
 \frac 12\int_{\partial B_b} j(\tilde u) \cdot \tau d\calH^1 - 
\frac 12\int_{\partial B_r} j(\tilde u) \cdot \tau d\calH^1 \\
&= \
\frac 18 \left( 
\int_{\partial B_b} \nabla \varphi_b\cdot \tau  d\calH^1
-
\int_{\partial B_r} \nabla \varphi_r\cdot \tau  d\calH^1\right)\\
& \
 \ge  \frac \pi 4 .
\end{align*}
On the other hand, $b-a = |I| \le \delta b$ by assumption, so
\[ 
\int_I \frac{\pi D^2}{r+{c_0D^2}\e} dr 
\ \le \ 
\pi D^2 \int_{(1-\delta) b}^b  \frac {dr}r \le \pi D^2 \ln ( \frac 1{1-\delta}).
\]
We now fix $\delta$ small enough that $\pi D^2 \ln(\frac 1{1-\delta}) \le \frac \pi 4$.
For example, we may take $\delta = \frac 1{CD^2}$ for some universal $C$.
Then  \eqref{ann.bd.component} follows by combining the above two inequalities.

{{Case 2}. $|I| \ge \delta b$ for $\delta=  \frac 1{CD^2}$ as fixed above.
Then $s_\e (1+\ln^+b) \le \frac 12 
\delta b$ for all $b \ge 1$ so long as $s_\e D^2 \le \frac 12 \delta = \frac 1{CD^2}$, and then it follows from \eqref{I.est} that
\[
\int_{B_b\setminus B_a} e_\e(u) \ dx\ge \frac 1{C\e}(|I|- s_\e(1+\ln^+b)) \ge \frac
1{2C\e}|I|  = \int_I \frac 1{2C\e} \ dr.
\]
Then \eqref{ann.bd.component} follows, after adjusting $c_1$ if necessary,
by  the same estimate as  in the case when $I_1$ was empty,
see \eqref{I2empty} above.}
\end{proof}


\section {Quantitative rate of second order $\Gamma$-convegence}\label{s:gstab}


Our $\Gamma$-stability result for infinite-energy configurations on $\R^2$ is

\begin{proposition}  \label{Propgammastability}
There
{exists a constant $c_2\in (0,1)$  such that}
for any $u\in [U^D] + H^1(\R^2)$ with $D\neq 0$, if 
\begin{equation}  
\| J(u) - \sum_{j=1}^n \pi d_j \delta_{\alpha_j} \|_{X^*_{\ln}} \leq s_\e \quad \hbox{ for some }
{s_\e \in [\e {\sqrt{\ln(\rho_\alpha/\e)} \over \ln R_\alpha} , 
\frac {c_2 \rho_\alpha} { n^3 \ln R_\alpha  }  ], }   
\label{gstab.h1}\end{equation}  
for some $\alpha = \left( \alpha_1, \ldots, \alpha_n \right) \in \R^{2n*}$ 
and $d\in \{ \pm 1 \}^n$,
and if 
\begin{equation} \label{sigmastardefinition}
{8  s_\e \ln R_\alpha  }\le \sigma_\star 
:= \sqrt{ \frac{\rho_\alpha}{n^{3}}(s_\e \ln R_\alpha + \e( \E_\e(u) +  \pi D^2 \ln R_\alpha))}
 \le  \frac{\rho_\alpha}{40 n } , 
 \end{equation}
then we have  {(using notation introduced in \eqref{cutout})},
\begin{equation}
\begin{split}
 \frac12 \int_{\R^2_{\sigma_\star} (\alpha)} e_\e(|u|)  + 
\frac 14\left| { j(u) \over |u|} - j(u_*) \right|^2
dx  \le \ 
\Sigma(u;\alpha, d) +
{C {n^4 \over \rho_\alpha} \sigma_\star} \\
 \le 
 \Sigma(u;\alpha, d)+\sqrt {s_\e}\, \calP
\end{split}
\label{gstab.c1}
\end{equation}
for a constant $C$, where
$\Sigma(u;\alpha,d) = \E_\e(u) - W_\e(\alpha,d)$ and $u_* = u_*( \cdot; \alpha, d)$ and $\calP$ are defined in \eqref{canonicalharmonicmap} and
\eqref{P.def}, respectively. 
Finally, we  have the estimates 
\begin{equation}
\label{moduL4L2}
\| 1-|u| \|_{L^4(\R^2)}^4  \ \le \  \| 1-|u|^2 \|_{L^2(\R^2)}^2 \ \le \   {\e^2\calP,}
\end{equation}
\begin{equation}\label{Lpcurrentboundsbigball}
 \LN j(u) - j(u_*) \RN_{L^{4/3}+L^2(\R^2)} 
\leq 
\ \sqrt{|\Sigma(u;\alpha,d)|} \ + \    s_\e^{1/4}\calP.
\end{equation}
\end{proposition}


{The  choice of $\sigma_\star$ arises from optimizing an error estimate in the 
proof of \eqref{gstab.c1}. }

In this proposition, the bounds appearing on the right-hand side are expressed in terms of the
energy $\E_\e (u)$ renormalized at infinity, whereas in the corresponding result
in \cite{JSp2}, the right-hand side involves the whole energy $E_\e$.

{Note that \eqref{gstab.h1} implies that  
\begin{equation} \label{restrict.ep.gamma} 
 s_\e \le 
\frac {c_2}{n^3}, \qquad
\e\le \frac {c_2}{n^3} \frac{ \rho_\alpha}{ \sqrt{\ln(\rho_\alpha/\e)}} \le \frac {c_2}{n^3}
 \end{equation} 
 and hence that the hypotheses of Proposition  \ref{Propannularlower}
are satisfied, as long as we insist that $c_2\le {c_1}$.
Some of the restrictions  in \eqref{gstab.h1} are also needed as hypotheses for
some "black box" results that we will use from \cite{JSp2}, see \eqref{differencesetareabound} below. }

%

An essential ingredient of the proof of Proposition \ref{Propgammastability} are the following lemmas which provide a lower bound for a single vortex in a ball.  

\begin{lemma}[\cite{JSp}, Lemma 6.8] \label{Singlevortexdecaylemma}
Let 
\[
I( \sigma, \e) := \inf \Big\{  \int_{B_\sigma} e_\e(u) \ dx \ : \ 
 u \in H^1(B_\sigma; \C), \  u= e^{i\theta} \hbox{ on } \p B_\sigma \Big\}.
\] 
Then $I(\sigma, \e) = I(\sigma/ \e , 1)$, and there exists a constant $\gamma$ such that 
\begin{equation}\label{gamma.def}
\LV I(\sigma,\e) -  \LB \pi \ln {\sigma \over \e} + \gamma \RB \RV \leq C {\e^2 \over \sigma^2}.
\end{equation}
\end{lemma}

and

\begin{lemma}[\cite{JSp}, Theorem 1.3] \label{Singlevortexenergylemma}
There exists an absolute constant $C$ such that if $u \in H^1(B_\sigma; \C)$ 
satisfies
\[
\LN J(u) \pm \pi \delta_0 \RN_{\dot{W}^{-1,1}(B_\sigma)} \leq {\sigma \over 4}
\]
then 
\begin{equation}
0 \leq \Sigma_{B_\sigma}(u) + C {\e \over \sigma} \sqrt{ \ln { \sigma\over \e}} + {C \over \sigma} \LN J(u) \pm \pi \delta_0 \RN_{\dot{W}^{-1,1}(B_\sigma)} 
\end{equation}
where 
\[
\Sigma_{B_\sigma(\alpha)}(u) = \int_{B_\sigma(\alpha)} e_\e(u) \ dx - I(\e, \sigma).
\] 
\end{lemma}

The rest of this section is devoted to the

\begin{proof}[proof of Proposition \ref{Propgammastability}]
The outline of the proof is the same as that of 
Theorem 2 in \cite{JSp2}. The new point is to check that 
far-field contributions to the energy can be controlled
by  the energy $\E_\e(u)$ renormalized at infinity and 
information \eqref{gstab.h1} about the $X_{\ln}^*$ norm of $J(u)$.

1. 
{For $\sigma$ to be chosen later, but satisfying
\begin{equation}\label{sig.constraints}
8 s_\e \ln R_\alpha \le \sigma \le \frac{\rho_\alpha}{40 n}
\end{equation}
we compute (recalling the definitions \eqref{EE.def} and \eqref{Wep.def}
of $\E_\e(u)$ and $W_\e(\alpha,d)$) }
\begin{align}
\Sigma(u; \alpha, d)
& = \E_\e(u) - W_\e(\alpha, d) \nonumber \\
& = \lim_{R\to \infty}   \LB\int_{B_{R}}  e_\e(u) \ dx  \ - \   \pi D^2 \ln R\RB - \LB n \LC \pi \ln\frac 1 \e + \gamma \RC + 
  W(\alpha, d)   \RB  \nonumber\\
  & = \lim_{R\to \infty}  \int_{B_{R, \sigma}} e_\e(u) \ dx -  \LB \pi D^2 \ln R + n \pi  \ln\frac 1 \sigma  + W(\alpha, d) \RB    \nonumber \\
& \quad   +  \sum_{j=1}^n  \LB \int_{B_\sigma(\alpha_j)} e_\e(u)  \ dx-  \LC  \pi \ln \frac \sigma \e + \gamma \RC  \RB   \nonumber\\
    & \stackrel{\eqref{FiniteWapprox}}
    {\geq} 
\lim_{R\to \infty}  \int_{B_{R, \sigma}} \LB e_\e(u)  -  e_\e(u_*) \RB \ dx + \sum_{j=1}^n 
 \Sigma_{B_\sigma}(u)    - C {n^3 \sigma^2 \over \rho_\alpha^2}.
\label{EEEa}\end{align}
{By \eqref{convertnorm},  \eqref{gstab.h1}, and \eqref{sig.constraints}, }
\begin{align*}
\LN J(u) - \pi d_j \delta_{\alpha_j} \RN_{\dot{W}^{1,1}(B_\sigma(\alpha_j))}
 & \leq \LC 1 + \LC \ln (\sigma + |\alpha_j|) \RC^+ \RC 
 \LN J(u) - \pi \sum_{j=1}^n d_j \delta_{\alpha_j}  \RN_{X_{\ln}^*} \\ 
 & \leq 2 s_\e \ln R_\alpha < {\sigma \over 4} .
\end{align*}
Then Lemma \ref{Singlevortexenergylemma} and the assumption
that $s_\e\ln R_\alpha \ge \e\sqrt{\ln (\rho_\alpha/\e)}$ 
imply that
\[
 \Sigma_{B_\sigma}(u) 
  \geq - C {\e \over \sigma} \sqrt{ \ln { \sigma \over \e}}
- C { s_\e \ln R_\alpha  \over \sigma} \gtrsim  - {s_\e  \ln R_\alpha \over \sigma},
\]
and it follows that
\begin{equation} \label{ExpansionofEnergyExcess}
\lim_{R\to \infty}\int_{B_{R,\sigma}} [e_\e(u) - e_\e(u_*) ] \ dx
\leq \Sigma(u; \alpha, d) + 
 C s_\e  { n \ln R_\alpha  \over \sigma} + C {n^3 \sigma^2 \over \rho_\alpha^2}.
\end{equation}

2.
We will use the identity
\begin{equation}
\begin{split}
e_\e(|u|) + 
\frac 12\left| \frac {j(u)}{  | u | } -  j(u_*)  \right|^2  
& = 
[e_\e(u) - e_\e(u_*)]  + j(u_*)\cdot \LC j(u_*) - \frac{j(u)}{|u|} \RC .
\end{split}
\label{eqn:quotient}\end{equation}
Let $(\chi_R)_{R>1}$ denote a family of functions to be specified shortly,
which will be chosen to satisfy \eqref{chiRdef1}, \eqref{chiRdef2}.
We multiply both sides of \eqref{eqn:quotient} by
$\chi_R$,
integrate over $\R^2_\sigma$, and let $R$ tend to $\infty$.
In view of \eqref{limits.exist}, \eqref{ExpansionofEnergyExcess} and
properties of $(\chi_R)_{R\ge1}$,
this yields
\begin{align}
\int_{\R^2_\sigma}
e_\e(|u|) + 
\frac 12\left| \frac {j(u)}{  | u | } -  j(u_*)  \right|^2  \ dx
&\le
\Sigma(u;\alpha, d) +
 C s_\e  { n \ln R_\alpha  \over \sigma} + C
 {n^3 \sigma^2 \over \rho_\alpha^2}
 \nonumber\\
&\hspace{3em}
+
\lim_{R\to \infty}\int_{\R^2_\sigma} \chi_R  j(u_*)\cdot \LC j(u_*) - \frac{j(u)}{|u|} \RC \ dx.
\label{reducetoA}\end{align}
So our main task is to estimate the last term.
It is convenient to define
\[
\mathcal{G}(x; \alpha, d) = \sum_{j=1}^n d_j \ln |x - \alpha_j |, \qquad \mbox{ and }
\]
\[ 
\widetilde \R^2_\sigma 
:= 
\hbox{  largest open subset of $\R^2_\sigma$ such that }
\mathcal{G} \hbox{ is constant on
every component of } \partial \widetilde \R^2_\sigma. 
\] 
Thus, for $\sigma$ small, $\widetilde \R^2_\sigma$ looks like 
$\R^2$ with a (closed) slightly distorted ball of radius $\approx \sigma$ 
removed around  each $\alpha_i$.

We will write 
\[
\int_{\R^2_\sigma} \chi_R  j(u_*)\cdot \LC j(u_*) - \frac{j(u)}{|u|} \RC \ dx = A_1 + A_2+ A_3
\]
where
\begin{align*}
A_1
&:= 
\int_{\widetilde \R^2_\sigma} \chi_R\ j(u_*) \cdot (j (u_*) - j(u)) \ dx, \\
A_2
&:=
\int_{\widetilde \R^2_\sigma} \chi_R \  j(u_*) \cdot \frac{j(u)}{|u|} \ (|u| -1) \ dx  ,\\
A_3
&:=
\int_{\R^2_\sigma \setminus \widetilde \R^2_\sigma} \chi_R \ j(u_*) \cdot (j(u_*) -\frac{ j(u)}{|u|}) \ dx.
\end{align*}

{\bf Estimate of $A_3$}: 
In \cite{JSp2}
it is checked\footnote{In \cite{JSp2} this is proved for a bounded domain $\Omega$, but the 
proof is also valid on $\R^2$, where in fact it becomes a little simpler, since there are no 
boundary terms in ${\mathcal{G}}$.}
that
\begin{equation} \label{differencesetareabound}
|\R^2_\sigma \setminus \widetilde \R^2_\sigma| 
\lesssim {n^2 \sigma^3 \over \rho_\alpha},
\end{equation}
as long as $\sigma \le \rho_\alpha/40n$, {which we have assumed in \eqref{sig.constraints}.}
Then by \eqref{Ggradientestimateerror1} and \eqref{differencesetareabound}
\begin{align}
|A_3|
& \leq  \int_{\R^2_\sigma \setminus \widetilde \R^2_\sigma}  \LV j(u_*) \RV^2 
+ {1\over 4}  \LV j(u_*) - {j(u) \over |u|} \RV^2  \ dx
\nonumber
\\
& \leq  C {n^4 \sigma \over \rho_\alpha} 
+ {1\over 4} \int_{\R^2_\sigma}  \LV j(u_*) - {j(u) \over |u|} \RV^2  \ dx.
\label{A3.mainest}
\end{align}

{\bf Estimate of $A_2$}:
Since $\widetilde \R^2_\sigma \subset \R^2_\sigma$ we have (using the notation \eqref{cutout})
\begin{align*}
|A_2| & \leq 
 \int_{\R^2_\sigma} \LV j(u_*) \RV  \LV  {j(u) \over |u|} \RV \LV 1 - |u| \RV \ dx
 \\
 & \leq   \int_{B_{R_\alpha},\sigma} \LV j(u_*) \RV \LV {j(u) \over |u|}  \RV \LV 1 - |u|^2 \RV \ dx
 \\
 &\qquad \qquad\qquad
 +  \int_{\R^2 \backslash B_{R_\alpha}} \LV j(u_*) \RV^2  \LV 1 - |u|^2 \RV + \LV j(u_*) \RV \LV {j(u) \over |u|}  - j(u_*) \RV \LV 1 - |u|^2 \RV \ dx.
\end{align*}
From \eqref{energyupperboundBigBall} we have 
\begin{align*}
 \int_{B_{{R_\alpha}, \sigma}} \LV j(u_*) \RV \LV {j(u) \over |u|} \RV \LV 1 - |u|^2 \RV  \ dx
& \leq C \e \LN j(u_*) \RN_{L^\infty( B_{{R_\alpha},\sigma} )}
\int_{B_{R_\alpha} }  e_\e(u)   \ dx
\\
& \le 
 C \e {n \over \sigma} 
\left(
\E_\e(u) +  \pi D^2 \ln R_\alpha\right).
\end{align*}
On the other hand, 
\begin{align*}
& \int_{\R^2 \backslash B_{R_\alpha}} \LV j(u_*) \RV^2  \LV 1 - |u|^2 \RV + \LV j(u_*) \RV \LV {j(u) \over |u|}  - j(u_*) \RV \LV 1 - |u|^2 \RV \ dx \\
& \stackrel{\eqref{Ggradientestimateerror}}{\leq} C \e \int_{R_\alpha}^\infty {n^4 \over r^3} dr + \e \int_{\R^2 \backslash B_{R_\alpha}}  e_\e(|u|)  \ dx\\
& \quad 
+ C {n \e\over R_\alpha} 
\int_{\R^2 \backslash B_{R_\alpha}} \left[ \LV { j(u) \over |u|} - j(u_*)\RV^2 + 
 e_\e(|u|) \right]  \ dx \\
& \leq C \e { n^4 \over R_\alpha^2} + C \e{n \over R_\alpha}\LB  \int_{\R^2 \backslash B_{R_\alpha}} e_\e(|u|) + {1\over2} \LV {j(u) \over |u|} -j(u_*) \RV^2 \RB \ dx + \e \int_{\R^2 \backslash B_{R_\alpha}} e_\e(|u|) \ dx.
\end{align*}
{It follows from  
\eqref{restrict.ep.gamma}, after taking $c_2$ smaller if necessary, that
 $C {\e  n \over R_\alpha} \leq {1\over 10}$, and then}
\begin{equation} \label{A1estimatebound}
|A_2| \leq C {\e n \over \sigma} \LB \E_\e(u) + \pi D^2 \ln R_\alpha \RB + C {\e n^4 \over R_\alpha^2} +{1 \over 8}  \LB \int_{\R^2_\sigma} e_\e(|u|) + {1\over4} \LV { j(u) \over |u|} - j(u_*) \RV^2  \RB \ dx.
\end{equation}
Estimate 
\eqref{A1estimatebound} will be used also in the proof of the localization result, Proposition~\ref{P2}.  

{\bf Estimate of $A_1$}:
First note that, exactly as in \cite{JSp2}, we can write
\begin{equation}
\chara_{\widetilde \R^2_\sigma} j(u_*) := \nabla \times \tilde {\mathcal{G}}_\sigma
\label{Gts1}\end{equation}
where $\chara$ is the characteristic function,
and  $\tilde {\mathcal{G}}_\sigma$ is the Lipschitz continuous function
of the form
\begin{equation}
\tilde {\mathcal{G}}_\sigma = \begin{cases}
{\mathcal{G}}(\, \cdot \,; \alpha,d) &\mbox{ in }\widetilde \R^2_\sigma\\
\mbox{constant}&\mbox{ in each component of  $\R^2\setminus \widetilde \R^2_\sigma$.}
\end{cases}
\label{Gtildesig}\end{equation}
This is the point of replacing $\R^2_\sigma$ by $\widetilde \R^2_\sigma$
in our above decomposition.
It is straightforward to check from \eqref{Gts1}, \eqref{Gtildesig}
and \eqref{Ggradientestimateerror1} that 
\begin{equation}
\LN \tilde {\mathcal{G}}_\sigma   \RN_{X_{\ln }}  \lesssim {n \ln R_\alpha \over \sigma}  \label{boundGsigmaXnorm}
\qquad
\qquad
\mbox{for $\sigma \leq \rho_\alpha$.}
\end{equation}

We now specify that 
\[
\chi_R =
\begin{cases}
1&\mbox{ in }B_R\\
f_R({\mathcal{G}})	&\mbox{ in }\R^2\setminus B_R
\end{cases}
\]
where $f_R:\R\to [0,1]$ is smooth and satisfies
\[
f_R(s)=  \begin{cases}
1 &\mbox{ if } s \le D \ln R + \beta \\
0 &\mbox{ if } s \ge D \ln 2R - \beta
\end{cases},\qquad
\mbox{ for some $\beta>0$}
\]
with $|f_R'|$ and $|f_R''|$ bounded uniformly in $R$.

We must check that $\chi_R$ satisfies \eqref{chiRdef1}
and \eqref{chiRdef2} for all large $R$. Toward this end,
note that if $|x|$ is sufficiently large
(which implies that $\tilde {\mathcal{G}}_\sigma(x) = {\mathcal{G}}(x)$), then 
\begin{equation}
\left|\tilde {\mathcal{G}}_\sigma(x) - D\ln |x| \right|
\ = \ 
\left|\sum_{i=1}^n d_i ( \ln|x-\alpha_i| - |\ln |x| ) \right|
\ = \ 
\left|\sum_{i=1}^n d_i  \ln \frac {|x-\alpha_i|}{|x|} \right|
\ \le  2  n \frac{R_\alpha}{|x|},
\label{Gsig1}\end{equation}
and it follows from this that if $R$ is large enough (which we henceforth assume to
be the case), then $\chi_R$ is smooth and
satisfies \eqref{chiRdef1}.
Moreover, \eqref{chiRdef2} follows by noting that
\begin{equation}
|\nabla \chi_R| = |f'_R( {\mathcal{G}} )| \ |\nabla {\mathcal{G}}_\sigma|
\le \frac {C(n,D)}R \qquad \mbox{ on }B_{2R}\setminus B_R \supset \mbox{supp}(\nabla \chi_R).
\label{chiRlip}\end{equation}

Now
\[
\chi_R \nabla \times {\mathcal{G}}_\sigma = 
\nabla \times ( \chi_R \tilde {\mathcal{G}}_\sigma) -  \tilde {\mathcal{G}}_\sigma\nabla\times \chi_R
\]
and, since $\tilde {\mathcal{G}}_\sigma= {\mathcal{G}}$ on $\mbox{supp}\nabla\chi_R$, our choice of $\chi_R$ implies that 
\[
\tilde {\mathcal{G}}_\sigma \nabla\times \chi_R = \nabla \times Y_R\qquad\qquad\mbox{ for }Y_R := 
\begin{cases}
\mbox{constant} &\mbox{ in }B_R\\
F_R({\mathcal{G}})&\mbox{ in }\R^2\setminus B_R
\end{cases}
\]
where the constant is chosen to make $Y_R$ continuous (in fact it is $F_R(0)$) and 
\[
F_R(s) =  -\int_s^\infty t f_R'(t) dt.
\]
Since $\chi_R {\mathcal{G}}$ and $Y_R$ are compactly supported, we can integrate by
parts to find that
\begin{align*}
A_1 
&
\overset {\eqref{Gts1}}
=
\int_{\R^2}  \chi_R \ \nabla\times \tilde {\mathcal{G}}_\sigma \cdot (j(u_*)-j(u)) \ dx\\
&=
2\int_{\R^2} \left[
\chi_R\tilde {\mathcal{G}}_\sigma - Y_R
\right]
 \left(
J(u) - \sum \pi d_i \delta_{\alpha_i}
\right) \ dx.
\end{align*}
Thus
\[
|A_1| \lesssim \left( 
\left\| \chi_R \tilde {\mathcal{G}}_\sigma \right\|_{X_{\ln}}
+\left\| Y_R \right\|_{X_{\ln}}
\right) 
\left\|
J(u) - \sum \pi d_i \delta_{\alpha_i} \right\|_{X^*_{\ln}}.
\]
We know from \eqref{normlim} and \eqref{boundGsigmaXnorm} that 
\[
\lim_{R\to\infty}\left\| \chi_R \tilde {\mathcal{G}}_\sigma \right\|_{X_{\ln}} =
\left\| \tilde {\mathcal{G}}_\sigma \right\|_{X_{\ln}} \le n \ln R_\alpha/\sigma,
\]
and it is easy to check from \eqref{chiRlip} that  $\| Y_R\|_{X_{\ln}}\to 0$
as $R\to \infty$, so we conclude that
\begin{equation}
\limsup_{R\to \infty}|A_1| \lesssim s_\e n \ln R_\alpha/\sigma.
\label{A1.est}\end{equation}



3. 
By combining \eqref{reducetoA} with the above estimates of $A_1, A_2$ and $A_3$,
we {find that
\begin{equation} \label{fixedradiusgammastability} \begin{split}
& \LB \frac12 \int_{\R_\sigma(\alpha)} e_\e(|u|) + {1\over 4} \LV { j(u)\over |u| } - j(u_*) \RV^2  \ dx \RB
- \Sigma(u; \alpha, d)  \\
& \qquad \qquad \qquad  \lesssim  {n \over \sigma} \LB s_\e \ln R_\alpha + \e \E_\e(u)  + \e \pi D^2 \ln R_\alpha \RB + { n^4 \sigma \over \rho_\alpha}  
\end{split}\end{equation}
for any $\sigma$ satisfying \eqref{sig.constraints}.}
If $\sigma_\star = \sqrt{ {\rho_\alpha \over n^3} \LC s_\e \ln R_\alpha + \e \E_\e(u) + \e \pi D^2 \ln R_\alpha  \RC}$ then the right-hand side is minimized and we obtain
\eqref{gstab.c1}. 

4. Next we note that 
since $(1-a)^4 \le (1-a^2)^2$ for $a>0$, 
\[
\int_{\R^2} (1-|u|)^4 \ dx \lesssim  \e^2\int_{\R^2} e_\e(|u|) \ dx \lesssim \e^2 \int_{B_{R_\alpha}}
 e_\e(u) \ dx \ + \ \e^2 \int_{\R^2_\sigma} e_\e(|u|) \ dx \ \le \ { \e^2\calP, }
\]
where we have used  \eqref{energyupperboundBigBall}, \eqref{gstab.c1}
for the last inequality. This is \eqref{moduL4L2}.

5.  Finally, we establish the bound \eqref{Lpcurrentboundsbigball} on the current, modifying
the proof from \cite{JSp2}. To simplify notation we will write
$\chi_{1} :=  \chi_{\cup_j B_{\sigma_\star}(\alpha_j)}$
and $\chi_1 = 1-\chi_1 = \chi_{\R^2_{\sigma_\star}}$.
We decompose $j(u)-j(u_*)$ differently in $\cup_j B_{\sigma_\star}(\alpha_j)$ and
$\R^2_{\sigma_\star}$, as follows:
\begin{align*}
j(u)-j(u_*) & = \chi_1\left[ \frac {j(u)}{|u|}(|u|-1) + \frac{j(u)}{|u|}  - j(u_*)\right]  \\
& \quad +
 \chi_2 \left[ (\frac {j(u)}{|u|}- j(u_*)) + (\frac {j(u)}{|u|}- j(u_*))(|u|-1) + j(u_*)(|u|-1)\right].
\end{align*}
We now proceed to bound all the summands in either $L^{4/3}$ or $L^2$.
For the terms supported on $\cup B_{\sigma_\star}(\alpha_j)$ we estimate
\begin{align*}
\left\|\chi_1\frac {j(u)}{|u|}(|u|-1)
\right\|_{L^{4/3}} 
&\le \| \nabla u\|_{L^2(B_{R_\alpha})} \| 1-|u| \,\|_{L^4(B_{R_\alpha})} \ \le \e^{1/2}\calP, \\
\left\|\chi_1\frac {j(u)}{|u|}
\right\|_{L^{4/3}} 
&\le \| \frac {j(u)}{|u|}\|_{L^2(B_{R_\alpha})}\  \| \chi_1\|_{L^4} \le   \sigma_\star^{1/2}  \calP,\\
\left\|\chi_1\ j(u_*)
\right\|_{L^{4/3}}  & \leq C \LC n \int_0^{\sigma_\star} \LC { n \over r} \RC^{4/3} r dr \RC^{3/4} 
\le   \sigma_\star^{1/2} \calP,
\end{align*}
where the last estimate uses the explicit form of $j(u_*)$.
Outside of $\cup B_{\sigma_\star}(\alpha_j)$, 
\begin{align*}
\| \chi_2 (\frac {j(u)} {|u|}- j(u_*))\|_{L^2}
&\le |\Sigma|^{1/2} + s_\e^{1/4}\calP,
\\
\| \chi_2 (\frac {j(u)}{|u|}- j(u_*))(|u|-1) \|_{L^{4/3}}
&\le \| \chi_2 (\frac {j(u)}{|u|}- j(u_*))\|_{L^2}\  \| 1-|u|\, \|_{L^4} \le \e^{1/2}\calP,
\\
\| \chi_2  j(u_*)(|u|-1) \|_{L^{4/3}}
&\le
\| \chi_2 j(u_*) \|_{L^4} \ \| (1-|u|^2)\|_{L^2} \ \le \e \calP.
\end{align*}
Since $\e \le \sigma_\star \le s_\e^{1/2}\calP$, this implies 
 \eqref{Lpcurrentboundsbigball}.

\end{proof}


\section{Bounds on the localization of the Jacobian}
\label{sec.loc}

The second main element of the proof of Theorem \ref{thmdynamics}, also adapted from \cite{JSp2},
is

\begin{proposition}  \label{P2}
{There exists a positive constant $c_3 \le c_2$ such that 
for any $u\in  [U^D] + H^1(\R^2; \C)$ with $D\ne 0$, if there exist 
distinct points
$\alpha_1, \ldots, \alpha_n  \in \R^2$
and $d_1,\ldots d_n \in \{\pm 1\}$
such that 
\begin{equation}  
  \left\| J(u) - \sum_{j=1}^n \pi d_j \delta_{\alpha_j} \right\|_{X^*_{\ln}} 
\  \le 
\ \frac{ c_3\rho_\alpha }{ n^6 \ln R_\alpha} ,  
\label{p2.h1}\end{equation}
then there exist $\xi_1,\ldots, \xi_n\in\R^2$ such that 
$|\xi_i - \alpha_i|\le   \frac{ \rho_\alpha}{80 n^4 }$ for all $i$, and 
\begin{equation}\label{p2.c1}
\left\| J(u) - \pi \sum _{i=1}^n d_i\delta_{\xi_i} \right\|_{X^*_{\ln}}  
\ \le \  \e\left[n (C + \Sigma  )^2
e^{C\Sigma } +  \calP \right]
\end{equation} 
where
$\Sigma =
\Sigma(u; \alpha, d) =\E_\e(u) - W_\e(\alpha,d)$ and the notation
$\calP$ is introduced in \eqref{P.def}.}
\end{proposition}

The proof of Proposition~\ref{P2} follows arguments
developed 
for a bounded domain in \cite{JSp2},
with modifications
due for example to the differences between the 
weak topologies $X_{\ln}^*$ and $\dot{W}^{-1,1}(\Omega)$ and the fact
that we control only $\E_\e(u)$, rather than $\int e_\e(u) dx$.

{Assumption \eqref{p2.h1} is a little stronger than  the corresponding assumption in \eqref{gstab.h1} in Proposition~\ref{Propgammastability} due to requirements that arise in  step 1 of the proof of  Proposition~\ref{P2}.
}

An essential tool in establishing Proposition \ref{P2} is the modified Jacobian, $J'(u)$, of Alberti-Baldo-Orlandi \cite{ABO}. A convenient definition may be given by setting
{
\begin{equation}\label{auxg.def}
J'(u) := J(u'),\qquad
u' := g(|u|) u, \qquad 
\end{equation}
for some fixed smooth $g:\R\to \R$   such that
$g(s) = \frac 1s$ if $s\ge \frac 12$.
Then $|u'|=1$ whenever $|u|\ge 1/2$,} 
from which one can deduce that
$\operatorname{supp}(J'(u)) \subset \{ |u| < 1/2\}$.
Thus $J'(u)$ is more concentrated than $J(u)$.
One can also check that
\begin{equation}  \label{quantizationmodifiedjac}
\deg(u, \p V) = {1\over \pi} \int_V J'(u) \ dx \hbox{ if } |u| \geq {1\over2} \hbox{ on } \p V .
\end{equation}
We first show that $J'(u)$ is close to $J(u)$ in the $X_{\ln}^*$ topology, provided that the vortices are not too
far from the origin.   


{\begin{lemma}
If  $u$ satisfies the hypotheses of Proposition \ref{P2}, then 
\begin{equation}  \label{diffjacmodjacbound}
\LN J'(u) - J(u) \RN_{X^*_{\ln}} \lesssim  \e  \LC \E_\e(u) + \pi  D^2 \ln R_\alpha + \Sigma +C +{n^2} \RC
\le \e \calP .
\end{equation}
\label{L.JJprime}\end{lemma}}

\begin{proof}
{For any smooth compactly supported $\phi$ such that $ \| \phi\|_{X_{\ln}^*}
\le 1$, }
\begin{equation} \begin{split}
\left|\int\phi (J(u) - J'(u)) \ dx \right| & =  \left|\int (\nabla\times \phi)\cdot (j(u) - j(u')) \ dx \right| \\
&{\le 
C  \int_{\R^2}\LV \nabla \phi   \RV \LV  {j(u)}  \RV \LV g^2(|u|) - 1 \RV \ dx}\\
& \le 
C  \int_{\R^2}\LV \nabla \phi   \RV \LV  {j(u) \over |u|}  \RV \LV |u|^2 - 1 \RV \ dx,
\end{split} \label{modJac1}\end{equation}
{since $| g^2(s)-1| \le \frac C s |s-1| \le  \frac C s |s^2-1|$ for all $s\ge 0$.
Because $\|\nabla\phi\|_\infty \le \| \phi\|_{X_{\ln}^*}$, it follows that }
\begin{align*}
\LV \int_{\R^2} \phi \LC J(u) - J(u') \RC \ dx\RV
&\overset{\eqref{modJac1}} \leq C \e \int_{B_{R_\alpha}} e_\e( u ) \ dx
\ + \  C  \int_{\R^2\backslash B_{R_\alpha}} \LV \nabla \phi   \RV  { \LV j(u) \RV \over |u|}  \LV |u|^2 - 1 \RV
\ dx \\
& \leq C \e \int_{B_{R_\alpha}} e_\e( u ) \ dx 
\  + \  C  \int_{\R^2\backslash B_{R_\alpha}}   \LV {  j(u)  \over |u|} - j(u_*) \RV  \LV |u|^2 - 1 \RV \ dx \\
& \qquad\qquad + C \int_{\R^2\backslash B_{R_\alpha}} \LV \nabla \phi   \RV  \LV j(u_*) \RV  \LV |u|^2 - 1 \RV \ dx.
\end{align*}
By \eqref{energyupperboundBigBall}, 
\begin{equation}  \label{Jaccomparelocalenergy}
\e \int_{B_{R_\alpha}} e_\e( u )  \ dx
\lesssim \e \LC \E_\e(u) + \pi D^2 \ln R_\alpha  \RC.
\end{equation}

Next, 
\begin{align}
\label{Jaccomparefarenergy1}
\int_{\R^2 \backslash B_{R_\alpha}}   \LV {  j(u)  \over |u|} - j(u_*) \RV  \LV |u|^2 - 1 \RV \ dx
& \lesssim {\e }  \int_{\R^2 \backslash B_{R_\alpha}} e_\e(|u|) 
+ \LV {j(u) \over |u|} - j(u_*) \RV^2 \ dx \\
\nonumber
& \stackrel{\eqref{gstab.c1}}{\lesssim} { \e } \LB \Sigma(u;\alpha,d) + C \RB.
\end{align}
Furthermore, by \eqref{Ggradientestimateerror} we have $| j(u_*)(x) | \lesssim
{ n \over |x|}$ for all $|x| \geq  R_\alpha$, thus
\begin{align}\label{Jaccomparefarenergy2}
\int_{\R^2 \backslash B_{R_\alpha}} \LV \nabla \phi   \RV  \LV j(u_*) \RV  \LV |u|^2 - 1 \RV \ dx
& \lesssim \int_{\R^2 \backslash B_{R_\alpha}} {1 \over \ln |x|}  {n \over |x|}    \LV |u|^2 - 1 \RV \ dx \\
\nonumber
& \lesssim \e n  \LC \int _{R_\alpha}^\infty { dr \over (\ln r)^2 r} \RC^{1\over2 } \LC \int_{\R^2 \backslash B_{R_\alpha}} e_\e(|u|) \ dx\RC^{1\over 2} \\
\nonumber
&{ \lesssim    {\e n \over \sqrt{\ln R_\alpha} }\LB \Sigma(u;\alpha,d) + C \RB^{1/2}.
}\end{align}
Combining \eqref{Jaccomparelocalenergy}-\eqref{Jaccomparefarenergy2} yields
 \eqref{diffjacmodjacbound}. 
\end{proof}

The strategy of proving the Proposition \ref{P2} is to combine a precise localization of a single vortex, Lemma \ref{localizationsinglevortexlemma} below, with a global Jacobian bound away from the vortices in Lemma \ref{farfieldJacobiancontrol}.

\begin{lemma}[\cite{JSp}, Theorem 1.2']  \label{localizationsinglevortexlemma}
There exists an absolute constant $C$ such that {the following holds}
for any $0 < \e \leq \tau$ and any $u \in H^1(B_{\tau}; \C)$ satisfying
\[
\LN J(u) - \pi d \delta_0 \RN_{\dot{W}^{-1,1}(B_{\tau})} \leq {\tau \over 4} \hbox{ with } d = \pm 1 .
\]
If we write $\Sigma_{B_{\tau}} = \int_{B_{\tau}} e_\e(u) - \LC \pi \ln {{\tau} \over \e} + \gamma \RC$ and 
\begin{equation}
\ell_\e(B_\tau) = \ell_\e := \e C \LC C + \Sigma_{B_{\tau}} \RC e^{\Sigma_{B_{\tau}}},
\label{elleps.def1}\end{equation}
then there exists a point $\xi \in B_{{\tau}/2}$ such that for any $\sigma < {\tau} - \ell_\e$,
\begin{equation}
\LV \{ s \in [\sigma, {\tau}] \hbox{ such that } u \hbox{ satisfies } \eqref{Lipstarestimatesinglevortex} 
 \hbox{ on } B_s \} \RV \geq \tau- \sigma - \ell_\e
\end{equation}
where 
\begin{equation} \label{Lipstarestimatesinglevortex}
|u| > \frac 12 \mbox{ on }\partial B_s(\xi), \qquad \mbox{ and }\quad
\LN J'(u) - \pi d \delta_\xi \RN_{Lip^*(B_s)} \leq \ell_\e \LC C + \Sigma_{B_{\tau}}\RC.
\end{equation}
\end{lemma}

The conclusion $|u| > \frac 12$ on $\partial B_s$ does not appear in
the statement of Theorem 1.2' 
in \cite{JSp}, but is established in the proof \footnote{The proof of  \cite[Theorem 1.2']{JSp} uses 
conclusions (4.7) and (4.9) of
\cite[Proposition 4.2]{JSp}  to show that the set 
$\{ s \in [0, {\tau}]  : \inf_{\partial B_s} |u| \le \frac 12\}$
has measure at most  $ \ell_\e$,
and then proves that $\LN J'(u) - \pi d \delta_\xi \RN_{Lip^*(B_s)} \leq \ell_\e \LC C + \Sigma_{B_\tau}\RC$
for every $s$ such that $\inf_{\partial B_s}|u|> \frac 12$.}.

Here and below, we use the notation
\[
\| \mu \|_{Lip^*(\Omega)} := \sup \{ \int_\Omega \phi \ d\mu : \| \nabla \phi \|_\infty \le 1\}.
\]
Thus, in computing the $Lip^*$ norm, we allow test functions that do not vanish on 
$\partial \Omega$.

Our next result, also taken from \cite{JSp}, will be applied to the function
$w = u/u_*$.


\begin{lemma}\label{farfieldJacobiancontrol}
There exists an absolute constant $C$ such that for any $w \in H^1(\R^2 ; \C)$ and any $\e \in (0,1]$,
if
\[
\int_{\R^2_\sigma} e_\e(w) \ dx \le \ \frac \pi 2 \ln \frac 1\e
\]
and if 
\begin{equation} \label{cutoutest2}
|w| > {1\over2} \hbox{ on } \cup \p B_\sigma(\alpha_j) = \partial \R^2_\sigma
\end{equation}
then for any $\phi \in C^1_c(\R^2)$,
 we have the inequality
\begin{equation} \label{cutoutest1}
\left| \int_{\R^2_\sigma} \phi J'(w) \ dx \right|
\  
\leq  \ \e C \| \nabla \phi \|_{L^\infty}  \LC \int_{\R^2_\sigma} e_\e(w)  \ dx \RC \exp \LB {1\over \pi  } \int_{\R^2_\sigma} e_\e(w) \ dx \RB.
\end{equation}
\end{lemma}

The lemma is a special case\footnote{obtained by taking $n=0, M=\pi/2$
and $\Gamma = \emptyset$
in  Theorem 1.1', using notation from \cite{JSp}.
The conclusion as stated in \cite{JSp}  is somewhat more complicated, asserting that there exists 
a large set of radii $s$ satisfying both \eqref{cutoutest2} and \eqref{cutoutest1},
but the proof in fact shows that  there is a large set of radii satisfying \eqref{cutoutest2}, and that \eqref{cutoutest2} implies \eqref{cutoutest1}.} of  Theorem 1.1' in \cite{JSp}.
The hypothesis of finite energy is shown in \cite{JSp} to eliminate any possible difficulties
arising from the unboundedness of the domain.

\begin{proof}[ Proof of Proposition~\ref{P2}]
We  follow the argument of the proof of Theorem 3 in \cite{JSp2} with the adjustments due
 to the unbounded domain.  

{Observe  that for any specific term $\calP$ of the form \eqref{P.def}, we may assume that 
\begin{equation}
 \e \calP\  \le \   1
\label{simp}\end{equation}
as otherwise \eqref{p2.c1} follows immediately from
\eqref{p2.h1}, with $\xi_i=\alpha_i$ for all  $i$ (after possibly increasing
the choice of $\calP$ in \eqref{p2.c1}). Similarly, we may assume that
\begin{equation}
 \e \frac{\sqrt{\ln(\rho_\alpha/\e)}} {\ln R_\alpha}\ \le \ 
  \Big\| J(u) - \sum_{j=1}^n \pi d_j \delta_{\alpha_j} \big\|_{X^*_{\ln}} 
\label{sep.lowerbd}\end{equation}
as otherwise there is nothing to prove. It follows
from
\eqref{p2.h1} and \eqref{sep.lowerbd} 
that
\begin{equation} \label{restrict.ep.loc} 
\e { n^6  \over \rho_\alpha} \leq {c_3}.
\end{equation} 
and hence that the hypotheses of Proposition \ref{Propannularlower}
are satisfied. It also follows from \eqref{p2.h1} and \eqref{sep.lowerbd} 
that the hypotheses \eqref{gstab.h1} and \eqref{sigmastardefinition}
of Proposition \ref{Propgammastability} hold,
with $ \| J(u) - \sum_{j=1}^n \pi d_j \delta_{\alpha_j}\|_{X^*_{\ln}} $ playing the role of $s_\e$. The verification of
\eqref{sigmastardefinition} uses \eqref{simp}. }

Based on \eqref{p2.h1} we define a family of radii
\begin{equation} \label{sigmachoices}
\sigma \in \LB {3 \over 4} \sigma_1, \sigma_1 \RB \hbox{ with } \sigma_1 = { {\rho_\alpha \over 40 n^4 }}.
\end{equation}
Note that $\sigma^{-1} \le \calP$.  
{Again owing to \eqref{p2.h1}, any $\sigma$ satisfying \eqref{sigmachoices} must also satisfy 
conditions \eqref{sig.constraints} (again with $s_\e = \| J(u) - \sum_{j=1}^n \pi d_j \delta_{\alpha_j}\|_{X^*_{\ln}} $)
from the proof of Theorem \ref{Propgammastability}. Thus
all estimates concerning balls $B_\sigma$
and their complements from the proof of that result are valid here.
}

We now begin to localize the vortex found inside of each vortex ball $B_\sigma (\alpha_j)$.  

1.  Single vortex localization.  First by \eqref{convertnorm}, \eqref{p2.h1}, and \eqref{sigmachoices},
{ \begin{align*}
\LN J(u) - \pi d_j \delta_{\alpha_j} \RN_{\dot{W}^{-1,1}(B_\sigma)} 
& \leq 
\LC 1 + (\ln (|\alpha_j| + \sigma))^+\RC \| J(u) - \sum_{j=1}^n \pi d_j \delta_{\alpha_j} \|_{X^*_{\ln}}\\
&  \leq 
2\ln R_\alpha {  c_3 \rho_\alpha \over n^6 \ln R_\alpha } \leq { \sigma \over 4 n^2} \le \frac \sigma 4,
\end{align*}
if $c_3$ is chosen small enough.
Thus by Lemma \ref{Singlevortexenergylemma} we have
\begin{equation}  \label{Sigmajlowerbound}
\Sigma_{B_\sigma(\alpha_j)} \geq - C{ \e \over \sigma} \sqrt{ \ln { \sigma \over \e}} - {C \over \sigma} \LN J(u) - \pi d_j \delta_{\alpha_j} \RN_{\dot{W}^{-1,1}(B_{\sigma})} \geq  - {C \over n},
\end{equation}
using the fact that $\frac \e \sigma \le Cn^{-2}$, which follows from  \eqref{sigmachoices} and \eqref{restrict.ep.loc}.}

We begin to bound the excess energy in small balls in order to provide a good localization for each single vortex ball $B_r (\alpha_j)$.  In particular, given the $\sigma$ from \eqref{sigmachoices} and 
using arguments from the proof of Proposition~\ref{Propgammastability}, we obtain
\begin{align}
\Sigma := \Sigma(u; \alpha, d) & 
\stackrel{\eqref{EEEa}}{\geq} \int_{\R^2_\sigma } [e_\e(u) - e_\e(u_*)] \ dx+ \sum_{j=1}^n \Sigma_{B_\sigma(\alpha_j)} - C {n^3 \sigma^2\over \rho_\alpha^2 }.
\end{align}
Next, {\eqref{fixedradiusgammastability} (with  $s_\e =  \| J(u) - \sum_{j=1}^n \pi d_j \delta_{\alpha_j}\|_{X^*_{\ln}} $)
and \eqref{p2.h1}}
imply that
\begin{align}
\int_{\R^2_\sigma } e_\e(u) - e_\e(u_*) \ dx
& 
\geq  {1\over 2} \int_{\R^2_\sigma} e_\e(|u|) + {1\over4} \LV {j(u) \over |u|} - j(u_*) \RV^2 \ dx   + \nonumber \\
& \quad - { C n \over \sigma} \LC {c_3\rho_\alpha \over  n^6 \ln R_\alpha} \ln R_\alpha + \e \calP\RC
  - C{n^4 \sigma \over \rho_\alpha } - {\e \calP} . 
  \nonumber 
\end{align}
From the last two inequalities, recalling \eqref{simp} and the choice \eqref{sigmachoices} of $\sigma$, we
obtain
\begin{align}
\Sigma  &  
\ge  {1\over 2} \int_{\R^2_\sigma} e_\e(|u|) + {1\over4} \LV {j(u) \over |u|} - j(u_*) \RV^2 \ dx   + \sum_{j=1}^n \Sigma_{B_\sigma(\alpha_j)} - C , \nonumber
\end{align}
and we infer that
\begin{equation} \label{excesseenergyboundsinglevortex}
\Sigma_{B_\sigma(\alpha_j)} \leq \Sigma + C \hbox{ for } j = \{1, \ldots, n\}
\end{equation} 
and also, using  \eqref{Sigmajlowerbound}, that
\begin{equation} \label{singlevortexenergycontrolbyglobalenergy}
{1\over2} \int_{\R^2_\sigma} e_\e(|u|) + {1\over4} \LV {j(u) \over |u|} - j(u_*) \RV^2 \ dx \leq \Sigma + C.
\end{equation}
It follows from \eqref{excesseenergyboundsinglevortex}
that
\begin{equation}\label{definitionellepsilon}
\ell_\e(B_\sigma(\alpha_i))  \ \le  \  \e C \LC C + \Sigma  \RC e^{(C+\Sigma )} := \ell_\e  \qquad
\mbox{ for all $i$ and all  } \sigma\in [\frac 34 \sigma_1, \sigma_1]
\end{equation}
using notation from \eqref{elleps.def1}. (The constants in the definition
of $\ell_\e$ will be adjusted later to satisfy some additional estimates.)
Then Lemma \ref{localizationsinglevortexlemma} (with $\tau = \sigma_1$ and $\sigma = \frac 34 \sigma_1$) and  \eqref{excesseenergyboundsinglevortex}
imply that there exists $\xi_i \in B_{\sigma_1 / 2} (\alpha_i)$ such that 
\[
\LV \{ \sigma \in [ {3\over 4} \sigma_1, \sigma_1] \hbox{ such that } u \hbox{ does not satisfy } \eqref{Lipstarestimatesinglevortex2}
\} \RV \leq \ell_\e
\]
where
\begin{equation} \label{Lipstarestimatesinglevortex2}
|u|> \frac 12\mbox{ on }\partial B_\sigma(\alpha_i), \qquad
\LN J'(u) - \pi d \delta_{\xi_j} \RN_{Lip^*(B_\sigma(\alpha_i))} \leq \ell_\e \LC C + \Sigma \RC.
\end{equation}
We may assume that
 $\ell_\e \le \frac {\sigma_1}{4n}$, as otherwise \eqref{p2.c1} follows from
 \eqref{p2.h1} and \eqref{sigmachoices}, and it follows that we can fix some $\sigma\in [\frac 34 \sigma_1, \sigma_1]$ such that \eqref{Lipstarestimatesinglevortex2}
holds for every  $i$. We henceforth assume that this has been done.

2.  We now look to prove a Jacobian estimate in $\R^2_\sigma(\alpha)$ by {bounding the energy}  of $u/u_*$.  Let $w = u/u_*(\alpha, d)$.  
Our starting point is the identity
\begin{align*}
 e_\e(w) 
& = e_\e(|u|)   + {1\over2} \LV { j(u) \over |u| } - j(u_{ *}) \RV^2  \\
& \quad +  { j(u) \over|u|} \cdot j(u_*) \LC 1 -  |u|  \RC + {1\over2} \LV j(u_*) \RV^2 \LC  |u|^2 - 1  \RC
\end{align*}
from \cite{JSp2}, which is not hard to verify.
Using the explicit form of $j(u_*)$, 
we estimate
\[
\left|
\int_{\R^2_\sigma}  \LV j(u_*) \RV^2 \LC  |u|^2 - 1  \RC \ dx
\right|\le \e \left( \int_{\R^2_\sigma}e_\e(|u|) \ dx \ \int_{\R^2_\sigma} |j(u_*)|^4 \ dx \right)^{1\over 2} \le \e  C(C+\Sigma)^{1\over 2}\frac{n^2}{\sigma} \le  \e\calP.
\]
Also, it follows from \eqref{A1estimatebound} that
\[
\left |\int_{\R^2_\sigma} \frac{j(u)}{|u|} \cdot j(u_*)(1-|u|) \ dx  \right| \ \le \ 
\frac 18  \int_{\R^2_\sigma} e_\e(|u|)+ \LV {j(u) \over |u|} - j(u_*) \RV^2 dx
+ \e\calP.
\]
{Since $\sigma_\star \le \frac 34 \sigma_1 \le \sigma$ (this is a consequence of
\eqref{p2.h1}, \eqref{sep.lowerbd},  and \eqref{simp})
we see from \eqref{gstab.c1} and \eqref{simp} that
\begin{equation} \label{Globalenergyboundw}
\int_{\R^2 _\sigma(\alpha)} e_\e(w) \ dx
\leq 
\int_{\R^2 _{\sigma_\star}(\alpha)} e_\e(w) \ dx
\le
C(C + \Sigma).
\end{equation}}
It follows from \eqref{Lipstarestimatesinglevortex2} that
\begin{equation}  \label{moduluswbound}
|w| >  {1\over2} \hbox{ on } \p\R^2_\sigma .
\end{equation}
Thus
from \eqref{Globalenergyboundw} and Lemma \ref{farfieldJacobiancontrol},
after adjusting the constants in the definition \eqref{definitionellepsilon},
we can conclude that if $\phi\in C^1_c(\R^2)$  and $\|\phi\|_{X_{\ln}}\le 1$
then (since $\|\phi\|_{Lip} \le \|\phi\|_{X_{\ln}}$)
\begin{equation}  \label{jacobianboundw}
\int_{\R^2_\sigma} \phi J'(w) \ dx
\ \leq C \ell_\e \LC C + \Sigma \RC.
\end{equation}

3.  Next we show that $J'(w)$ is close to $J'(u)$ on $\R^2_\sigma$. It is straightforward to check (or see \cite{JSp2}) that 
\[
J'(u) - J'(w)  = {1\over2} \nabla \times [ (|w'|^2 - 1) j(u_*)]
\]
for $w' = g(|w|) w$, where $g$ was defined in \eqref{auxg.def}.
By construction of $g(s)$, we know that $J'(u)$ and $J'(w)$ vanish on the set $\mathcal{S} = \{ x \in \R^2 :  |u(x)|\ge  {1/2} \}$,
and hence on $\partial \R^2_\sigma$ for $\sigma$ satisfying \eqref{moduluswbound}.  For 
$\phi \in C^1_c(\R^2)$ such that
$\| \phi \|_{X_{\ln}}\le 1$ we then
have 
\begin{align*}
\int_{\R^2_\sigma} \phi \LC J'(u) - J'(w) \RC \ dx 
& = \frac 12 \int_{\R^2_\sigma}  \nabla^\perp \phi \cdot \LB (|w'|^2 - 1 ) j(u_*) \RB \ dx \\
& \leq  C\int_{\R^2_\sigma} \LV \nabla \phi \RV  \LV |w|^2 - 1 \RV \LV j(u_*) \RV \ dx \\
& \leq \e C \LN {1 - |w|^2 \over \e} \RN_{L^2(\R^2_\sigma)}  \LC \LN j(u_*) \RN_{L^2(B_{R_\alpha,\sigma})}
+ \LN {j(u_*) \over \ln |x| } \RN_{L^2(\R^2 \backslash B_{R_\alpha})} \RC \\ 
& \stackrel{\eqref{Ggradientestimateerror1}, \eqref{Globalenergyboundw}}{\leq} C \e {n R_\alpha \over \sigma} \LC C + \Sigma \RC \\
& \stackrel{\eqref{sigmachoices}}{\lesssim} \e {n^5 R_\alpha \ln R_\alpha \over \rho_\alpha} (C + \Sigma).
\end{align*}
For such $\phi$ we combine this with \eqref{jacobianboundw}
and \eqref{Lipstarestimatesinglevortex2} and recall the definition 
\eqref{definitionellepsilon} of $\ell_\e$ to find that
\begin{align*}
&\int_{\R^2} \phi \LC J'(u) - \sum_{j=1}^n \pi d_j \delta_{\xi_j} \RC \ dx \\
&\hspace{4em} = \int_{\R^2_\sigma} \phi  (J'(u) - J'(w)) \ dx + \int_{\R^2_\sigma} \phi  J'(w) \ dx + \sum_{j=1}^n \int_{B_s(\alpha_j)} \phi \LC J'(u) - \pi d_j \delta_{\xi_j} \RC \ dx \\
&\hspace{4em} {\lesssim} \ 
  \e \LB n C (C + \Sigma)^2 e^{\Sigma} 
+C  {n^5 R_\alpha \ln R_\alpha \over \rho_\alpha} (C + \Sigma) \RB.
\end{align*}
Taking the supremum over all such $\phi$, we find that $\LN J'(u) - \sum_{j=1}^n \pi d_j \delta_{\xi_j}  \RN_{X_{\ln}^*} $ is bounded by the
right-hand side.
Finally, using \eqref{diffjacmodjacbound} we have $\LN J(u) - J'(u) \RN_{X_{\ln}^*} \lesssim \e (\E_\e(u) + \pi D^2 \ln R_\alpha)$ and the bound \eqref{p2.c1} follows by the triangle inequality.

\end{proof}


\section{ vortex dynamics in infinite-energy configurations on $\R^2$} \label{sec.gronwall}


In this section we carry out the first part of the proof of \eqref{thmdynamics},
in which we reduce the theorem to controlling the rate of
growth of a scalar quantity that we call $\eta(t)$, defined in \eqref{eta.def}.
In the subsequent two sections we compute and bound
$\frac d{dt}\eta$ and  $\frac d{dt}\LA \eta\RA_\delta$ for a suitable $\delta$,
where 
\begin{equation}
\LA f\RA_\delta(t) = \frac 1\delta \int_{t-\delta}^t f(s)ds.
\label{avg.def}\end{equation}
The proof of the theorem is finally completed in Section \ref{S.conc}
by applying Gr\"onwall's inequality
to $\LA \eta \RA_\delta$ and using a preliminary, weaker estimate
of $\frac d{dt}\eta$ to deduce pointwise bounds on $\eta$.

We recall that under the assumptions of the theorem,
and using conservation of energy for the PDE
\eqref{nls} and the ODE \eqref{PV2},
quantities such as $\E_\e(u(t)), W(a(t)), W_\e(a(t)), R_{a(t)}$,
and  $\rho_{a(t)}^{-1}$ are bounded uniformly, for all $t$,
as discussed in Section \ref{ss.noworries}.

\subsection{True vortex positions at $\{\xi_j\}_{j=1}^n$. }

By conservation of the energy $W(\cdot)$, and hence $W_\e(\cdot)$, for the ODE \eqref{PV2} and the conservation \eqref{consEr} of   the energy renormalized at infinity $\E_\e$
for the Bethuel-Smets solutions of \eqref{nls}, we deduce from
\eqref{t1.h2}  that
\begin{equation}
{
\Sigma(u(t); a(t))  = \Sigma(u(0); a(0))
}
 \le  \e^{1\over 2}
\quad\mbox{ for all $t$}.
\label{cons.Sig}\end{equation}
We define
\begin{align}
\tau_1& := \sup \left\{ \tau>0: \| J(u(t)) - \pi \sum  \delta_{a_i(t)}\|_{X_{\ln}^*} \le
\e^{1/3}  \mbox{ for all }0 \le t \le \tau \right\}.
\label{tau1.def}
\end{align}
Taking $\e_0$ to be suitably small, 
see the discussion in Section \ref{ss.noworries},
it follows that the hypotheses of
Proposition \ref{P2} are satisfied by $u(t), a(t)$ for
all $t\in [0, \tau_1]$.
Therefore, for such $t$ there exist distinct points
$\xi_1(t),\ldots, \xi_n(t)$ in $\R^2$ such  that 
\begin{equation}
\mbox{{$|\xi_i - a_i| \le \frac{ \rho_{a(t) }}{80 n^4 }$} for all $i$, \qquad and 
}\quad  \| J(u (t))- \sum_{i=1}^n\pi  \delta_{\xi_i(t)}\|_{X_{\ln}^*}
\le  \e\calP.
\label{xis.def}\end{equation} 
Since $t\mapsto J(u(t))$ is continuous as a map from $\R$ into $X_{\ln}^*$, see \eqref{XcJ}, we can
choose the points $\xi_i(t)$ to depend continuously, hence measurably, on $t$. (Note that \eqref{xis.def} only constrains
 $\xi_i(t)$ up to length scales of order $\e\calP$.)
 
Note that $|\rho_{\xi(s)}  - \rho_{a(s)}| \le  \frac 1{10}\rho_{a(s)}$ for all $s$,
so that we may freely replace $\rho_{\xi(s)}$ by $\rho_{a(s)}$  in all our estimates 
at the expense of adjusting some constants.

\subsection{Controlling $\eta(t)$ instead of $\sum_{j=1}^n \LV \xi_j(t) - a_j(t) \RV$.}
\label{subsectionetaintro}

We next introduce a quantity $\eta$, and we prove that control over $\eta$
will suffice to establish the main conclusions of the Theorem.
We define
\begin{equation}
\eta (t ):= \sum_{j=1}^n |\eta_j(t)|
 := \sum_{j=1}^n 
\LV \int J(u) \, \Phi_j(x,t) dx \RV 
\label{eta.def}\end{equation}
for
\[
\Phi_j(x,t) := 
\varphi(x-a_j(t), \rho_{a(t)}), \qquad
\mbox{ where }\varphi(x, \rho) := x \,\chi(\frac x \rho)
\]
and  $\chi\in C^\infty_0(\R^2)$ is a fixed function such that
\[
\chi(x) = \left\{ \begin{array}{ll} 
1 & \hbox{for } |x| \leq 1 \\
0 & \hbox{for } |x|\geq 2. \end{array} \right.
\]

Since $\varphi(x,\rho) = \rho\, \varphi(x/\rho, 1)$, it is clear that $\| \nabla \Phi_j(\cdot, t)\|_\infty$ is independent of $t$ and $j$. It follows directly from the definition of $\rho_a$ that
$\left\{ \operatorname{supp} \Phi_j(\cdot ,t) \right\}$ are pairwise 
disjoint.

\begin{lemma}  If $0\le t \le \tau_1$ then 
\begin{align} \label{eta.rocks1} 
\LV \eta(t) -  \ \pi \sum_i|\xi_i(t)-a_i(t)| \RV & \le \e \calP,
\\
{\eta(t) 
\le \calP  \| J(u(t)) - \pi \sum  \delta_{a_i(t)}\|_{X_{\ln}^*} 
}
&\le  \e^{1/3} \calP .  
\label{eta.bound}\end{align} 
\end{lemma}

\begin{proof} 
The definition of $\Phi_j$ implies that
\[
\eta_j(t) - \pi(\xi_j(t) - a_j(t)) = \int (J(u) - \pi\sum \delta_{\xi_i}) \Phi_j(t,x)\ dx
\]
for every $j$. It is also easy to see that $\| \Phi_j \|_{X_{\ln}^*}\le C \ln R_{a(t)} \le \calP$,
so we deduce that
\begin{equation}
\big| \, |\eta_j(t)| - \pi |\xi_j(t) - a_j(t)|\, \big|  \ \le \  |\eta_j(t) -  \pi (\xi_j(t) - a_j(t))| \ \overset{\eqref{xis.def}}\le \e \calP
\label{futureref}\end{equation}
for every $j$,
and  then it is straightforward to obtain \eqref{eta.rocks1}.
Similarly, 
\[
\eta (t ) = \sum_{j=1}^n 
\LV \int  ( J(u) - \pi \sum \delta_{a_i}) \, \Phi_j(x,t) dx \RV 
{\le \calP  \| J(u(t)) - \pi \sum  \delta_{a_i(t)}\|_{X_{\ln}^* } }
\overset{\eqref{tau1.def}}\le \e^{1/3} \calP .
\]
\end{proof}


\subsection{Approximation of $u$ by canonical harmonic maps $u_*(\xi)$.}

We now state an important lemma that estimates how close $u(t)$ is to the canonical harmonic map $u_*(\xi(t))$ for times $t$ within $[0,\tau_1]$, defined in section~\ref{subsectionetaintro},
for  $\xi_i(t)$ chosen in \eqref{xis.def}.

\begin{lemma}
For all $0 \leq t \leq \tau_1$ we have
\begin{align}
\Sigma(u(t); \xi(t))
& \lesssim  \frac{n}{\rho_{a(t)}} \eta(t) + \e^{1\over 2},  \label{Sigma.est1}  \\
 \int_{\R^2 \setminus \cup B_{\rho_{a(t)}}(a(t))}   e_\e(|u(t)|)  + 
\frac 14\LV { j (u(t)) \over |u(t)|} - j( u_*(\xi(t)) )\RV^2 dx 
&  \lesssim \frac  {n}{\rho_{a(t)}} \eta(t) + \e^{\frac 12}\calP ,  \label{gstab.ref1}\\
\| j(u)(t) - j(u_*(\xi))(t) \|_{L^{4\over3} +L^2(\R^2)}
&  \lesssim  \left( \frac{n  }{\rho_{a(t)}} \eta(t) \right)^{1\over2} + \e^{\frac 14}\calP, 
\label{gstab.ref2}\\
\| 1-|u|\|_{L^4(\R^2)}^4 \le 
\| 1-|u|^2\|_{L^2(\R^2)}^2 &\le
{\e^2\calP.}
\label{gstab.ref3}
\end{align}
\label{uptotau1}\end{lemma}

\begin{proof}
Fix $t\in [0, \tau_1]$.
{From  the definition of $\Sigma(u; \cdot)$ and \eqref{cons.Sig}
we have that }
\begin{align}
&{
\Sigma (u(t); \xi(t))
= 
\Sigma(u(t); a(t)) + W(a(t)) - W(\xi(t))}
\nonumber \\
&\qquad \le  \e^{1 \over 2} +
\left(\sum_{j=1}^n |\xi_j(t) - a_j(t)| \right)\  
(\sup_j \sup_{ |y-a(t)| \le |\xi(t)-a(t)| }  |\nabla_{y_j} W(y)| ) .
\label{Sigma.est0}\end{align}
From \eqref{xis.def} it follows that 
{if  $y\in (\R^2)^n$ is such that
$|y - a(t)|\le |\xi(t)- a(t)|$, then
$\rho_y \ge \frac 12 {\rho_{a(t)}}$}, so we can use \eqref{Wderivatives} to find that $|\nabla_{y_j} W(y)|\le \frac {Cn}{\rho_a}$.
Hence \eqref{Sigma.est0} and \eqref{eta.rocks1} yield
\begin{align*}
\Sigma(u(t); \xi(t))
\lesssim  \frac{n}{\rho_{a(t)}} (\eta(t) +  \e\calP )  +  \e^{1 \over 2} 
\end{align*}
and therefore \eqref{Sigma.est1}. Then \eqref{gstab.ref1}  - \eqref{gstab.ref3}
follow from conclusions \eqref{gstab.c1}  -  \eqref{Lpcurrentboundsbigball}
of Proposition \ref{Propgammastability} (applied with $\xi_j$ replacing $\alpha_j$
for all $j$, and with $s_\e = \e\calP$ from \eqref{xis.def}).
\end{proof}



\section{Initialization of Gronwall loop}  \label{S.etadot}

We now begin to estimate $\dot \eta$. We will prove

\begin{proposition}  \label{propgronwall}
For all $0 \leq t \leq \tau_1$ we have 
{ \begin{align}
|\dot \eta (t)| 
\ \lesssim  \ {n \over \rho_{a(t)}^2} \eta(t) + \e^{1\over2}\calP + \frac {n^{3\over2}}{\rho_{a(t)}} 
\sqrt{ {n\over \rho_{a(t)}} \eta + \e^{1\over 2}\calP } 
\ \overset{\eqref{eta.bound}, \eqref{Psmall}}\le  \  \e^{1 \over 9 }. 
\label{doteta.bound1}\end{align} }
As a result,  for $t\in [\delta_\e, \tau_1]$ and $\e < \e_0$,
\begin{equation}  \label{lipschitzetaestimate}
\LV \eta(t) - \LA \eta(t) \RA_{\delta_\e} \RV   \le   \e^{\frac 25 + \frac 1 9} 
\qquad\mbox{for $\delta_\e = \e^{2 \over 5}$}. 
\end{equation}
\end{proposition}

Due to the term involving $\sqrt{\eta(t)}$, the above estimate 
cannot be employed in a Gr\"onwall inequality to derive
very good bounds on the growth of $\eta$.   On the other hand, conclusions such
as \eqref{lipschitzetaestimate} will be used heavily in our later arguments.

To prove Proposition~\ref{propgronwall} we have the following decomposition of 
$\dot\eta$, which will be used again later.

\begin{lemma}
Let $u$ be a solution to the Schr\"odinger equation. Then for 
$0\le t \le \tau_1$ and $j=1,\ldots, n$
\begin{equation} \label{expansionetadot}
\dot{\eta_j} \ = \  \sum_{\ell = 1}^7 T_{j,\ell} 
\end{equation}
where 
\begin{align*}
T_{j,1} & =   \nabla_x \varphi(\xi_j - a_j, \rho_{a}) 
\cdot \J \LC \nabla_j W(\xi) - \nabla_j W(a) \RC \\
T_{j,2} & = \int J(u) \partial_\rho\varphi(x-a_j, \rho_a) \,\dot \rho_a \ dx\\
T_{j,3} & =  \int \LC J(u) - \sum_{i=1}^n \pi \delta_{\xi_i} \RC 
\dot a_j \cdot \nabla_x\varphi(x-a_j, \rho_a) dx \\
T_{j,4} & =  \int \J_{kl} \p_{x_l  \,x_m} \Phi_j 
\p_{x_k}\LV u \RV \p_{x_m} \LV u \RV dx  \\
T_{j,5} & = \int \J_{kl} \p_{x_l \,x_m} \Phi_j  
\LC {j(u) \over |u|} - j(u_*) \RC_k 
\LC {j(u) \over |u|} - j(u_*) \RC_m dx \\
T_{j,6} & =  \int \J_{kl} \p_{x_l \,x_m} \Phi_j  
\LC {j(u) \over |u|} - j(u_*) \RC_k (j(u_*))_m  dx \\
T_{j,7} & =   \int \J_{kl} \p_{x_l \,x_m} \Phi_j  
\LC {j(u) \over |u|} - j(u_*) \RC_m (j(u_*))_k  dx
\end{align*}
where $(j(u))_m$ is the $m$th component of the vector $j(u(t))$ and $u_*
= u_*(\cdot; \xi(t), d)$. 
\label{decomp.etadot}\end{lemma}

\begin{proof}
As in the proof of Lemma 1 of \cite{JSp2},
we write
\[
\frac d{dt} \eta_j = \frac d {dt}\int \Phi_j \ J(u) \ dx \ = \ 
\int  \frac d {dt}\Phi_j \ J(u)  \ dx +  \int\Phi_j \frac d {dt}\ J(u) \ dx,
\]
then expand using the definition of $\Phi_j$ and
the equation  \eqref{consjac} for the evolution of the Jaobian.
The proof then follows
from routine algebraic manipulations, together with the definition of
$\Phi_j$, the equation \eqref{PV2} for $\dot a_j$, and the fact that
\[
  \int \J_{kl} \p_{x_l \,x_m} \Phi_j  
 (j(u_*))_m (j(u_*))_k  \ dx =  \nabla_x\varphi(\xi_j-a_j, \rho_a)\cdot \nabla_j W(\xi).
\]
See \cite{JSp2} for a very similar argument\footnote{Indeed, compared to \cite{JSp2},
the only new term  here is $T_{j,2}$, which
arises from
the fact that in \cite{JSp2}, the definition of $\Phi_j$ had the form
$\Phi_j(x,t) := \varphi(x-a_j, \rho)$, for some $\rho$ independent of $t$,
whereas here  $\rho = \rho_{a(t)}$ depends on $t$. Apart
from this, the statement and proof are exactly the same.} with more details.
\end{proof}

%
%
%
%
%
%

Next we follow \cite{JSp2} and {estimate} $\dot \eta$ by separately considering
contributions from the different terms
isolated in Lemma \ref{decomp.etadot}.

\begin{proof}[proof of Proposition \ref{propgronwall}]
Note from Lemma  \ref{decomp.etadot} and the definition \eqref{eta.def} of $\eta$ that 
\begin{equation}
\dot \eta =  T_1+ \ldots + T_7 .\quad
\mbox{ where }
T_k = \sum_{j=1}^n \frac{\eta_j}{|\eta_j|} \cdot T_{j,k}.
\label{doteta.split2}\end{equation}
{We estimate these terms in turn, suppressing the argument $t$.
From \eqref{Wderivatives} we have }
\begin{equation} \label{rhodotbound}
| \dot \rho_a | \lesssim \sup_{j} | \dot a_j | = \sup_{j} \LV \nabla_{a_j} W \RV \lesssim {n \over \rho_a} \le \calP.
\end{equation}
To estimate $T_1$, note that $\nabla_x \varphi(\xi_j-a_j, \rho)$ is the identity matrix, so
in fact
\begin{align*}
|T_1| 
&\le 
 \sum_j |\nabla_jW(\xi) - \nabla_j W(a)| \\ 
 &\le
 \sum_{k=1}^n |\xi_k(t) - a_k(t)|
(\sup_k \sup_{ |y-a(t)| \le |\xi(t)-a(t)| }  |\nabla_k \nabla_{j  } W(y)|) .
\end{align*}
Using  \eqref{eta.rocks1}, as well as bounds on $\nabla^2W$
from \eqref{Wderivatives}, and arguing 
as in the proof of \eqref{Sigma.est1}, we conclude that
\begin{equation}
|T_1| 
\le
 C \frac {n}{\rho_a^2}\eta(t) + \e \calP.
 \label{T1.est}
\end{equation}
Next, the definition of $\varphi$ implies that $\partial_\rho\varphi(x-a_j, \rho_a) = - \frac {(x-a_j)}{\rho}[\frac {(x-a_j)}{\rho}\cdot\nabla\chi(\frac {x-a_j}\rho)]$, which vanishes in $B_{\rho(a)}(a_j)$, and in
particular at $x=\xi$. Thus
\begin{align}
|T_2|  & = \sum_j \left|\int (J(u) - \pi \sum  \delta_{\xi_i}) \partial_\rho\varphi(x-a_j, \rho_a) \,\dot \rho_a \ dx
\right|\nonumber \\
&\le |\dot \rho_a| \ \|J(u) - \pi \sum  \delta_{\xi_i}\|_{X^*_{\ln}}
\sum_j \| \partial_\rho\varphi(\cdot -a_j, \rho_a) \|_{X_{\ln}} \le \e\calP,
 \label{T2.est}\end{align}
where we have used {\eqref{xis.def} and \eqref{rhodotbound}.} Similar arguments yield
\begin{equation}
|T_3| \leq  \e\calP.
\label{T3.est}
\end{equation}
Continuing, 
since $\nabla^2 \Phi_j$ vanishes in $B_{\rho_a}(a_j)$ and noting that
$\|\nabla^2\Phi_j\|_\infty \le C\rho_a^{-1}$, we conclude from 
\eqref{gstab.ref1} that
\begin{equation}
\label{T45.est}
|T_4|, \, |T_5| \ \le \ \frac n {\rho_a^2}\eta + \e^{\frac 12}\calP.
\end{equation}
Finally, since 
$\|  j(u_*)\|_{L^2(\cup_j \mbox{\scriptsize{supp}} \nabla^2\Phi_j)}
\le \frac {Cn}{\rho_a}( C n \rho_a^2)^{1\over2} \lesssim n^{3\over2}$,
we can again {use} \eqref{gstab.ref1} and the same properties of $\nabla^2\Phi_j$ as above to
deduce that
\[
|T_{6}|, \, |T_7| \ \le
\frac C {\rho_\alpha}
\|  j(u_*)\|_{L^2(\cup_j \mbox{\scriptsize{supp}} \nabla^2\Phi_j)}
\LN \frac {j(u)}{|u|} - j(u_*) \RN_{L^2(\R^2_{\rho_a})}
\le C \frac{n^{\frac 32}}{\rho_a}\left( \frac{n}{\rho_a} \eta + \e^{\frac 12}\calP\right)^{1\over 2}.
\]
By assembling these estimates, we obtain inequalities \eqref{doteta.bound1},
and then \eqref{lipschitzetaestimate} follows rather easily.
\end{proof}

\section{Bounds on the supercurrent}  \label{section.currentbounds}

In this section we will prove the following

\begin{proposition}  \label{currentbounds}

For all $t \in [\delta_\e, \tau_1]$ we have
\begin{equation}
\frac {d}{dt}
\LA \eta \RA_{\delta_\e} (t) \  \lesssim  \ {n \over \rho^2_{a(t)}} \LA \eta \RA_{\delta_\e} (t) (t) + \e^{1\over 2}\calP  \label{avgcurrentboundsresult}
\end{equation}
for ${\delta_\e} = \e^{2 \over 5}$.
\end{proposition}

{The main point is to improve our earlier estimates of the terms $T_6$ and $T_7$ arising in Lemma
\ref{decomp.etadot}.
We will focus on  $T_6$, as the argument for $T_7$ is identical.  We will write}
\begin{equation}  \label{zeta.def}
T_6 =
 \int \zeta\cdot\    (\frac{j(u)}{|u|} - j(u_*))  \ dx, \qquad
\zeta_k := \sum_j \J_{kl} \p_{x_l x_m} 
\LC {\eta_j \over |\eta_j|} \cdot \Phi_j \RC j_m(u_*), \quad k =1, 2,
\end{equation}
where $u_*(t,x) = u_*(x; \xi(t),d)$, and we recall that $\xi(t)$ is characterized by
\eqref{xis.def}.
We note  that  $|\zeta| \leq C {n \over \rho_\xi^2}$, since $\nabla^2\Phi_j=0$
in $B_{\rho(a)}(a_j)$, and also that $| \operatorname{supp} \zeta | \leq C n \rho_\xi^2$.  Therefore,
\begin{equation} \label{zetaboundLp}
\LN \zeta \RN_{L^q(\R^2)} \lesssim n^{1 + {1\over q}} \rho_\xi^{{ 2 \over q} - 2} \ \le \ \calP
\end{equation}
for $1 \leq q \leq \infty$.  

For every $s$, we carry out a Hodge decomposition,
writing
\begin{equation}
j(u) - j(u_*) = f+ g, \qquad\mbox{ where }
\nabla\times f =  \nabla\cdot g = 0
\label{hodge}\end{equation}
and
\begin{equation}
\| f \|_{L^{\frac 43}+L^2}
+
\| g \|_{L^{\frac 43}+L^2}
\lesssim
\| j(u) - j(u_*) \|_{L^{\frac 43}+L^2}\overset{\eqref{gstab.ref2}} \lesssim
  \left( \frac{n  }{\rho_{a(s)}} \eta(s) \right)^{1\over2} + \e^{\frac 14}\calP.
\label{fgLqp}\end{equation}
(We suppress the dependence of $f,g$ on $s$.)
The existence of such $f,g$ is standard. Indeed, if we temporarily write $\psi = j(u)- j(u_*)$, then
in terms of the Fourier transform
\[
\hat f(\xi) = \frac { \xi (\xi \cdot \hat \psi(\xi))}{|\xi|^2},
\qquad
\hat g(\xi) = \frac { \xi^\perp (\xi^\perp \cdot \hat \psi(\xi))}{|\xi|^2},
\]
and we deduce  \eqref{fgLqp} from standard harmonic analysis facts.
Using \eqref{consmass}, \eqref{hodge},  and properties of $j(u_*)$,
we find that
\begin{equation}\label{fgequation}
\nabla\cdot f= \frac 12 \partial_t (|u|^2-1),
\qquad
\nabla\times g  = 2 (J(u) - \sum_{j=1}^n \pi  \delta_{\xi_j}  ) . 
\end{equation}

We may now write
\begin{align}
T_6 
&= 
\int \zeta \cdot {j(u) \over |u|} \LC 1 - |u| \RC\,dx \   + \ 
\int \zeta \cdot g\, dx 
\ +\ 
\int\zeta \cdot f\,dx
\nonumber\\
&=:
\Xi_1+\Xi_2+\Xi_3.
\label{T6decomp}\end{align}

We easily dispense with $\Xi_1$ by using \eqref{zetaboundLp} and \eqref{energyupperboundBigBall}  to find that
\begin{equation} \label{chi1bound} 
\LV \Xi_1 \RV \  \leq  \ 2 \e \LN \zeta (s) \RN_{L^2} \sqrt{ \int_{B_{R_\alpha}} e_\e(u(s)) \ dx} \\
\  \le\ \e \calP.
\end{equation}

In order to bound $\Xi_2$ and $\Xi_3$, we will analyze $f$ and $g$ separately.

\subsection{Curl bounds}

Our estimate of $\Xi_2$ will use the following lemma to exploit the fact that $\nabla\times (j(u)-j(u_*) = \nabla\times g$
is very small.

\begin{lemma}
Let $g\in L^r + L^s(\R^2;\R^2)$ for some $r,s<\infty$. Assume that
$\nabla\cdot g=0$, and that 
$\mu := \nabla \times g$ is a signed measure 
with $\| \mu \|_{X_{\ln}^*}<\infty$.

Then for any $p\in (1,2)$ and any $R>0$,
\[
\| g\|_{L^p(B_R)} \le C(1+\ln^+R) 
\Big[
\| \mu\|_{X_{\ln^*}}^{\frac 2p-1}  \   \big( | \mu|(B_{R+1}) \big)^{2- \frac 2p}
+
\| \mu\|_{X_{\ln^*}}
\Big]. 
\]
\label{poissonX}
\end{lemma}

\begin{proof}
{1.} Fix a smooth vector field $w$ with compact support in $B_R$.
There exist smooth functions $\alpha,\beta$ on $\R^2$
such that 
\begin{equation}\label{zetahodge}
w = \nabla \alpha + \nabla^\perp \beta, \qquad\quad
\| \nabla \alpha\|_{L^p} + \|\nabla^\perp\beta\|_{L^p} \le C_p\|w\|_{L^p}  \ \le \ \calP \quad \mbox{ for }1<p<\infty.
\end{equation}
Indeed, by differentiating \eqref{zetahodge}, we find that
\[
\Delta \alpha = \nabla\cdot w, \qquad \Delta \beta = \nabla\times w,
\]
so $\alpha, \beta$  may be found by convolution with 
the Green's function $G(x) = \frac 1{2\pi}\ln |x|$ and 
an appeal to standard elliptic estimates.
For example, we have
\[
\alpha(x) = 
 \frac 1{2\pi} \int _{B_{R}}\ln|x-y|  (\nabla\cdot w)(y) dy 
\ = \ 
\frac  1{2\pi} \int _{B_{R}}\frac{x-y}{|x-y|^2}  \cdot w(y) dy . 
\]
This formula can also be differentiated with respect to $x$, and a
similar formula also holds for $\beta$. From these considerations one easily sees that
for $k \ge 0$ and $|x| \ge  2R$,
\begin{equation}\label{ab.decay}
|\nabla^k \alpha(x)| \le  C \|  w\|_{L^1}|x|^{-k-1}, 
\quad\
|\nabla^k \beta(x)| \le  C \| w\|_{L^1}|x|^{-k-1} .
\end{equation}

{2.} We now claim that 
\begin{equation}
\int w \cdot g \ dx=   - \int \beta \,\nabla\times g \ dx  = -  \int \beta\, d\mu.
\label{ibyphodge}\end{equation}
This is formally clear, and so the point is to justify the integration by parts.
For any family  $(\chi_{\widetilde R})_{\widetilde R\ge 1}$ satisfying \eqref{chiRdef1}, \eqref{chiRdef2},
it follows from \eqref{ab.decay} that $\| \alpha \chi_{\widetilde R} \|_{L^p(\R^2)}\to 0$ as $\widetilde R\to\infty$, for any $p>1$.
So, recalling that $\nabla\cdot g=0$,
we have 
\begin{align*}
\Big|\int \nabla\alpha \cdot g \ dx \Big|
= 
\lim_{\widetilde R\to\infty}\Big| \int \chi_{\widetilde R} \nabla \alpha \cdot g \ dx \Big|
&= 
\lim_{\widetilde R\to \infty} \Big|\int  \alpha  \nabla\chi_{\widetilde R } \cdot g \ dx\Big| \\
&
\le \lim_{\widetilde R\to\infty} \big(\| g \|_{L^r+L^s} \| \alpha\nabla \chi_{\widetilde R} \|_{L^{r'}\cap L^{s'}} \big) = 0.
\end{align*}  
Essentially the same argument shows that 
\[
\int \nabla^\perp \beta \cdot g \ dx =  -\int \beta\, \nabla\times g \ dx,
\]
and this completes the proof of \eqref{ibyphodge}.

{3.} Now we fix $p,q$ such that $1<p<2$ and $\frac 1p+\frac 1q=1$, and
we write
\[
\int \beta\, d\mu = 
\int \eta_\delta * \beta\, d\mu + 
\int (\beta - \eta_\delta*\beta) \,d\mu ,
\]
where $\eta_\delta$ is a radially symmetric mollifier with support in $B_\delta$ for some $\delta\le 1$ to be chosen below.
By elementary computations, a Sobolev embedding theorem, and
properties \eqref{zetahodge} of $\beta$, 
\[
| \beta-\eta_\delta *\beta| \le \delta^\theta [ \beta ]_\theta \le 
C\delta^\theta
\| \nabla \beta \|_{L^q} \le 
C\delta^\theta
\| w \|_{L^q} \
\]
where $\theta = 1- \frac 2q = \frac 2p-1$ and  $[  \ \cdot \  ]_\theta$ denotes the $\theta$-H\"older seminorm.
In addition, since $\eta_\delta$ is radially symmetric and 
$\beta$ is harmonic outside $\mbox{supp}\,w\subset B_R$,
it follows that $\eta_\delta* \beta = \beta $ outside of $B_{R+\delta}$.
Therefore
\[  
 \int ( \beta-\eta_\delta *\beta)d\mu
\le C \delta^{1- \frac 2q} \| w\|_{L^q} |\mu|(B_{R+\delta}).
\] 
Now we again use properties \eqref{zetahodge} of $\beta$
to find that
\begin{equation}
\|\nabla (  \eta_\delta *\beta )\|_{L^\infty} =
\| \eta_\delta *\nabla \beta \|_{L^\infty} 
\le \|\eta_\delta\|_{L^p} \ \| \nabla \beta \|_{L^q} \le \ C \, \delta^{-\frac 2q}\, \| w \|_{L^q} .
\label{pqpq2}\end{equation}
Next,  for $|x|\ge 2R+1$, it follows from \eqref{ab.decay} and the fact that $\eta_\delta*\beta = \beta$
that
\begin{equation}
|\nabla  ( \eta_\delta *\beta)(x)| \le C |x|^{-2} \|  w \|_{L^1}
\le  C |x|^{-2} R^{\frac 2 p} \| w\|_{L^q} \ \le  C |x|^{-\frac 2q } \|w\|_{L^q}. 
\label{pqpq3}\end{equation}
By combining \eqref{pqpq2} and \eqref{pqpq3}, we deduce that
\[
\| \eta_\delta *\beta \|_{X_{\ln}} \ \le C(1+ \ln^+ R) \  \delta^{-\frac 2q} \|w\|_{L^q},
\]
and it follows from Lemma \ref{lem.dual} that
\[
\left|\int \eta_\delta*\beta \ d\mu \right| \le  C(1+ \ln^+ R) \  \delta^{-\frac 2q} \|w\|_{L^q} \| \mu\|_{X_{\ln}^*}.
\]
Assembling the above, we find that if $\delta \le 1$ then 
\[
\int w \cdot g \ dx \le \ C \left( \delta^{1-\frac 2q}a  + 
\delta^{-\frac 2q} b 
\right)\|w\|_{L^q}\qquad\mbox{ for }a:=|\mu|(B_{R+1})\mbox{\ \ and  \  \ }
b:= (1+ \ln^+ R) \| \mu\|_{X_{\ln}^*}.
\]
Choosing $\delta = \min \{ 1, b/a\} \le 1$, we deduce that
\[
\int w \cdot g  \ dx \le C \| w \|_{L^q} \Big(  a^{\frac 2q} b^{1-\frac 2q} + b \Big).
\]
By density, the same inequality holds for all $w\in L^q(B_R)$, and by duality, this
implies the conclusion of the lemma, after a little rewriting.
\end{proof}

The vector field $g$ appearing in $\Xi_2$ satisfies
$\mu :=\nabla\times g =  2 \LC J(u) - \sum \pi  \delta_{\xi_j} \RC$.
Thus
\[
|\mu|(B_{R_{a(t)}+1})
\ \le C \int_{B_{R_{a(t)}+1}} e_\e(u) \ dx + n\pi \\
 \overset{\eqref{energyupperboundBigBall}} \le \calP.
\]
We therefore deduce from the above lemma that for $1<p<2$,
\[
\| g \|_{L^p(B_{R_{a(t)}})} \le  \e^{\frac 2p-1} \calP
\]
where the coefficients in the polynomial $\calP$ depend on $p$.
In particular, noting that the function $\zeta$ appearing in $\Xi_2$ is supported in
$B_{R_{a(t)}}$, we have proved
\begin{lemma}
For any $p\in (1,2)$  {and $t\in [\delta_\e, \tau_1]$,}
\begin{equation}
| \Xi_2 |  \ \leq \  \LN \zeta \RN_{L^{q}} \LN  g \RN_{L^{p}(B_{R_{a(t)}}) }
 \ \lesssim \ 
\e^{\frac 2p-1} \calP.
\label{chi2bound}
\end{equation}
\end{lemma}

\subsection{Divergence bounds}

We now exploit the fact that the divergence of $j(u)-j(u_*)$ is small, after averaging in the $t$ variable.
We start by
noting that
\begin{align*}
\LA \Xi_3\RA_{\delta_\e} 
&= \int \LA  \zeta \RA_{\delta_\e}\cdot \LA f \RA_{\delta_\e} dx
+ \LA  \int \big( \zeta - \LA  \zeta\RA_{\delta_\e} \big)\cdot \big(f- \LA f \RA_{\delta_\e}\big) dx\RA_{\delta_\e}\\
&=
 \Xi_{3,1} + \Xi_{3,2}.
\end{align*}

\begin{lemma}
For any $\theta\in (0,1)$  {and $t\in [\delta_\e, \tau_1]$,}
\begin{equation}  \label{chi3bound}
\LV \Xi_{3,1} \RV  \le  \left( \frac \e {\delta_{\e}}\right)^{1-\theta}\calP
\end{equation}
(for $\calP$ depending on $\theta$). In particular, since $\delta_\e = \e^{2/5}$,
we have $\LV \Xi_{3,1} \RV  \le \e^{1/2}\calP$.
\end{lemma}

\begin{proof}

By \eqref{fgequation}, 
\[
\nabla \cdot\LA f  \RA_{\delta_\e} 
= 
\LA  \nabla \cdot f  \RA_{\delta_\e} 
= \frac 12 \LA \partial_t (|u|^2-1) \RA_{\delta_\e} = 
{1\over \delta_\e} \left. { |u|^2 - 1 \over 2} \right|_{t - \delta_\e}^t. 
\]
Hence
$\|\nabla\cdot \LA f  \RA_{\delta_\e} \|_{L^2}  \le  \frac\e{\delta_\e}\calP$, by \eqref{gstab.ref3}.
Also, since $\nabla\times \LA  f \RA_{\delta_\e} = \LA \nabla\times f \RA_{\delta_\e} = 0$, elliptic regularity implies that 
$\| \nabla \LA f \RA_{\delta_\e} \|_{L^2} \le   C\| \nabla\cdot \LA f \RA_{\delta_\e} \|_{L^2}$.
Then we can apply the Sobolev-Nirenberg-Gagliardo inequality 
to find that for $\theta\in (0,1)$,
\[
\|  \LA f  \RA_{\delta_\e} \|_{L^{\frac 4{3\theta}} + L^{\frac 2\theta}} \le \
\|  \LA f  \RA_{\delta_\e} \|_{L^{\frac 43} + L^2}^\theta 
\| \nabla  \LA f  \RA_{\delta_\e} \|_{L^2}^{1-\theta}  \ \le  \ \calP^\theta (\frac \e{\delta_\e}\calP)^{1-\theta}. 
\]
Since
$\|\zeta\|_{L^r} \le \calP$ for every $r$, the same holds for $\LA \zeta\RA_{\delta_\e}$, and we conclude
by H\"older's inequality that
\[
\LV \Xi_{3,1} \RV   
\ \le
\calP \ \|  \LA f  \RA_{\delta_\e} \|_{L^{\frac 4{3\theta}} + L^{\frac 2\theta}} \ 
\le \  \calP (\frac\e{\delta_\e})^{1-\theta}.
\]
\end{proof}

Finally, we consider $\Xi_{3.2}$.

\begin{lemma} For all $\delta_\e \leq t \leq \tau_1$ we have 
\begin{equation}  \label{chi4bound}
\LV \Xi_{3,2} \RV
\le
\e^{1/2}\calP .
\end{equation}
\end{lemma}

\begin{proof}
{ 1}. First, note that
\begin{equation}
|\Xi_{3,2}| 
\le 
\Big(\sup_{s\in [t-\delta_\e, t]}\ \| \zeta(s) - \LA \zeta\RA_{\delta_\e}\|_{L^4 \cap L^2} \Big)
\Big( \sup_{s\in [t-\delta_\e, t]}\ \|  f(s) - \LA f\RA_{\delta_\e} \|_{L^{4\over3} + L^2}\Big).
\label{Xi32a}
\end{equation}
The term involving $\zeta - \LA \zeta \RA_{\delta_\e}$ is more complicated, and 
we consider it first.
Clearly
\[
\sup_{s\in [t-\delta_\e, t]}\ \| \zeta(s) - \LA \zeta\RA_{\delta_\e}\|_{L^4 \cap L^2} 
\le
\sup_{s, s'\in [t-\delta_\e, t]}\   \| \zeta(s) -  \zeta(s')\|_{L^4 \cap L^2}  .
\]
We henceforth assume that  $0 \le t-\delta_\e \le s, s' \le t \le \tau_1$.
From the definition \eqref{zeta.def} of $\zeta$,
\begin{align}
\zeta_k(s) - \zeta_k(s')
& = \sum_j \J_{kl} \p_{x_l x_m} \LB \frac{\eta_j}{|\eta_j|}\cdot ( \Phi_j (s) - \Phi_j (s') )\RB  \ 
j_m(u_* )(s) 
\nonumber\\
& \quad  + \sum_j \J_{kl} \p_{x_l x_m}(\frac{\eta_j}{|\eta_j|}\cdot  \Phi_j) (s')  \ 
\LB j_m(u_* )(s) - j_m(u_* )(s') \RB 
\nonumber \\
& = \Lambda_1 + \Lambda_2.
\label{Lambda.defs}\end{align}
Recall that $\Phi_j(s) = \varphi(\cdot-a_j(s), \rho_{a(s)})$ and   $u_* = u_*(\cdot, \xi(s), d)$.
So we  need to establish several facts 
about these various parameters and their effect on support of the $\Phi_j$'s.

We first note that $$\mbox{supp}\, \nabla^2 \Phi_j(s) \subseteq A_{\rho_{a(s)}}^j := B_{2 \rho_{a(s)}}(a_j(s)) \backslash  B_{ \rho_{a(s)}}(a_j(s)).$$
Next, since $|\dot a_j| \le  \calP$ for every $j$,
 it follows from \eqref{rhoRbds} that
\begin{equation}
|a_j(s) - a_j(s')| \ \le \ \calP |s-s'| \  \le \ \calP \delta_\e \  \le \  \frac 1{100} \inf_{\sigma>0} \rho_{a(\sigma)}
\le \frac 1 {100}\rho_{a(t)}.
\label{a.change}\end{equation}
It follows that (for the  same $\calP$ as in \eqref{a.change})
\begin{equation}
|\rho_{a(s)} - \rho_{a(s')} | \ \le \  \calP \delta_\e . 
\label{rho.change}\end{equation}
This, in turn, implies that 
\begin{equation}
A_{\rho_{a(s)}}^j \subseteq \widetilde{A}^j_t :=
B_{3 \rho_{a(t)}}(\xi_j(t)) \setminus B_{\frac12 \rho_{a(t)}}(\xi_j(t)) .
\label{supports.bound}\end{equation}
We also infer from \eqref{futureref}, \eqref{doteta.bound1},
\eqref{a.change}, that
\begin{equation}
\sum \LV \xi_j (s) - \xi_j (s') \RV \ \le
\sum \left( | \eta_j(s) - \eta_j(s')| + |a_j(s) - a_j(s')| \right) + \e \calP
\le \ \calP \delta_\e.
\label{xi.change}\end{equation}
Finally, we have for any $s,s' \in [t- \delta_\e, t]$, using \eqref{rho.change},
\begin{equation}  \label{estimateson1overrhoa1}
\LV {1\over \rho_{a(s)}^2} -  {1\over \rho_{a(s')}^2} \RV 
\lesssim \frac{|\rho_{a(s)}-\rho_{a(s')}|} {\rho_{a(t)}^3} \le
\calP \delta_\e.
\end{equation}


{2.}
We now consider $\Lambda_1$.  If we temporarily write 
$S :=  \{(x,\rho) : |\rho - \rho_{a(t)}| \le \rho_{a(t)}/10\}$, then
we deduce from
the definition of $\Phi$, the mean value theorem, and
\eqref{a.change}, \eqref{rho.change} that 
\begin{align*}
\LN 
\p_{x_l x_m} ( \Phi_j (s) - \Phi_j (s') )  \ 
\RN_{L^\infty}
& \leq  \LN \nabla_x^2 \nabla_{x,\rho} \varphi
\RN_{L^\infty (S)}
( |a_j(s) - a_j(s')| + |\rho_{a(s)} - \rho_{a(s')}|) \\
& \lesssim {n \over \rho_{a(t)}^3} \delta_\e  \ \ =\  \calP \delta_\e.
\end{align*}
From \eqref{supports.bound} we infer that 
$| j(u_*)(\xi(s))| 
\lesssim { n \over \rho_{a(t)}}$ on the support of $ \Lambda_1$,
and since the support of $\Lambda_1$ has measure bounded by $C n \rho_{a(t)}^2$,
we conclude that
\begin{align}  \label{t2biiestI}
\LN \Lambda_1 \RN_{L^4} \leq  {C n^{2} \over \rho_{a(t)}^4}\left( C n \rho_{a(t)}^2\right)^{1\over4}
\delta_\e =  \calP\delta_\e.
\end{align}
Next we consider $\Lambda_2$. 
Since $\LN \sum_j \p_{x_l x_m}  \Phi_j  \RN_{ L^\infty}  \leq {C \over \rho_{a(t)}}$, and
noting that $\mbox{supp}\, \Lambda_2$ has measure at most $Cn\rho_{a(t)}^2$,
we use
H\"older's inequality to estimate
\begin{equation} \label{lambda2est1}
\LN \Lambda_2 \RN_{L^4} \le
\frac C{\rho_{a(t)}}
\LN j (u_* )(s) - j(u_* )(s') \RN_{L^\infty( \cup_j \widetilde{A}_t^j )}
(C n \rho_{a(t)}^2)^{1\over4}.
\end{equation}
For the $L^\infty$ bound we have 
\begin{equation}
\LN j (u_* )(s) - j(u_* )(s') \RN_{L^\infty( \cup_j \widetilde{A}_t^j )} 
 \stackrel{\eqref{jstarLipschitz},\eqref{supports.bound}}{\leq} { C \over  \rho^2_{a(t)} } \sum_{j=1}^n \LV \xi_j (s) - \xi_j (s') \RV ;
  \label{jstardifflambda2}
\end{equation}
consequently, we deduce from \eqref{xi.change} that
\begin{equation}
\LN \Lambda_2 \RN_{L^4} \le \calP \delta_\e.
\label{t2biiestII}
\end{equation}
Note also that by H\"older's inequality and \eqref{supports.bound},
\[
\| \zeta(s) - \zeta(s')\|_{L^2} \le | \cup \widetilde A_t^j |^{1/4} \| \zeta(s) - \zeta(s')\|_{L^4}
\le \calP \| \zeta(s) - \zeta(s')\|_{L^4}.
\]
So it follows from \eqref{t2biiestI} and \eqref{t2biiestII} that
\begin{equation}
\| \zeta(s) -  \zeta(s')\|_{L^4 \cap L^2}  \le  \calP \delta_\e
\ \ \ \qquad\mbox{ for every }s,s'\in[t-\delta, t].
\label{zd}\end{equation}

{3}. For the other term in \eqref{Xi32a}, we simply note that for $t<\tau_1$,
\begin{align*}
\sup_{s\in [t-\delta_\e, t]}\ \|  f(s) - \LA f\RA_{\delta_\e} \|_{L^{4\over3} + L^2}
&\le
2 \sup_{s\in [t-\delta_\e, t]}\ \|  f(s)  \|_{L^{4\over3} + L^2}\\
& \overset{\eqref{fgLqp}, \eqref{tau1.def}}
\lesssim
\e^{1/6}\,  \calP. 
\end{align*}
The conclusion of the lemma follows from this and \eqref{zd}, recalling that $\delta_\e  = \e^{2/5}$.
\end{proof}

\begin{proof}[proof of Proposition~\ref{currentbounds}]

From \eqref{doteta.split2} we have
\[
\frac d{dt}\LA \eta\RA_{\delta_\e} = 
\LA T_1\RA_{\delta_\e}
+\ldots + 
\LA T_7\RA_{\delta_\e}
\]
using the notation of \eqref{doteta.split2}.
In view of \eqref{T1.est}-\eqref{T45.est},
\begin{align*}
\sum_{i=1}^5 | \LA T_i \RA_{\delta_\e} | 
& \le
\sum_{i=1}^5  \LA |T_i| \RA_{\delta_\e}  
\ \lesssim
\LA\frac {n}{\rho_{a(\cdot)}^2}  \eta + \e^{1\over 2}\calP\RA_{\delta_\e} \\
& \lesssim 
\frac {n}{\rho_{a(t)}^2}\LA \eta\RA_{\delta_\e} + \sup_{s\in[t-\delta_\e,t]} \LV {1\over \rho_{a(s)}^2} - {1\over \rho_{a(t)}^2} \RV \sup_{r \in [t-\delta_\e,t]} \LV \eta (r) \RV + \e^{1\over 2}\calP \\
& \stackrel{\eqref{estimateson1overrhoa1}}{\lesssim} \frac {n}{\rho_{a(t)}^2}\LA \eta\RA_{\delta_\e} +  \e^{1\over2}\calP.
\end{align*}
Finally, from \eqref{T6decomp}, \eqref{chi1bound}, \eqref{chi2bound}, \eqref{chi3bound} and \eqref{chi4bound} we have 
\begin{equation} 
 | \LA T_6 \RA_{\delta_\e} | +
  | \LA T_7 \RA_{\delta_\e} | \ \lesssim \ \e^{1/2}\calP,
\label{T6.avgest}\end{equation}
recalling that $| \LA T_7 \RA_{\delta_\e} |$ is bounded in exactly
the same way as $ | \LA T_6 \RA_{\delta_\e} |$. 
\end{proof}


\section{Completion of Gronwall argument} \label{S.conc}

\begin{proof}[conclusion of the proof  of Theorem \ref{thmdynamics}]

{1}. We first claim 

\begin{equation}
\eta(t) \le \frac 12 \e^{1/ 3}  \quad \mbox{ for } \quad 0 \le t \le\min \{ \tau_1,\tau_\star\},
\label{dtsts}\end{equation}
where $\tau_\star$ was defined in \eqref{Tstar.def}, and we recall that 
\[
\tau_1 := \inf\{ \tau>0 : \lambda(\tau) >  \e^{1/3}  \}
\quad\mbox{ for }\ \ \lambda(t) := \| J(u(t)) - \pi \sum \delta_{a_i(t)} \|_{X_{\ln}^*}.
\]
We have assumed in \eqref{t1.h1} that $\lambda(0) \le \e^{1/2}$, and it follows 
from
\eqref{eta.bound} that
$\eta(0) \le   \e^{1/2}\calP$. 
Taking $\e_0$ sufficiently small (see \eqref{Psmall}),  we then see from \eqref{doteta.bound1} that
\begin{equation}
\eta(t) \le  \frac 12 \e^{2/5} + t \  \e^{1/ 9}  \qquad \mbox{ for }0\le t \le \tau_1
\label{eta.bound0}\end{equation}
which in particular implies that \eqref{dtsts} holds if $\tau_1 \le \delta_\e = \e^{2/5}$.
Thus we may assume that $\tau_1\ge \delta_\e$. Then it follows from 
\eqref{eta.bound0} that $\eta(t)\le \delta_\e$ for $t\le \delta_\e$, and hence that
$\LA \eta \RA_{\delta_\e}(\delta_\e) \le \delta_\e$.

Now, recalling the differential inequality \eqref{avgcurrentboundsresult}
satisfied by $\LA \eta \RA_{\delta_\e}$,
Gr\"onwall's inequality implies that
\begin{equation}
\LA \eta \RA_{\delta_\e}(t) \ 
\le \exp \LB C n \int_{\delta_\e}^t \rho_{a(s)}^{-2} ds \RB \left ( \LA \eta \RA_{\delta_\e}(\delta_\e) 
+ \e^{1/2}\calP
\right)
\ \le \ 2 \delta_\e \exp \LB C n \int_{0}^t \rho_{a(s)}^{-2} ds \RB 
\label{Grnwll}
\end{equation}
for $t\in [\delta_\e, \tau_1]$. But the definition \eqref{Tstar.def} of $\tau_\star$
is exactly chosen (once the constants are adjusted 
correctly) so that the right-hand side of \eqref{Grnwll}
is less than $\frac 14 \e^{1/3}$ when $0 \le t \le \tau_\star$,
and then \eqref{dtsts} follows from our estimate \eqref{lipschitzetaestimate}
of $\eta - \LA \eta\RA_{\delta_\e}$.

{2}.
Next, from \eqref{xis.def}, the definition of the $X_{\ln}^*$ norm, the characterization 
of the $\dot{W}^{-1,1}$ norm as the ``length of a minimal connection" (see \cite{BCL}),
and \eqref{eta.rocks1},
we see that for $0\le t \le \tau_1$, 
\begin{align*}
\lambda(t) 
&\le  \| \pi \sum(\delta_{\xi_i(t)}- \delta_{a_i(t)}) \|_{X_{\ln}^*} + 
\e\calP  \\
&\le   \| \pi \sum(\delta_{\xi_i(t)}- \delta_{a_i(t)}) \|_{\dot{W}^{-1,1}(\R^2)}  + \e\calP\\
&= \pi \sum  |\xi_i(t) - a_i(t)| + \e\calP \\
& \le\   \eta(t) +\e\calP.
\end{align*}
Also, it is a consequence of 
\eqref{XcJ} that $\lambda$ is continuous. 
As a result, 
\eqref{dtsts} easily implies that $\tau_1\ge \tau_\star$, since if not, we would find that
that $\lambda(t) <   \e^{1/3}$ for $0\le t \le \sigma$, for some $\sigma>\tau_1$, contradicting
the definition of $\tau_1$.

Since
the inequality $\tau_1\ge \tau_\star$ is
a restatement of the  conclusion
\eqref{t1.c1} of the theorem, the proof is complete. 

\end{proof}

\section{{hydrodynamic limit}}\label{S:hydro}

{In this section we complete the proof of Theorem~\ref{T1},
and we
establish in Theorem \ref{T.random}  {\em almost sure} convergence for
certain sequences of initial data with random vortex locations.}

We will need the following estimate.

\begin{lemma} \label{interp.X.W} Assume that $\mu_1$ and $\mu_2$ are two probability
measures on $\R^2$ and that 
\[
\int|x|^2 d\mu_i \le M, \qquad i=1,2.
\]
Then 
\begin{equation} \label{weaknorminterp}
\| \mu_2 - \mu_1\|_{X^*_{\ln}} \le  C \sqrt{M} \| \mu_2 - \mu_1\|_{W^{-2,1}(\R^2)}^{1/4}.
\end{equation}
\end{lemma}

\begin{proof}
First we recall that for any signed measure $\nu$ on $\R^2$ with finite total mass,
\begin{equation}
\| \nu \|_{W^{-1,1}(\R^2)} \le C \sqrt{ \| \nu \|_{W^{-2,1}(\R^2)}  \ |\nu|(\R^2)}.
\label{interp00}\end{equation}
Indeed, taking any smooth, compactly supported $\phi$, and letting $\eta_\e$ denote
a standard mollifier supported in a ball of radius $\e$,
\[
\int \phi d\nu = \int \eta_\e*\phi d\nu +  \int (\phi - \eta_\e*\phi) d\nu \le
\| \eta_\e*\phi - \phi\|_{L^\infty} |\nu|(\R^2) + \| \eta_\e*\phi\|_{W^{2,\infty}} \|\nu\|_{W^{-2,1}}.
\]
The claim \eqref{interp00} follows by noting that
\[
\| \eta_\e*\phi - \phi\|_{L^\infty} \le \e \| \phi\|_{W^{1,\infty}},
\qquad
 \| \eta_\e*\phi\|_{W^{2,\infty}} \le \frac C \e  \| \phi\|_{W^{1,\infty}}\mbox{ for } \e \le 1
\]
and then selecting $\e := ( \|\nu\|_{W^{-2,1}} / |\nu|(\R^2) )^{1/2} $, which is clearly bounded by $1$.

Now let $\phi$ be any compactly supported function such that
$\| \nabla \phi\|_{L^\infty}  \le \|\phi \|_{X_{\ln}^*} \le 1$.
Let $(\chi_R)_{R>0}$ be a family of functions satisfying \eqref{chiRdef1} and \eqref{chiRdef2}.
Note that every constant is $(\mu_2-\mu_1)$-integrable, and integrates to zero, so
\[
\int \phi d(\mu_2-\mu_1) = \int \tilde \phi d(\mu_2-\mu_1), \qquad  \mbox{ for }\tilde \phi(x) := \phi(x) - \phi(0).
\]
Also, $\| \chi_R \tilde \phi \|_{W^{1,\infty}} \le C{(R+1)}$, and 
$|(1-\chi_R)\tilde \phi|(x) \ \le R^{-1}|x|^2$. Thus
\begin{align*}
\int \phi d(\mu_2-\mu_1)
&=
\int \chi_R \tilde \phi d(\mu_2-\mu_1)
+
\int  (1-\chi_R) \tilde \phi d(\mu_2-\mu_1)\\
&\le
C{(R+1)}\| \mu_2-\mu_1\|_{W^{-1,1}(\R^2)}
+
\frac 1R \int|x|^2 d(\mu_1 + \mu_2) 
\\
&\le 
C{(R+1)}\| \mu_2 -\mu_1\|_{W^{-1,1}(\R^2)}+ \frac {2M}R.
\end{align*}
The proof is concluded by taking
 the supremum over $\phi$ as above, optimizing over $R$,
 and using \eqref{interp00} and the fact that $\mu_1,\mu_2$ are
 probability measures.
\end{proof}

\subsection{deterministic initial data}

%
%
%

We can now complete the

\begin{proof}[proof of Theorem~\ref{T1}]
1. We first claim that the initial data
\begin{equation}
u_\e^0 = \prod_{j=1}^{n_\e}  \phi_\e(x - a_j^\e), \qquad \mbox{$\phi_\e$ defined in \eqref{modelvortex}}
\label{data.bis}\end{equation}
satisfies the hypotheses \eqref{t1.h1}, \eqref{t1.h2} of Theorem \ref{thmdynamics}.
This is rather standard without the precise error estimates needed here, 
and these sorts of error estimates are checked in detail
in Lemma~14 in \cite{JSp2}. Indeed,  \eqref{t1.h1} follows
directly from arguments of \cite{JSp2}. To prove \eqref{t1.h2},
we appeal to
step 1 of the proof of Proposition~\ref{Propgammastability} (noting that $u_\e^0 = u_*$ outside $\cup B_{\sqrt \e}(a_j)$), then use \eqref{FiniteWapprox}  and 
\eqref{gamma.def}  to find that
\begin{align*}
\Sigma(u_\e^0; a_j^\e) & \leq
 \LV \lim_{R\to\infty} \int_{B_R\backslash (\cup B_{\sqrt{\e}}(a_j) )} {1\over 2} \LV \nabla u_* \RV^2  dx
 - \LB \pi D^2 \ln R + n \pi \ln {1\over \sqrt{\e}}   +W (a) \RB \RV   \\
& \quad  +  \LV \sum_{j=1}^{n_\e} \int_{B_{\sqrt{\e}}(a_j)} e_\e(u_\e^0)) \ dx -  \LC  \pi \ln{\sqrt{\e} \over {\e}} + \gamma \RC \RV  
 \leq C \e ( 1 + {n_\e^2 \over \rho^2_{a_j^\e}} ) \leq  \e^{1\over 2},
\end{align*}
which is \eqref{t1.h2}.

2. 
It  now follows by Theorem~\ref{thmdynamics} the solution $u_\e(t,x)$ of \eqref{nls}
with initial data \eqref{data.bis}
satisfies 
\begin{equation}
\Big\|  J(u(t)) - \sum_{j=1}^n \pi  \delta_{a_j(t)} \Big\|_{X^*_{\ln}} 
\leq \e^{1\over 3}
\mbox{\ \ \ for  $0<t<\tau_\star = \sup\{ T  > 0: \frac C n \int_0^T \rho^{-2}_{a(t)} dt \leq \logep \}$,}
\label{recall.t1c1}\end{equation}
where $a(\cdot)$ solves \eqref{PV2} with initial data $(a^\e_1,\ldots, a^\e_n)$.  We claim that for every $T>0$, there exists some $\e_T>0$ such that
$\tau_* \ge \frac T {2\pi n_\e}$ for $0< \e < \e_T$.
The proof of this is where we need the restriction $n_\e^2 =  o( \ln \logep)$ on the 
number of vortices. Indeed, it follows from 
this assumption and  \eqref{rhoRbds} 
that\footnote{Note that \eqref{rhoRbds} is applicable here, since 
conditions \eqref{Mbounds0} are preserved by the point vortex ODEs.} \[
\rho_{a^{n_\e}(t)} \ge \frac 14 \exp( - M_0 n_\e^2) = \frac 14 \exp[- o(1)\cdot\ln \logep]
= \frac 14 \logep^{-o(1)}
\]
as $\e\to 0$ for all $t$, and the claim easily follows.

3. As noted in Remark \ref{rem:trescale}, if
we define $b^{n_\e}(t) = a(\frac t{2\pi n_\e})$, then $b^{n_\e}(\cdot)$ solves \eqref{ODE},
with initial data satisfying \eqref{Mbounds0} for every $\e$.  
Hence we find from Theorem \ref{Thm.Schochet} that
the sequence
\[
\omega^{n_\e} (t) \  = \  {1\over n_\e} \sum_{j=1}^{n_\e} \delta_{b^{n_\e}_j(t)} \ = \ 
 {1\over n_\e} \sum_{j=1}^{n_\e} \delta_{a^{n}_j(\frac t{2\pi n_\e})}
\]
is precompact in $C([0,T], W^{-2,1})$ for every $T>0$,
and hence by \eqref{weaknorminterp} in $C([0,T], X_{{\ln}^*})$. In addition, any
limit  is a 
weak solution to the 2D equation with initial data $\omega_0 = \operatorname{wk\,lim}_{\e\to 0}\omega^{n_e}(0)$.

On the other hand, recalling the definition of the rescaled vorticity,
\[
\widetilde \omega_{\e}(t, x) := \frac 1{2\pi n_\e}\omega(u_\e)(\frac t{2\pi n_\e}, x) = {1\over \pi n_\e} 
J(u_\e)(\frac t{2\pi n_\e}, x),
\]
we deduce from \eqref{recall.t1c1}
that for any $t\in (0,T]$ and $\e<\e_T$,
\[
\| \widetilde \omega(t) - \omega^{n_\e}(t)\|_{X_{\ln}^*}
=
\frac 1{\pi n_\e} \left\|  J(u_\e(\frac t{2\pi n_\e})) - \pi \sum_{j=1}^{n_\e} \delta_{a^{n}_j(\frac t{2\pi n_\e)})} \right\|_{X_{\ln}^*} \ \le \e^{\frac 13}.
\]
Hence the compactness of $(\widetilde \omega_\e)_{\e\in (0,1]}$ in $C(0,T;X_{\ln}^*)$,
and fact that any limit must be a weak solution  of the Euler equations, follow by the triangle inequality from the corresponding properties of $(\omega^{n_\e})_{\e\in (0,1]}$.
\end{proof}


\subsection{random initial data}

We next consider solutions of \eqref{nls} with initial vortex locations
near points $a = (a_1,\ldots, a_n) \in \R^{2n}$ chosen at random from among those satisfying
\begin{equation}
|a|^2 = \sum_{j=1}^n |a_j|^2 \le n R^2
\label{2momR}\end{equation}
and
\begin{equation}
-\sum_{i\ne j} \ln |a_i-a_j| \le n(n-1) M
\label{Hminus1M}\end{equation}
for some $M$ and $R$. 
We will use the notation
\begin{equation}
A(n,M,R):=\{ a\in \R^{2n} : \eqref{2momR} \mbox{ and }\eqref{Hminus1M} \mbox{ hold} \}.
\label{AnMR}\end{equation}
We will write $P_n$ (suppressing the dependence on $M$ and $R$) to 
denote normalized Lebesgue measure on $A(n,M,R)$, so that
$P_n := \frac 1{|A(n,M,R)|} \calL^n\rest{A(n,M,R)}$.

\medskip

We will prove that for suitable $M$ and $R$, if we choose $a^n$ at random, according to the
probability measure $P_n$,  
then solutions of the Gross-Pitaevskii equation \eqref{nls} with
well-prepared initial data having
vortices at $(a^n_1,\ldots, a^n_n)$
{\em almost
surely} have vorticity governed by the Euler equations in the
limit $n\to \infty$, for
times that are arbitrarily large (after the natural time rescaling),
as long as $n \le\logep^{\frac 1{5+\delta}}$ for some $\delta>0$.
{This is a much larger number of vortices than is allowed in Theorem \ref{T1}.}

A {\em serious limitation} of this result is that if $a_n$ is chosen 
as described above, then the sequence of measures $\frac 1n \sum_{1=i}^n \delta_{a^n_i}$
weakly converges a.s. to a Gaussian; 
see Lemma \ref{L.Jeremy} below.
So in fact the theorem only allows us to obtain the
hydrodynamic limit for this particular initial data, for which 
the solution of the Euler equations is trivial in the sense that it is
independent of $t$.

Nonetheless, we do not know a way to prove Theorem \ref{T.random} without the full strength of
Theorem \ref{thmdynamics}. 

The proof of the theorem
suggests that for
an arbitrary
probability measure in $H^{-1}$, and with finite second moment,
there should be some random choice of sequence of initial
data that would yield a well-behaved hydrodynamic limit
almost surely with the same number $n_\e \le \logep^{\frac 1{5+\delta}}$
of vortices, but we do not know how to prove this.

\begin{theorem}
For every $R>0$, there exists some $M(R)>0$ such that the following holds
for every $M\ge M(R)$:

For $n \in \mathbb N$,
let $a^n \in A(n, M,R)$ be
chosen at random (according to the probability measure $P_n$).
Let $\e_n$ be such that $n \le |\log \e_n|^{\frac 1{5+\delta}}$ for some $\delta>0$, and 
let $u_n$ solve \eqref{nls} for $\e = \e_n$, with initial data
\begin{equation}
u_n^0(x) \ = \ \prod_{j=1}^{n} \phi_{\e_n}( x - a^n_j),
\label{T4data}\end{equation}
for  $\phi_{\e_n}:\R^2\to \C$ described in \eqref{modelvortex}.
Define the {\em current},  {\em vorticity}, and {\em rescaled vorticity}
as in Theorem \ref{T1}.

Then for any $T>0$,
 the rescaled vorticities {are {\em almost surely} precompact in  $C((0,T), X_{\ln}^*)$, and any limit is a  weak solution  of the Euler equations \eqref{Euler}, \eqref{BiotSavart}.}
\label{T.random}
\end{theorem}

\begin{remark}
In fact, in view of Lemma \ref{L.Jeremy} below, under the hypotheses of 
Theorem \ref{T.random}, the rescaled vorticities 
a.s. satisfy
\[
{\sup_{0<t<T} \left \|\widetilde \omega_{\e_n}(t)  - \frac 1{\pi R^2} \exp( -\frac{|\cdot|^2}{R^2}) \right\|_{X_{\ln}^*} \to 0\qquad\mbox{ as }
n\to \infty.}
\]
\end{remark}




\begin{proof}
1. Let us write $\Phi^n: \R^{2n}\times \R \to \R^{2n}$
to denote the solution operator associated to the point vortex ODEs 
\eqref{PV2}, so that
\[
t\mapsto  \Phi^n(a^n, t)
\qquad
\mbox{ is the solution of  \eqref{PV2} with initial data $a^n\in \R^{2n}$.}
\]
Given $a^n\in \R^{2n}$,  we will write
\[
\tau^n_*(a^n) := 
\sup\left\{ T  > 0: C n \int_0^T \rho^{-2}_{\Phi^n(a^n, t)} dt \leq |\ln \e_n|  \right\} .
\]
Exactly as in the proof of Theorem \ref{T1}, if $a^n \in \R^{2n}$ is any
sequence of initial vortex locations such that
\begin{equation}
 \liminf_{n\to \infty} [ (2 \pi n) \tau^n_*(a^n) ]\ge  T ,
\label{lstaun}   
\end{equation}
and if $u^n$ solves \eqref{nls} with initial data \eqref{T4data},
then the associated rescaled vorticities 
{are  precompact in  $C((0,T), X_{\ln}^*)$, and any limit is a  weak solution  of the Euler equations.}

So we only need to prove that for any $T>0$, condition \eqref{lstaun}
is satisfied a.s. for sequences of initial data, if $a^n$ is chosen
according to the probability measure $P_n$.

For this, by the Borel-Cantelli Lemma, it suffices to show that
\begin{equation}
\sum_{n=1}^\infty P_n (\calB_n) \ < \infty ,
\label{BCLemma}\end{equation}
where $\calB^n$ is the set of bad initial vortex locations, defined by
\[
\calB_n := 
\left\{ a^n \in A(n,M,R) : \
 C n \int_0^{T/2 \pi n} \rho^{-2}_{\Phi^n(a^n,t)} dt > |\ln \e_n|  \right\}.
\]
By Chebyshev's inequality, it  is therefore enough  to show that
there exists some $C, \delta>0$ such that for all sufficiently large $n$, 
\[
P_n(\calB_n) \le \frac 1{|\ln \e_n| }\int_{A(n,M,R)}\left( Cn 
 \int_0^{T/2 \pi n} \rho^{-2}_{\Phi^n(a^n,t)} dt \right) P_n(da^n)
{\le C T n^{-(1+\delta)}.}
 \]
Conservation laws for \eqref{PV2} imply that $\Phi^n$ is a diffeomorphism
of $A(n,M,R)$ onto itself, and then Liouville's Theorem implies that
$\Phi^n$ preserves Lebesgue measure on $A(n,M,R)$, and hence
preserves $P^n$. 
Thus for every $t$,
\[
\int_{A(n,M,R)} \rho_{a^n}^{-2} \ P_n(da^n)
=
\int_{A(n,M,R)} \rho_{\Phi^n(a^n,t)}^{-2} \ P_n(da^n).
\]
From this and Fubini's Theorem, we conclude that the theorem will
follow if we can prove that
\[
\int_{A(n,M,R)} \rho_{a^n}^{-2} \ P_n(da^n) 
\ \le  \  C \ n^{-(1+\delta)} |\ln \e_n| .
\]
{Given the assumption that $n^{5+\delta} \le |\ln \e_n|$,}
this will follow once we establish that
\begin{equation}
\int_{A(n,M,R)} \rho_{a^n}^{-2} \ P_n(da^n) \le C n ^4
\qquad\mbox{ for all }n.
\label{T4reduction}\end{equation}
And we prove  in Lemma \ref{L.pintegral}
below that \eqref{T4reduction} holds if $M= M(R)$ is
sufficiently large. This will complete the proof of the theorem.

We remark that the basic point is that for $R$ fixed, one can choose $M(R)$
such that if $M\ge M(R)$, then $A(n,M,R)$ occupies at least half of the ball 
$B^{2n}_{R\sqrt n}$,
and this makes it rather easy to 
estimate integrals with 
respect to $P_n$. 
 
\end{proof}

The proof of Lemma \ref{L.pintegral}
will use the following calculation several times.

\begin{lemma}Let $F_n:\R^{2n}\to \R$ be a  function of the form
\[
F_n(a^n) := \frac 1{n(n-1)} \sum_{i\ne j} f(|a^n_i-a^n_j|)
\]
where $f$ is locally integrable on $[0,\infty)$ and with at most
polynomial growth.
Then  for $n\ge 2$ and $R>0$,
 \begin{equation}
\barint_{B^{2n}_{R\sqrt n}} F_n(a^n) \ da^n 
\le \frac 1{R^2} \int_0^\infty  |f(\sqrt{2t})| \exp(-\frac t{2R^2}) \, dt
\label{integral.c1}\end{equation}
and 
\begin{equation}
\lim_{n\to\infty}\barint_{B^{2n}_{R\sqrt n}} F_n(a^n) \ da^n 
\to \frac 1{R^2} \int_0^\infty f(\sqrt{2t}) \exp(-\frac t{R^2}) \, dt \qquad
\mbox{ as }n\to \infty.
\label{integral.c2}\end{equation}
\label{integral0}\end{lemma}

\begin{proof}We will write simply  $a = (a_1,\ldots a_n)$, without  superscripts, for 
convenience.

First, by symmetry it is clear that $\int_{B^{2n}_{R\sqrt n}} F_n(a) da = \int_{B^{2n}_{R\sqrt n}} f(|a_1-a_2|) da.$

Next, consider the change of variables
\[
(a_1,\ldots, , a_n) \mapsto (y_1,\ldots, y_n) = ( \frac 1{\sqrt{2}}(a_1 -a_2), \frac 1{\sqrt{2}}(a_1 +a_2),
a_3,\ldots, a_n). 
\]
The Jacobian is $1$, and $|a| = |y|$, so (writing $\rho$ for $R\sqrt n$ for convenience)
\[
\int_{B_\rho^{2n}}
f (|a_1-a_2|) da
=
\int_{B_\rho^{2n}}
f (\sqrt 2 |y_1|) dy .
\]
If we integrate first in $y_2,\ldots, y_n$, then we find that
\begin{align*}
\int_{B_\rho^{2n}}
f  (\sqrt 2 |y_1|)dy 
&=
\int_{B_\rho^{2}}
f  (\sqrt 2 |y_1|)  \  \Big| B^{2n-2}_{(\rho^2 - |y_1|^2)^{1/2}}\Big| \ dy_1 
\\
&=
\omega_{2n-2} \int_{B_\rho^{2}}
f  (\sqrt 2 |y_1|) (\rho^2 - |y_1|^2)^{n-1} dy_1 
\\
&=
2\pi \omega_{2n-2} \int_0^\rho
f (\sqrt 2 r) (\rho^2 - r^2)^{n-1} r \ dr
\\
&=
\pi \rho^{2(n-1)} \omega_{2n-2} \int_0^{\rho^2}
f (\sqrt{2 t}) (1 - \frac t {\rho^2})^{n-1}  \ dt.
\end{align*}
Since
$\omega_{2n} = \frac \pi n \omega_{2(n-1)}$ (this is a textbook identity) 
we infer that
\begin{equation}
\barint_{B^{2n}_{R\sqrt n}} F_n(a) \ da 
= 
\frac 1{R^2} \int_{0}^{nR^2} f(\sqrt{2t}) (1 - (\frac t{ nR^2}))^{n-1} dt.
\label{int.eee0}\end{equation}
Then \eqref{integral.c1} follows from the fact that
\begin{equation}
 (1 -  (\frac t{ nR^2}))^{n-1}  \le  \left[\exp( - \frac t{nR^2})\right]^{n-1} \le \exp[- \frac t{2R^2}]
\label{integral.int1}\end{equation}
for $0\le t/nR^2\le 1$ and $n\ge 2$. In view of \eqref{integral.int1},
we may deduce 
\eqref{integral.c2} from  \eqref{int.eee0} and the
Dominated Convergence Theorem.
\end{proof}

The following lemma completes the proof of Theorem \ref{T.random}.

\begin{lemma}For every $R$, there exists $C, M_0>0$ such that if $M\ge M_0$
then
\begin{equation}
\barint_{A(n,M,R)} \rho_a^{-2}\, da \ \le C n^4.
\label{pintegral.c}\end{equation}
\label{L.pintegral}\end{lemma}

\begin{proof}
{\bf 1.} 
We first note that if $a$ satisfies \eqref{2momR}, then (since $|a_i - a_j|^2 \le 2(|a_i|^2+|a_j|^2)$
\[
\sum_{i,j=1}^n |a_i - a_j|^2 \le 
2
\sum_{i,j=1}^n (|a_i |^2+ | a_j|^2)  \le 4 n^2R^2.
\]
Since $r^2 -\ln r \ge 0$ for all $r>0$, it follows that 
\[
- \sum_{i\ne j} \ln |a_i-a_j| \ge 
 \ -4n^2R^2 .
\]
Thus the function
\[
F(a) = 
4(\frac n{n-1}) R^2 -\frac 1{n(n-1)} \sum_{i\ne j}\ln |a_i-a_j| 
\]
is nonnegative in $B^{2n}_{R\sqrt n}$.
By Lemma \ref{integral0}, 
\[
\lim_{n\to\infty}\barint_{B^{2n}_{R\sqrt n}} F(a)\,da\  = 4R^2-\frac 1{R^2}\int_0^\infty \ln(\sqrt{2t}) \exp(-\frac t {R^2})\,dt\  =: 4R^2+ \lambda(R).
\]
In particular the limit exists. 
Moreover, for any $M$, if \eqref{Hminus1M} fails then 
{$F(a) > M + 4(\frac n{n-1})R^2$,}
so
by Chebyshev's inequality (applicable since $F$ is nonnegative)
\begin{align*}
\frac{|B^{2n}_{R\sqrt n}\setminus
A(n,M,R)|  } {|B^{2n}_{R\sqrt n}|} 
&\le 
\frac{|\{ a\in B^{2n}_{R\sqrt n} \ : F(a) > M+4(\frac n{n-1})R^2  \}| } {|B^{2n}_{R\sqrt n}| }\\
&\le \frac 1{M+4(\frac n{n-1})R^2} 
\barint_{B^{2n}_{R\sqrt n}} F(a) \ da
\\
&\to \frac { \lambda(R)+4R^2}{M+4R^2}\qquad \mbox{ as \ }n\to \infty.
\end{align*}
So if $M \ge M_0 = 3\lambda(R) + 8R^2$, then for all sufficiently large $n$,
\begin{equation}
\frac{|A(n,M,R)|  } {|B^{2n}_{R\sqrt n}|}  \ge \frac 12.
\label{Ahalf}\end{equation}

{\bf 2.} 
From above we know that if $a\in A(n,M,R)$ then
\[
\sum_{i,j=1}^n\left( |a_i-a_j|^2 - \ln|a_i-a_j|\right) \ \le \  4 n^2R^2 + n(n-1)M \   < \  n^2 M_1
\]
for $M_1 = M+4R^2$.
Since $r^2-\ln r \ge 0$ for all $r>0$, it follows that 
\[
 |a_i-a_j|^2 - \ln|a_i-a_j| \le 4 n^2R^2 + n(n-1)M  \le  n^2M_1 \quad\mbox{ for all }i\ne j
\]
and hence that 
\[
|a_i-a_j| \le \exp[ - n^2 M_1] \qquad\mbox{ for all } {i\ne j.}
\]
It follows from this and \eqref{Ahalf} that if $M \ge M_0$, then
\[
{\barint_{A(n,M,R)} \rho_a^{-2} da \le
\frac 1{|B^{2n}_{R\sqrt n}|}\int_{B^{2n}_{R\sqrt n}}
\sum_{i\ne j} \chara_{|a_i-a_j|\ge\exp[ - n^2 M_1] } \frac 1{|a_i-a_j|^2} \ da.}
\]
Using Lemma \ref{integral0}  again, we find that
\begin{align*}
\frac 1{n(n-1)}\barint_{A(n,M,R)} \rho_a^{-2} da 
&\le
{\frac 1 {R^2}}\int_0^{nR^2} 
\chara_{ \sqrt {2t} \ge \exp[-n^2M_1]} (\sqrt {2t})^{-2} \exp(-\frac t {2R^2}) dt
\\
&\le
\frac 1{R^2}
 \int_{\exp[-2n^2M_1]}^\infty
\frac 1{2t}\exp(-\frac t {2R^2})dt.
\end{align*}
Then by breaking the integral into 2 pieces, 
we estimate
\begin{align*}
\frac 1{n(n-1)}\barint_{A(n,M,R)} \rho_a^{-2} da 
&\le \frac 1{2R^2} \int_{\exp[-2n^2M_1]}^{R^2}
\frac 1{t} dt +
\frac 2{R^2} \int_{R^2}^\infty
\frac 1{2t}\exp(-\frac t {2R^2})dt\\
& =  \frac 1{R^2}\left[ \ln R + n^2M_1 +C
\right].
\end{align*}
We conclude that for large enough $n$,
\[
\barint_{A(n,M,R)} \rho_a^{-2} da 
\le C(R) + C(R,M)n^4 \le Cn^4.
\]
\end{proof}

Finally we prove that the initial data chosen above in Theorem \ref{T.random} 
(with $R=1$ for simplicity) converges almost surely to a Gaussian. 
This is not needed in any of our arguments, but it is certainly relevant to 
Theorem \ref{T.random}.

\begin{lemma}
Let $P_n$ denote normalized Lebesgue measure on the ball $B^{2n}_{\sqrt n}$.
Given $a^n = (a^n_1,\ldots, a^n_n) \in \R^{2n} \cong (\R^2)^n$,
let
\[
\mu_{a^n} := \frac 1n \sum_{i=1}^n \delta_{a^n_i} \in \mathcal{M}(\R^{2n}).
\]
Then 
$
\mu_{a^n} \rightharpoonup G$ 
almost surely, where $G(x) =  \frac 1{\pi} e^{- |x|^2}$. 
\label{L.Jeremy}\end{lemma}

In the statement of the lemma, and throughout the proof, ``$\rightharpoonup$'' denotes weak convergence of measures.

The proof was explained to us by Jeremy Quastel

\begin{proof}
Let $\calD$ be a probability space, equipped with probability measure $\calP$,
and for $n\ge 1$ and $i=1,\ldots, n$, let $Z^n_i:\calD\to \R^2$ be i.i.d  Gaussian
random variables on the plane with probability density function $G$, so that 
$\calP(\{ Z^n_i \in A\}) = \int_A G(x)\ dx$ for Borel $A\subset \R^2$. 
We will write $Z^n = (Z^n_1,\ldots, Z^n_n)\in \R^{2n}$.

For $n \ge 1$, let $S_n:\calD\to \R$ be a random variable, independent of all $Z^i_n$,
with probability distribution function 
\[
s_n(r) := \chara_{[0, 1]}2n r^{2n-1}
\]
so that $\calP( a \le S_n \le b) = \int_a^b s_n(r) dr$.
Finally, define $Y^n :\calD\to \R^{2n}$ by
\begin{align*}
Y^n &:= (Y^n_1,\ldots, Y^n_n) 
:= \frac{ (Z^n_1, \ldots, Z^n_n)}{|Z^n|}\sqrt{n} S_n.
\end{align*}

One can check that $Y^n_*\calP = P_n(A)$ , where $Y^n_*(A) = \calP(\{ Y^n\in A\})$
for $A\subset \R^{2n}$.
Indeed, it is easy to see that $Y^n_*\calP$ is invariant with respect to rotations of $\R^{2n}$,
and the definition of $S_n$ implies that $Y^n_*\calP(B_r) = (r/{\sqrt{n}})^{2n}$, and these two 
properties characterize $P_n$. Thus we have to prove that $\calP$ almost surely,
 $\mu_{Y^n}\rightharpoonup G$ as $n\to \infty$. It is standard, and rather clear, that $\mu_{Z^n}\rightharpoonup G$ a.s., so it suffices to show that $\mu_{Y^n} - \mu_{Z^n}\rightharpoonup 0$ a.s..
 
%
Next, fix $f\in C_0(\R^2)$, and note that for every $\delta>0$ there exists  $M_\delta$ such that 
\[
|f(x)-f(y)|\le \delta + M_\delta|y-x|^2
\]
for all $x,y\in \R^2$.
Thus
\begin{align*}
\left|\int f d(\mu_{ Y^n}  -
\mu_{Z^n}) \right| 
&\le \frac 1n \sum_i |f(Y^n_i) - f(Z^n_i)| \\
&\le
\delta + \frac{M_\delta}n\sum_i |Z^n_i|^2 ( 1 - \frac {\sqrt n S_n}{|Z^n|})^2 \\
&=
\delta + M_\delta ( \frac{|Z^n|}{\sqrt n} - S_n)^2.
\end{align*}
However, it is immediate from the definition that
$S_n\to 1$ a.s, and it
is well known that $(n)^{-1/2}|Z^n|\to 1$ a.s..
(For example, this can be read off from the fact that $|Z^n|^2$
has probability density function of the form $\calP(|Z^n|^2 < R)
=\int_0^R \kappa_n s^{n-1}e^{-s} ds $ for some constant $\kappa_n$, which follows from a short computation.) 
It follows that
\[
\limsup_{n\to\infty}\left|\int f d(\mu_{ Y^n}  -
\mu_{Z^n}) \right| 
\le \delta \qquad a.s..
\]
Since $\delta$ and $f$ were arbitrary,  it follows that $\mu_{Y^n} - \mu_{Z^n}\rightharpoonup 0$
a.s. as needed.
\end{proof}

 \end{document}